\documentclass[a4paper]{amsart}

\usepackage{enumerate}
\usepackage{amssymb}
\usepackage{amsmath}
\allowdisplaybreaks[4]
\usepackage{mathrsfs}
\usepackage{amsthm}
\usepackage{mathtools}
\usepackage{tikz}
\usepackage[colorlinks,linkcolor=blue]{hyperref}

\theoremstyle{definition}
\newtheorem{theorem}{Theorem}[section]
\newtheorem{definition}[theorem]{Definition}

\newtheorem{lemma}[theorem]{Lemma}

\newtheorem{proposition}[theorem]{Proposition}
\newtheorem{corollary}[theorem]{Corollary}
\theoremstyle{remark}
\newtheorem{remark}[theorem]{Remark}

\numberwithin{equation}{section}
\newcommand{\eqlabel}[1]{\refstepcounter{equation}\label{#1}}

\begin{document}
	\title{A Llarull type theorem on complete noncompact manifolds}
	\author{Guangrui Zhu}
    \address{School of Mathematical Sciences, East China Normal University, Shanghai, 200241, P. R. China}
\email{52275500052@stu.ecnu.edu.cn}
	\date{}
\thanks{The author is supported in part by Fundamental and Interdisciplinary Disciplines Breakthrough Plan of the Ministry of Education of China (JYB2025XDXM112).}
	\begin{abstract}
Let $n\geq3$, $2\leq k\leq n-1$, and $m=n-k-1$.
We prove that a complete, connected, noncompact, oriented manifold $(M^n,g)$ with scalar curvature $R_M\geq k(k-1)$ is isometric to
$\mathbb{S}^k\times\mathbb{T}^m_\Lambda\times\mathbb{R}$ if it admits a smooth proper map of non-zero degree to $\mathbb{S}^k\times\mathbb{T}^m\times\mathbb{R}$ whose spherical component is 1-Lipschitz.
We also establish a weighted scalar curvature obstruction at infinity for sphere-torus targets with a higher dimensional Euclidean factor under suitable control of the Euclidean component. These results hold in all dimensions without a spin assumption.
\end{abstract}

	\maketitle
	\section{Introduction}
Scalar curvature rigidity asks whether a sharp lower bound for scalar
curvature, together with a suitable metric comparison map, determines the underlying Riemannian metric.  A fundamental result is Llarull's theorem \cite{Llarull}: if a closed spin manifold carries a 1-Lipschitz map of non-zero degree to the unit round sphere and its scalar curvature is bounded below by the scalar curvature of that sphere, then the map is an isometry. Thus, any nontrivial enlargement of the round metric in the sense that $g\gneqq g_{\mathrm{rd}}$ forces the scalar curvature to fall below the spherical threshold somewhere. The rigidity conclusion remains valid under a suitable weighted scalar curvature condition \cite{C-W,chow2025,Deng,Zhou-Zhu}. Llarull's proof is based on the Dirac operator. Recently, Wang-Wang-Xie-Zhu \cite{WWXZ2026} removed the spin assumption in all dimensions, deducing Llarull's theorem from a scalar-mean comparison theorem for manifolds with boundary. We record the version used in this paper.
 \begin{theorem}(\cite[Theorem~1.1]{WWXZ2026})\label{thm:WWXZ-Llarull}
  \textit{ Let $(N^n,g)$ be a closed, connected, oriented Riemannian manifold of dimension $n\ge2$. Suppose that $f:(N,g)\to(\mathbb{S}^{n},g_{\mathbb{S}^{n}})$ is a smooth $1$-Lipschitz map with non-zero degree. If $R_N\geq n(n-1)$, then $f$ is an isometry.}
    \end{theorem}

Zhang \cite{Zhang2020} studied smooth area non-increasing maps $f:M^n\to\mathbb S^n$ of non-zero degree from complete noncompact spin manifolds, assuming that $f$ is locally constant near infinity.
In even dimensions, he showed that the bound $R_M\ge n(n-1)$ on $\operatorname{supp}(df)$ forces the scalar curvature to be negative somewhere on $M$.
Li, Su, Wang, and Zhang \cite{LSWZ2024} proved the same conclusion in odd dimensions, where Zhang \cite{Zhang2020} had required a strict curvature bound.

The unit sphere $\mathbb{S}^k$ carries the positive scalar curvature
$k(k-1)$, whereas a flat torus contributes no positive scalar curvature. More intrinsically, the torus is enlargeable and admits no metric of positive scalar curvature. Thus $\mathbb S^k\times \mathbb T^m$ is the simplest closed model in which all of the positive scalar curvature comes from the spherical factor, while the torus records scalar flat directions constrained by topology. In the closed setting, rigidity results for such products have been obtained by minimal hypersurface and Dirac operator methods in \cite{chow2025,HSS,Zhou-Zhu}. Riedler-Tony \cite{RT2026} obtained product rigidity results in arbitrarily high dimensions through an approach based purely on Dirac operator methods, with the spherical case restricted to even dimensional spheres. The following weighted rigidity theorem of Chow is the closed model for the present work. Here and throughout the paper, for a Riemannian manifold $(N,g)$, a smooth function $\psi$ on $N$ and a real number $a$, we write
\begin{equation}\label{eq:Ra}
\mathcal R^a_g(\psi)=R_g-2\Delta_g\psi-a|\nabla_g\psi|^2,\qquad \mathcal R_g(\psi)=\mathcal R^1_g(\psi),
\end{equation}
where $\Delta_g=\operatorname{div}_g\nabla_g$. We call $\mathcal R^a_g(\psi)$ a weighted scalar curvature with weight $e^\psi$. Our results are stated for $\mathcal R_g(\psi)$. During the dimension descent, only the coefficient $a$ of $|\nabla_g\psi|^2$ changes (see Subsection \ref{subsec:weighted}).
\begin{theorem}(\cite[Theorem A]{chow2025})\label{Chow}
\textit{
For $3 \leq n \leq 7$ and $2 \leq k \leq n-1$, let $(N^n,g)$ be a closed, connected, spin Riemannian manifold. Suppose that
$$
A\coloneqq(A_\mathbb{S},A_\mathbb{T}):(N,g) \longrightarrow(\mathbb{S}^{k}\times\mathbb{T}^{n-k},\,g_{\mathbb{S}^{k}}+g_{\mathbb{T}^{n-k}})
$$
is a smooth map of non-zero degree such that
$A_\mathbb{S}:(N,g)\to(\mathbb{S}^{k},g_{\mathbb{S}^{k}})$ is
     $1$-Lipschitz. If there exists a smooth function $\psi$ on $N$ such that
$$
\mathcal R_g(\psi)\ge k(k-1),
$$
then $(N^n,g)$ is isometrically covered by
$(\mathbb{S}^{k}\times\mathbb{R}^{n-k},\,g_{\mathbb{S}^{k}}+g_{\mathbb{R}^{n-k}})$ and $\psi$ is a constant.
}
    \end{theorem}

The restriction $n\le 7$ in Theorem \ref{Chow}, as in the warped $\mu$-bubble arguments of \cite{Cho-Li,HSS,Zhou-Zhu}, comes from the regularity of area-minimizing hypersurfaces, which are smooth in ambient dimension at most seven. Brendle and Wang \cite{BW2026} used conformal blow up along the singular set to continue the dimension descent in their proof of the positive mass theorem in arbitrary dimension. Bi-Zhu \cite{BZ2026} recently developed a singular dimension descent for warped $\mu$-bubbles: a conformal blow up near the singular set of a $\mu$-bubble makes the regular part complete, allowing the descent to continue there. Together with Theorem \ref{thm:WWXZ-Llarull} and the scalar-mean comparison theorem of \cite{WWXZ2026}, this is the main tool that allows us to work in all dimensions.

The torus-line factor $\mathbb T^m\times\mathbb R$ shares the scalar curvature rigidity of the torus: completeness and nonnegative scalar curvature force the metric to be flat \cite{GL1983,Zhu1}. This leads naturally to $\mathbb S^k\times\mathbb T^m\times\mathbb R$ as a complete noncompact analogue of the closed model $\mathbb S^k\times\mathbb T^m$.

The complete noncompact case presents an additional compactness difficulty. Hao-Shi-Sun \cite[Theorem 1.5]{HSS} proved that a complete, connected, orientable four manifold with scalar curvature at least $6$ is isometric to $\mathbb S^3\times\mathbb R$ if it admits a proper non-zero degree $1$-Lipschitz map to $\mathbb S^3\times\mathbb R$ and has bounded geometry. They also observed that the bounded geometry assumption should be unnecessary. In their proof that
assumption enters when a uniform local area lower bound is used to show that a finite area limiting hypersurface is compact.  Removing the bounded geometry assumption from that limiting argument would therefore leave the compactness issue unresolved.

Zhu's work on complete manifolds with nonnegative scalar curvature
\cite{Zhu1} provides a different rigidity mechanism. Besides the use of $h_\varepsilon$-minimizing boundaries, an important feature of Zhu's deformation argument is that a strict scalar curvature improvement is first excluded by a topological obstruction. A compactly supported deformation in a negative Ricci direction then forces nonnegative Ricci curvature. Once nonnegative Ricci curvature is available, the geometry of the ends and the Cheeger-Gromoll splitting theorem \cite{C-G} convert the noncompact problem into a closed one. In this paper, we adapt this
deformation rigidity strategy to the sharp Llarull setting for
sphere-torus-line targets in all dimensions and without a spin assumption.

Our first main result is the following theorem. Throughout this paper, $\mathbb{S}^k$ denotes the unit round sphere, and $\mathbb{T}^m$ denotes a fixed flat torus. $\mathbb{T}^0$ is understood to be a point.
\begin{theorem}\label{main}
   \textit{
Let $n\ge 3$, $2\le k\le n-1$, and $m=n-k-1$. Let $(M^n,g)$ be a complete, connected, noncompact, oriented Riemannian
manifold with $R_M\ge k(k-1).$ Suppose that there exists a smooth proper map of non-zero degree
$$
F:(M,g)\longrightarrow
\left(\mathbb{S}^k\times\mathbb{T}^m\times\mathbb{R},\,g_{\mathbb {S}^k}+g_{\mathbb{T}^m}+dt^2\right)
$$
such that
$\operatorname{pr}_{\mathbb{S}^k}\circ F:
(M,g)\longrightarrow(\mathbb{S}^k,g_{\mathbb{S}^k})$
is $1$-Lipschitz. Then $(M,g)$ is isometric to
$$
\left(
\mathbb{S}^k\times\mathbb{T}_\Lambda^m\times\mathbb{R},\,
g_{\mathbb{S}^k}+g_{\mathrm{flat}}+dt^2
\right),
$$
where $\Lambda\subset\mathbb{R}^m$ is a full lattice and
$\mathbb{T}_\Lambda^m=\mathbb{R}^m/\Lambda$ carries the induced
Euclidean metric $g_{\mathrm{flat}}$.}
\end{theorem}

The flat torus in the conclusion of Theorem~\ref{main} need not be isometric to the target torus, even if the entire map $F$ is $1$-Lipschitz. To
see this, assume that $m\geq1$, fix $a>1$, and equip
$\mathbb S^k\times \mathbb T^m\times\mathbb R$ with the metric
$$
g=g_{\mathbb S^k}+a^2g_{\mathbb T^m}+dt^2.
$$
The identity map
$$
\operatorname{id}:
\bigl(\mathbb S^k\times \mathbb T^m\times\mathbb R,
      g_{\mathbb S^k}+a^2g_{\mathbb T^m}+dt^2\bigr)
\longrightarrow
\bigl(\mathbb S^k\times \mathbb T^m\times\mathbb R,
      g_{\mathbb S^k}+g_{\mathbb T^m}+dt^2\bigr)
$$
is proper, $1$-Lipschitz, of degree one, and the scalar curvature
of the source manifold is $k(k-1)$. However, the source torus
$(\mathbb T^m,a^2g_{\mathbb T^m})$ is not isometric to the target torus
$(\mathbb T^m,g_{\mathbb T^m})$.

When $n=4,\,k=3,\,m=0$, the conclusion is $\mathbb S^3\times\mathbb R$ and no
bounded geometry assumption is made.  Thus Theorem \ref{main} removes the bounded geometry condition from the
rigidity theorem of Hao-Shi-Sun.

The following corollary extends Theorem~\ref{Chow} to all dimensions without a spin assumption and identifies the global product and the spherical component of the map.
\begin{corollary}\label{cor:closed-rigidity}
\textit{
Let $k\ge2$ and $m\ge1$, and let $(N^{k+m},g_N)$ be a closed, connected, oriented Riemannian manifold. Suppose that
$$
A\coloneqq(A_\mathbb{S},A_\mathbb{T}):(N,g_N)\longrightarrow\mathbb{S}^{k}\times\mathbb{T}^{m}
$$
is a smooth map of non-zero degree such that
$A_\mathbb{S}:(N,g_N)\to(\mathbb{S}^{k},g_{\mathbb{S}^{k}})$ is $1$-Lipschitz. If there exists $\psi\in C^\infty(N)$ such that
$$
\mathcal R_{g_N}(\psi)=R_{g_N}-2\Delta_{g_N}\psi-|\nabla_{g_N}\psi|^2\ge k(k-1),
$$
then $\psi$ is constant and there exists a full lattice $\Lambda\subset\mathbb{R}^{m}$ such that $(N,g_N)$ is isometric to
$(\mathbb{S}^{k}\times\mathbb{T}^{m}_\Lambda,\,g_{\mathbb{S}^{k}}+g_{\mathrm{flat}})$. Moreover, under this isometry $A_\mathbb S(x,y)=J(x)$ for an isometry $J$ of $\mathbb S^k$.
}
\end{corollary}

The restriction to a single Euclidean factor in Theorem~\ref{main} is essential for rigidity. Indeed, Hao-Shi-Sun \cite[Example~1.4]{HSS} showed that rigidity fails in the presence of two or more Euclidean directions. More precisely, for every $r\geq 2$, let $\Sigma^r\subset\mathbb R^{r+1}$
be a paraboloid of revolution which is a graph over $\mathbb R^r$, and
let $\pi:\Sigma^r\to\mathbb R^r$ be the restriction of the standard
projection. Then
$$
\operatorname{id}_{\mathbb{S}^k}\times\pi:
\mathbb{S}^k\times\Sigma^r\longrightarrow \mathbb{S}^k\times\mathbb R^r
$$
is a proper $1$-Lipschitz map of non-zero degree, while
$$
R_{\mathbb{S}^k\times\Sigma^r}=k(k-1)+R_{\Sigma^r}>k(k-1).
$$
Nevertheless, $\mathbb{S}^k\times\Sigma^r$ is not isometric to
$\mathbb{S}^k\times\mathbb R^r$.

Although rigidity fails, a scalar curvature obstruction persists. The following theorem rules out a uniform positive weighted scalar curvature gap at infinity. When there are at least three Euclidean factors, we impose a coarse Lipschitz condition on the Euclidean component.
\begin{theorem}\label{thm:infinity}
\textit{Let $k\ge2$, $m\ge0$, $r\ge2$, and $n=k+m+r$. Suppose that $(M^n,g)$ is a complete, connected, oriented manifold without boundary and that
$$
F=(A_{\mathbb S},A_{\mathbb T},b):M\longrightarrow\mathbb S^k\times\mathbb T^m\times\mathbb R^r
$$
is a smooth proper map of non-zero degree, with $A_{\mathbb S}$ $1$-Lipschitz. If $r\ge3$, assume in addition that there are constants $L>0$ and $C_0\ge0$ such that
\eqlabel{eq:infinity-coarse}
$$
|b(x)-b(y)|\le Ld_g(x,y)+C_0
\qquad\text{for all }x,y\in M.
\leqno(\theequation)
$$
No additional condition on $b$ is required when $r=2$. Then, for every $\psi\in C^\infty(M)$,
$$
\liminf_{x\to\infty}\mathcal R_g(\psi)(x)\le k(k-1).
$$}
\end{theorem}

\begin{remark}\label{rem:low-dim}
When $3\le n\le 7$, all $\mu$-bubbles constructed below are smooth and compact, no conformal blow up is needed, and all the sphere-torus data of Section \ref{sec:descent} are closed manifolds. In this range the existence and regularity of the $\mu$-bubbles follow from \cite{Cho-Li}, and Appendix \ref{app:bubbles} is not needed. The nonspin results of \cite{WWXZ2026}, namely Theorem \ref{thm:WWXZ-Llarull} and the scalar-mean comparison theorem quoted as Theorem \ref{thm:WWXZ-mean}, are used in all dimensions.
\end{remark}

We briefly describe the proof of Theorem~\ref{main}.  The real component of the proper map has a
compact regular level on which the sphere-torus component of the map still has non-zero degree. A proper height function then determines a compact Riemannian band. Minimizing a warped $\mu$-bubble functional inside this band produces a compact hypersurface
without taking a limit at infinity. In dimensions at most seven this hypersurface is smooth. In general its singular set has codimension at least seven, and we follow Bi-Zhu \cite{BZ2026}: a conformal blow up near the singular set makes the regular part complete, and warping by a positive supersolution of the stability operator turns the stability inequality into a weighted scalar curvature lower bound on the regular part. The torus directions are then removed one circle at a time. In the infinite cyclic covering associated with the last circle factor, a band is cut out around a lift of a level set, and the $\mu$-bubble, blow up and warping steps are repeated. After removing all circle factors, we obtain a complete, possibly noncompact, $k$ dimensional manifold carrying a $1$-Lipschitz map of non-zero degree to $\mathbb S^k$ and a weighted scalar curvature gap above $k(k-1)$. This terminal configuration is excluded by the scalar-mean comparison theorem of Wang-Wang-Xie-Zhu \cite[Theorem 3.3]{WWXZ2026}, applied to an explicit conformal deformation of a half cylinder. This yields the noncompact weighted obstruction used in place of the compactness step that requires bounded geometry. The weighted obstruction has two consequences.  Kazdan's positive supersolution theorem \cite{Kazdan} first forces the scalar curvature to be
identically $k(k-1)$.  If the Ricci tensor had a negative direction, a
compactly supported deformation enlarging the metric in that direction
would preserve the spherical $1$-Lipschitz condition and create the local
spectral positivity needed for the same supersolution argument.  The
weighted obstruction therefore forces nonnegative Ricci curvature.  The
proper real component also implies that the manifold has at least two ends,
so the Cheeger-Gromoll theorem splits it as $N\times\mathbb R$ with $N$
compact.  A compactly supported cohomological degree argument shows that the
map induced on every product slice $N\times\{s\}$ has non-zero degree.
Finally, since $\operatorname{Ric}_N\ge0$, this map produces parallel $1$-forms with integral periods on $N$. They fiber $N$ over a torus with totally geodesic fibers, Theorem \ref{thm:WWXZ-Llarull} identifies a fiber with the round sphere, and an analysis of the deck transformation group
shows that $N$ is the product of the round sphere with a flat torus.  The
spherical component of every deck transformation is eliminated by the
lifted $1$-Lipschitz map, and the remaining Euclidean translations form a
full lattice.

Theorem~\ref{thm:infinity} is proved by pulling back connected bands of large width from the Euclidean factor, as in the torical band construction of \cite[Section~3.4]{HSS}, and by a weighted band width estimate (Proposition~\ref{prop:bandwidth}), which is proved by the same descent. This estimate replaces the band width estimate of Chow \cite[Theorem~B(ii)]{chow2025}, which requires $n\le7$ and a spin structure.

The paper is organized as follows.  Section \ref{sec-2} collects the conventions for weighted scalar curvature used in the dimension descent, the degree slicing lemmas and the Riccati barriers. Section \ref{sec:bubbles} discusses warped $\mu$-bubbles in bands, their stability inequality, and the conformal blow up and warping of Bi-Zhu. Section \ref{sec-sup} proves closed sphere-torus rigidity under nonnegative Ricci curvature. Section \ref{sec:terminal} excludes the terminal configuration, and Section \ref{sec:descent} carries out the circle-by-circle descent for sphere-torus data. Section \ref{w-o} proves the noncompact weighted obstruction and the weighted band width estimate. Section \ref{sec-scalar-ricci} establishes scalar
curvature equality and nonnegative Ricci curvature. Section~\ref{sec-main} contains the proofs of Theorems~\ref{main} and~\ref{thm:infinity} and of Corollary~\ref{cor:closed-rigidity}. Appendix~\ref{app:bubbles} contains the proofs of the existence and structure results for warped $\mu$-bubbles stated in Section \ref{sec:bubbles}.\\

\noindent\textbf{AI usage:}\, In preparing this version, the author used generative AI tools to assist in developing the dimension descent argument, particularly the parts involving geometric measure theory. The author has checked the entire manuscript and takes full responsibility for its mathematical content.\\

\noindent\textbf{Acknowledgements:}\, The author would like to thank Professor Linfeng Zhou for his continued encouragement and support. The author also thanks Thomas Tony for his interest in this work and for pointing out the related results in his joint work with Oskar Riedler.

    \section{Weighted scalar curvature and degree}\label{sec-2}
    \subsection{Weighted scalar curvature}\label{subsec:weighted}
Let $(N^d,g)$ be a Riemannian manifold and $\psi\in C^\infty(N)$. We use the weighted scalar curvature $\mathcal R^a_g(\psi)$ of \eqref{eq:Ra}. As in the warped $\mu$-bubble construction of \cite{Cho-Li}, the weight is always $e^\psi$: volumes and areas are computed with respect to $e^\psi d\mathcal H^d$ and $e^\psi d\mathcal H^{d-1}$, and the drift Laplacian $\Delta_g+\langle\nabla_g\psi,\nabla_g\,\cdot\,\rangle=e^{-\psi}\operatorname{div}_g(e^\psi\nabla_g\,\cdot\,)$ is symmetric with respect to $e^\psi d\mathcal H^d$. For constants $a$ and $\sigma$ we consider the quadratic form
\begin{equation}\label{eq:Q}
\mathcal Q^{g,\psi}_{a,\sigma}(\varphi)=\int_N e^\psi|\nabla_g\varphi|^2\,d\mathcal H^d+\frac12\int_N e^\psi\left(\mathcal R^a_g(\psi)-\sigma\right)\varphi^2\,d\mathcal H^d,\qquad \varphi\in C^\infty_c(N),
\end{equation}
and the corresponding operator
\begin{equation}\label{eq:L}
L^{g,\psi}_{a,\sigma}=-\Delta_g-\langle\nabla_g\psi,\nabla_g\,\cdot\,\rangle+\frac12\left(\mathcal R^a_g(\psi)-\sigma\right),
\end{equation}
so that $\mathcal Q^{g,\psi}_{a,\sigma}(\varphi)=\int_N\varphi\, L^{g,\psi}_{a,\sigma}\varphi\, e^\psi d\mathcal H^d$. We write $\mathcal Q^{g,\psi}_{a,\sigma}\ge0$ if this form is nonnegative on $C^\infty_c(N)$.

Bi-Zhu \cite{BZ2026} use the weight $e^{-f}$. Their weighted scalar curvature $R_g+2\Delta_gf-a|\nabla_gf|^2$ equals $\mathcal R^a_g(-f)$, and all statements quoted from \cite{BZ2026} below are translated by $\psi=-f$. The scalar-mean comparison theorem of \cite{WWXZ2026} is formulated with the weighted scalar curvature $\mathrm{Sc}^{\mu,f}(g)=\mathrm{Sc}(g)-2\Delta f-\mu|\nabla f|^2=\mathcal R^\mu_g(f)$ and the weighted mean curvature $H_{\partial M}(g)+\nu(f)$ of \cite[Definition 3.2]{WWXZ2026}, which are already in our convention.

For $d\geq1$ and $\lambda<1/d$, Bi-Zhu \cite[(2.1)--(2.2)]{BZ2026} set
$$
a_d(\lambda)=\frac{1-(d-1)\lambda}{1-d\lambda}.
$$
In the dimension descent, the weighted scalar curvature of a $d$ dimensional manifold is measured by $\mathcal R^{a_d(\lambda)}$. After passing to a stable hypersurface and warping, it is measured by $\mathcal R^{a_{d-1}(\lambda)}$ (see Lemmas \ref{lem:pointwise} and \ref{lem:warping}). The conformal blow up of \cite[Proposition 3.12]{BZ2026} on an $m$ dimensional space involves the quantities
$$
\bar q_m(\lambda)=\sqrt{\frac{m\left(1-(m-1)\lambda\right)}{2(1-\lambda)}},\qquad
\bar\tau_m(\lambda)=\frac{m-2}{1+\bar q_m(\lambda)},
$$
see \cite[(2.11), (3.27)]{BZ2026}, and exponents $\bar\alpha_m(\lambda)$, $\bar\beta_m(\lambda)$, $\bar\gamma_m(\lambda)$, see \cite[(2.12)--(2.14)]{BZ2026}. Throughout this paper we fix
$$
\lambda=-\frac12,\qquad a_d\coloneqq a_d\bigl(-\tfrac12\bigr)=\frac{d+1}{d+2},
$$
and write $\bar q_m=\bar q_m(-\frac12)$, with analogous abbreviations for the other constants.

\begin{lemma}\label{lem:parameter}
\textit{Let $\lambda=-\frac12$. Then the following hold.
\begin{itemize}
\item[(1)] For every $d\ge1$, $a_d=\frac{d+1}{d+2}\in(0,1)$, and $\frac{1}{1-\lambda}=\frac23$.
\item[(2)] For every $d\ge2$, $a_{d-1}=2-\frac{1}{a_d}$.
\item[(3)] For every $k\ge1$, $a_k>\frac{k}{k+1}$.
\item[(4)] For every $m\ge3$, $\lambda<\frac1{m+1}$, $\bar\tau_m<\sqrt6<3$, and $\bar\alpha_m=\frac{2(1+\bar q_m)}{m-2}>0$.
\item[(5)] For every Riemannian manifold $(N^d,g)$ and every $\psi\in C^\infty(N)$, $\mathcal R^{a_d}_g(\psi)\ge\mathcal R_g(\psi)$.
\end{itemize}}
\end{lemma}
\begin{proof}
(1) follows by substitution: $a_d=\bigl(1+\frac{d-1}2\bigr)/\bigl(1+\frac d2\bigr)=\frac{d+1}{d+2}$. (2) is the identity $\frac{d}{d+1}=2-\frac{d+2}{d+1}$. The corresponding identity for general $\lambda$ is given in \cite[(2.3)]{BZ2026}. (3) is equivalent to $(k+1)^2>k(k+2)$. In (4), the first inequality follows immediately. Moreover, $\bar q_m^2=\frac{m(m+1)}{6}>\frac{m^2}6$, so
$$
\bar\tau_m=\frac{m-2}{1+\bar q_m}<\frac{m-2}{\bar q_m}<\frac{\sqrt6\,(m-2)}{m}<\sqrt6.
$$
The positivity of $\bar\alpha_m$ is clear. Finally (5) follows from $a_d<1$.
\end{proof}

\begin{remark}\label{rem:lambda}
The descent from dimension $d$ to dimension $d-1$ uses $\lambda<1/d$ in Lemma \ref{lem:pointwise} and $a_d>0$ in Lemma \ref{lem:warping}. When the $\mu$-bubble is singular it also uses the hypotheses $\lambda<\frac1d$ and $\bar\tau_{d-1}(\lambda)<\tau_{\mathrm{pack}}$ of \cite[Proposition 3.12]{BZ2026}. Our $\mu$-bubbles satisfy $\tau_{\mathrm{pack}}\ge3$ by Proposition \ref{prop:bubble}. The terminal step uses $a_k>\frac{k}{k+1}$, which is the parameter condition $\mu>\frac{k}{k+1}$ of \cite[Theorem 3.3]{WWXZ2026} in dimension $k+1$. For general $\lambda$, the inequality $a_k(\lambda)>\frac{k}{k+1}$ is equivalent to $\lambda>-1$. Since $\bar\tau_m(\lambda)$ is increasing in $\lambda$, negative values of $\lambda$ are favorable for the descent. The choice $\lambda=-\frac12$ satisfies all these conditions in every dimension.
\end{remark}

\subsection{Slicing a proper map of non-zero degree}

We first give the topological use of the real component.

\begin{lemma}\label{lem-degree-slice}
\textit{
Let $X^{n-1}$ be a closed, connected, and oriented manifold. Let $M^n$ be a connected
and oriented manifold, and let $F=(A,b):M\to X\times\mathbb{R}$ be a smooth and proper map with non-zero degree. Then $b$ is proper and surjective. For any regular value
$t_0$ of $b$, the level set $S_0=b^{-1}(t_0)$ is a closed
hypersurface, and
$$
\deg(A|_{S_0})=\deg F\ne0.
$$
}
\end{lemma}

\begin{proof}
$b$ is surjective because $F$ is surjective. For a compact set $K\subset\mathbb{R}$, the identity
$b^{-1}(K)=F^{-1}(X\times K)$ gives properness.

Choose a top form $\Theta$ on $X$ with $\int_X\Theta=1$ and put
$\omega=A^*\Theta$. Choose a compactly supported function $\zeta$ on $\mathbb{R}$ such that $\int_{\mathbb{R}}\zeta(t)dt=1$. By the degree formula and the coarea formula, we have
\begin{align*}
    \mathrm{deg}(F)=&\int_M\omega\wedge\zeta(b)\,db\\
=&\int_\mathbb{R}\zeta(t)\left(\int_{b^{-1}(t)}\omega\right)dt.
\end{align*}
For regular values $t_1<t_2$, the region
$b^{-1}([t_1,t_2])$ is compact, and by Stokes' theorem we obtain
$$
\int_{b^{-1}(t_1)}\omega-\int_{b^{-1}(t_2)}\omega=\int_{b^{-1}([t_1,t_2])}d\omega=0.
$$
It follows that
$$
\mathrm{deg}(F)=\int_{b^{-1}(t_0)}\omega=\mathrm{deg}(A|_{S_0}).
$$
This completes the proof.
\end{proof}

We use the height function constructed in the proof of \cite[Lemma~2.1]{Zhu1}, choosing the signed distance to be negative on $\{b<t_0\}$.
\begin{lemma}\label{lem-proper-fcn}
\textit{
Let $(M,g)$ be a complete, connected Riemannian manifold, let $b:M\to\mathbb R$ be a smooth, proper and surjective function, let $t_0$ be a regular value of $b$, and put $S_0=b^{-1}(t_0)$. Then there
is a smooth, proper and surjective function $\phi:M\to\mathbb{R}$ such that
$$
\phi^{-1}(0)=S_0,
\qquad |\nabla_M\phi|<1, \qquad \{\phi<0\}=\{b<t_0\},
$$
and $0$ is a regular value of $\phi$.}
\end{lemma}

We will compute degrees on the new hypersurfaces using the following intersection count along compact curves.

\begin{lemma}\label{lem:intersection}
\textit{Let $P^n$ and $Z^{n-1}$ be oriented manifolds without boundary, let $G:P\to Z$ be smooth, let $y\in Z$ be a regular value of $G$, and let $\Gamma\subset G^{-1}(y)$ be a compact one dimensional submanifold, possibly with boundary. Orient $\Gamma$ as follows: if $x\in\Gamma$, $\mathcal K\subset T_xP$ is complementary to $T_x\Gamma$, and $(h_1,\ldots,h_{n-1})$ is a basis of $\mathcal K$ such that $(dG\,h_1,\ldots,dG\,h_{n-1})$ is positively oriented in $T_yZ$, then $(\dot\Gamma,h_1,\ldots,h_{n-1})$ is positively oriented in $T_xP$. For $e\in\partial\Gamma$ put $o(e)=+1$ if $e$ is a terminal point of $\Gamma$ and $o(e)=-1$ if it is an initial point. Let $E\subset P$, and suppose that on a neighborhood $O$ of $\Gamma$ the set $\partial E\cap O$ is a smooth embedded hypersurface $\Sigma$ with $E\cap O$ on one side of $\Sigma$, that $\Gamma$ is transverse to $\Sigma$, and that $\partial\Gamma\cap\Sigma=\varnothing$. Orient $\Sigma$ so that $(\nu,\sigma_1,\ldots,\sigma_{n-1})$ is positively oriented in $T_xP$ if and only if $(\sigma_1,\ldots,\sigma_{n-1})$ is positively oriented in $T_x\Sigma$, where $\nu$ points out of $E$. Then
$$
\sum_{x\in\Gamma\cap\Sigma}\operatorname{sgn}\det\left(dG|_{T_x\Sigma}\right)=-\sum_{e\in\partial\Gamma}o(e)\chi_E(e).
$$
In particular, the left hand side depends only on $E$ near $\partial\Gamma$. Moreover, at each point $x\in\Gamma\cap\Sigma$ the sign $\operatorname{sgn}\det(dG|_{T_x\Sigma})$ equals $+1$ if and only if $\Gamma$ leaves $E$ at $x$.}
\end{lemma}
\begin{proof}
Let $x\in\Gamma\cap\Sigma$. Since $dG_x$ vanishes on $T_x\Gamma$, is surjective, and $\dot\Gamma\notin T_x\Sigma$, the restriction $dG|_{T_x\Sigma}$ is an isomorphism. Taking $\mathcal K=T_x\Sigma$ in the orientation convention, we see that if $(\sigma_j)$ is a positively oriented basis of $T_x\Sigma$, then the orientation sign of $(\dot\Gamma,\sigma_1,\ldots,\sigma_{n-1})$ in $T_xP$ equals $\operatorname{sgn}\det(dG|_{T_x\Sigma})$. On the other hand, writing $\dot\Gamma=c\,\nu+(\text{tangential})$, this sign equals $\operatorname{sgn}c$, and $c>0$ if and only if $\Gamma$ leaves $E$ at $x$. This proves the last assertion. Put $\epsilon(x)=+1$ if $\Gamma$ leaves $E$ at $x$ and $\epsilon(x)=-1$ otherwise. Along an arc of $\Gamma$, from its initial point to its terminal point, $\chi_E$ changes only at intersection points, decreasing by one when $\Gamma$ leaves $E$ and increasing by one when it enters $E$. Hence $\chi_E(\text{terminal point})-\chi_E(\text{initial point})=-\sum\epsilon(x)$ over the intersection points on that arc, while on a closed component of $\Gamma$ the sum of $\epsilon$ vanishes. Summing over the components gives the claim.
\end{proof}

For a smooth map $G=(A,t):P^n\to Y^{n-1}\times\mathbb R$ between oriented manifolds and a regular level $S=t^{-1}(t_1)$ with its boundary orientation, $(y,t_1)$ is a regular value of $G$ if and only if $y$ is a regular value of $A|_S$. The corresponding local degrees differ by a common sign, so their sums over a finite common fiber are simultaneously nonzero.

\subsection{Riccati barriers}
The construction of the function $h$ in the following lemma is a slight modification of the one in \cite[Lemma 2.3]{Zhu1}, adapted to the weighted stability inequality of Lemma \ref{lem:pointwise}. The coefficient $\frac13=\frac{1}{2(1-\lambda)}$ comes from the choice $\lambda=-\frac12$.

\begin{lemma}\label{lem-h}
\textit{
Let $(M^n,g)$ be a complete open Riemannian manifold with $n\geq3$.
Let $E:M\to(0,\infty)$ be a continuous function, and let
$\phi:M\to\mathbb{R}$ be a smooth, proper, and surjective function with
$|\nabla\phi|<1$.
Then there exist $a>0$ such that $\pm a$ are regular values of $\phi$, a smooth strictly decreasing function
$h:(-a,a)\to\mathbb{R}$ satisfying
$$
\lim_{t\to-a}h(t)=+\infty,
\qquad
\lim_{t\to a}h(t)=-\infty,
$$
and a constant $\delta_1>0$ such that
$$
E+\frac13h(\phi)^2-|\nabla(h\circ\phi)|\ge\delta_1
\quad\text{on }\phi^{-1}((-a,a)).
$$
}
\end{lemma}

\begin{proof}
For $0<\varepsilon<1$, let $h_\varepsilon=\frac{n}{n-1}z_\varepsilon$, where $z_\varepsilon$ is the barrier in \cite[Lemma 2.3]{Zhu1}. On $\left(-\frac{1}{n\varepsilon},\frac{1}{n\varepsilon}\right)$ we have $h_\varepsilon'<0$, and $h_\varepsilon$ tends to $+\infty$ at the left endpoint and to $-\infty$ at the right endpoint. With $I=[-1/(2n),1/(2n)]$, the estimates for $z_\varepsilon$ give
$$
\frac12h_\varepsilon^2+h_\varepsilon'=\frac{n^2\varepsilon^2}{2}\quad\text{outside }I,
\qquad
\sup_I\left|\frac12h_\varepsilon^2+h_\varepsilon'\right|\le C(n)\varepsilon.
$$
Since $\phi$ is proper, the set $K=\phi^{-1}(I)$ is compact. Hence
$\delta=\min_K E>0.$
Choose $\varepsilon_0\in(0,1)$ sufficiently small such that
$$
\frac12h_\varepsilon^2+h_\varepsilon'
\ge
-\frac{\delta}{3}
\quad\text{on }I
$$
whenever $0<\varepsilon<\varepsilon_0$. By Sard's theorem, there are arbitrarily large numbers $a$ for which
both $a$ and $-a$ are regular values of $\phi$.  Choose such an $a$
satisfying $a>\frac{1}{n\varepsilon_0},$
and set
$$
\varepsilon=\frac{1}{na},
\qquad
h=\frac32h_\varepsilon.
$$
The strict monotonicity and endpoint behavior of $h$ follow directly
from the corresponding properties of $h_\varepsilon$, and
$$
\frac13h^2+h'=\frac32\left(\frac12h_\varepsilon^2+h_\varepsilon'\right).
$$
Since $h'<0$ and $|\nabla\phi|<1$, we have
$|\nabla(h\circ\phi)|=-h'(\phi)|\nabla\phi|\le-h'(\phi)$. It follows that
$$
E+\frac13h(\phi)^2-|\nabla(h\circ\phi)|
\ge
E+\frac32\left(\frac12h_\varepsilon^2+h_\varepsilon'\right)(\phi).
$$
On $K$, the right hand side is at least
$\delta-\frac{\delta}{2}=\frac{\delta}{2}$. Outside $K$, it is at least
$\frac32\cdot\frac{n^2\varepsilon^2}{2}=\frac{3}{4a^2}.$
Hence the lemma holds with $\delta_1=\min\{\frac\delta2,\frac{3}{4a^2}\}$.
\end{proof}

When the weighted scalar curvature has a uniform gap, we use the following tangent barrier, which is the form of \cite[(4.1)]{BZ2026} for $\lambda=-\frac12$.

\begin{lemma}\label{lem:tan}
\textit{For $L>0$ let
$$
h_L(t)=-\frac{3\pi}{2L}\tan\frac{\pi t}{2L},\qquad t\in(-L,L).
$$
Then $h_L$ is smooth and strictly decreasing, $h_L(t)\to\pm\infty$ as $t\to\mp L$, and for every smooth function $s$ with $|\nabla s|\le1$ and $|s|<L$,
$$
-2\left|\nabla(h_L\circ s)\right|+\frac23\left(h_L\circ s\right)^2\geq-\frac{3\pi^2}{2L^2}.
$$}
\end{lemma}

\section{Warped \texorpdfstring{$\mu$}{mu}-bubbles in bands}\label{sec:bubbles}

We adapt the warped $\mu$-bubble construction of Bi-Zhu \cite[Proposition 4.6]{BZ2026}. We allow potentials other than the tangent potential in \cite[(4.1)]{BZ2026} and compute perimeter in open subsets of the ambient manifold, so jumps at the two ends contribute to the functional (Remark \ref{rem:end-jump}).

\subsection{Bands}
Let $(P^d,g)$ be a Riemannian manifold without boundary, $d\geq3$, let $\psi\in C^\infty(P)$, and let $s\in C^\infty(P)$ and $L>0$ be such that $0$ and $\pm L$ are regular values of $s$. Put
$$
W=s^{-1}([-L,L]),\qquad \partial_\pm W=s^{-1}(\pm L),\qquad \Omega_0=\{s\le0\},
$$
and assume that $\partial_-W\neq\varnothing$ and $\partial_+W\neq\varnothing$. We call $W$ a band, which need not be compact. We always assume that
\begin{equation}\label{eq:band-assumption}
\text{either }(P,g)\text{ is complete, or }W\text{ is compact.}
\end{equation}
Let $h\in C^\infty((-L,L))$ satisfy $h(t)\to\pm\infty$ as $t\to\mp L$, and put $\widehat h=h\circ s$ on $\overset{\circ}{W}$. For a set $\Omega\subset P$ of locally finite perimeter with $\Omega\triangle\Omega_0\subset W$ and a precompact open set $U\subset P$, consider
\begin{equation}\label{eqn-omega}
\mathcal A_{\widehat h}(\Omega;U) = \int_{\partial^{\ast} \Omega\cap U}e^\psi d\mathcal{H}^{d-1} - \int_{U\cap\overset{\circ}{W}} (\chi_{\Omega}
		- \chi_{{\Omega_0}})\widehat he^\psi\,d\mathcal{H}^{d},
\end{equation}
where $\partial^\ast\Omega$ denotes the reduced boundary. Note that the perimeter is computed in the open subset $U$ of $P$. In particular, if $\Omega$ fails to contain $\partial_-W$ or to avoid $\partial_+W$, this is counted in the first term. The integral in the second term is well defined in $[-\infty,+\infty)$, so that $\mathcal A_{\widehat h}(\Omega;U)\in(-\infty,+\infty]$ (see Subsection \ref{app:existence}). We call $\Omega$ a \emph{warped $\mu$-bubble in $W$} if
$$
\mathcal A_{\widehat h}(\Omega;U)\le\mathcal A_{\widehat h}(\Omega';U)
$$
whenever $U\subset P$ is open and precompact and $\Omega'\triangle\Omega$ is contained in a compact subset of $U\cap\overset{\circ}{W}$. Sets of locally finite perimeter are always normalized so that their topological boundary is the closure of the reduced boundary \cite[Proposition 12.19]{Maggi}. If $P$ is not connected, we use the convention $d_g(x,y)=\infty$ for points in different components, and statements about the metric space $(\Sigma,d_g|_\Sigma)$ below are understood for $\Sigma\cap C$, separately for every component $C$ of $P$.

We use the following notions from \cite[Section 3.1]{BZ2026}. Let $(X,\mathsf d)$ be a complete metric space with a decomposition $X=\mathcal R\sqcup\mathcal S$, where $\mathcal S$ is closed and $\mathcal R$ is an $m$ dimensional smooth manifold with a Riemannian metric $g_{\mathcal R}$, such that every point of $\mathcal R$ has an open neighborhood $U\subset\mathcal R$ on which the length metric induced by $\mathsf d$ agrees with the intrinsic Riemannian distance of $(U,g_{\mathcal R})$, both computed with curves in $U$. The metric $\mathsf d$ itself need not be a length metric: in our application it is the restriction of the ambient distance to a hypersurface, which is in general smaller than the intrinsic distance. A positive function $\omega\in C^\infty(\mathcal R)$ is an \emph{admissible weight} \cite[Definition 3.1]{BZ2026} if the measure $\mu_\omega=\omega\,d\mu_{g_{\mathcal R}}$ on $\mathcal R$, extended by zero to $\mathcal S$, satisfies, uniformly on compact subsets of $X$, the Ahlfors regularity condition \cite[(3.1)]{BZ2026} in dimension $m$ and the local $(1,2)$-Poincar\'e inequality \cite[(3.2)]{BZ2026} (see Lemma \ref{lem:poincare}). If $\omega\equiv1$ is admissible, we call $(X,\mathsf d,\mu)$, $\mu=\mu_1$, an \emph{$m$ dimensional AM-PI space} \cite[Section 3.1]{BZ2026}. By \cite[Lemma 3.2]{BZ2026}, if $\omega$ is admissible and $\theta\in C^\infty(\mathcal R)$ is bounded on $K\cap\mathcal R$ for every compact $K\subset X$, then $e^\theta\omega$ is admissible. Finally, for a closed set $Z\subset X$ and $0\le\kappa\le m$ we write $\mathsf P^X_\kappa(Z)$, see \cite[Definition 3.3]{BZ2026}, if for every compact $K\subset X$ there are $C_K<\infty$ and $r_K>0$ such that $\#(P_\varrho\cap B_r(x))\le C_K(r/\varrho)^\kappa$ whenever $x\in K$, $0<\varrho<r<r_K$, and $P_\varrho\subset Z\cap K$ is $\varrho$-separated. With $d_{\mathrm{pack}}(Z)=\inf\{\kappa:\mathsf P^X_\kappa(Z)\}$ and $\tau_{\mathrm{pack}}(Z)=m-2-d_{\mathrm{pack}}(Z)$ as in \cite[(3.4)]{BZ2026}, the condition $\mathsf P^X_\kappa(Z)$ implies $\tau_{\mathrm{pack}}(Z)\ge m-2-\kappa$.

\begin{proposition}\label{prop:bubble}
\textit{In the setting above, there exists a warped $\mu$-bubble $\Omega$ in $W$ with $\Omega\triangle\Omega_0\subset W$ and the following properties.
\begin{itemize}
\item[(1)] (Barriers) Every point of $\partial_-W$ has a neighborhood contained in $\Omega$, and every point of $\partial_+W$ has a neighborhood disjoint from $\Omega$. Consequently, $\Sigma\coloneqq\partial\Omega$ is contained in $\overset{\circ}{W}$ and is closed in $P$. In particular, $(\Sigma,d_g|_\Sigma)$ is a complete metric space, and $\Sigma$ is compact if $W$ is compact.
\item[(2)] (Regularity) $\Sigma=\Sigma_{\mathrm{reg}}\sqcup D_\Sigma$, where $\Sigma_{\mathrm{reg}}$ is a smooth, embedded, two-sided hypersurface and $D_\Sigma$ is closed with $\dim_{\mathcal H}D_\Sigma\le d-8$. In particular, $D_\Sigma=\varnothing$ if $d\le7$.
\item[(3)] (Variations) Let $\nu$ be the unit normal of $\Sigma_{\mathrm{reg}}$ pointing out of $\Omega$, $H=\operatorname{div}_\Sigma\nu$ its mean curvature and $A_\Sigma$ its second fundamental form. Then
$$
H=-\langle\nabla_P\psi,\nu\rangle+\widehat h\qquad\text{on }\Sigma_{\mathrm{reg}},
$$
and for every $\varphi\in C^\infty_c(\Sigma_{\mathrm{reg}})$,
\begin{equation}\label{eq:stability-bubble}
\int_{\Sigma}e^\psi|\nabla_\Sigma\varphi|^2\ge\int_\Sigma e^\psi\left(\operatorname{Ric}_g(\nu,\nu)+|A_\Sigma|^2-\nabla^2_P\psi(\nu,\nu)+\langle\nabla_P\widehat h,\nu\rangle\right)\varphi^2.
\end{equation}
\item[(4)] (Analytic structure) If $D_\Sigma\neq\varnothing$, then $(\Sigma,d_g|_\Sigma,\mathcal H^{d-1}\llcorner\Sigma)$, with regular part $\Sigma_{\mathrm{reg}}$ and singular part $D_\Sigma$, is a $(d-1)$ dimensional AM-PI space, and $D_\Sigma$ satisfies the packing condition $\mathsf P^\Sigma_{d-6}(D_\Sigma)$. In particular, $\tau_{\mathrm{pack}}(D_\Sigma)\ge3$.
\end{itemize}}
\end{proposition}

When $d\le7$ and $W$ is compact, Proposition \ref{prop:bubble} is contained in \cite[Proposition 12, Lemmas 13 and 14]{Cho-Li}. Note that \eqref{eq:stability-bubble} is the second variation formula of \cite[Lemma 14]{Cho-Li} before the Gauss equation is used. Bi-Zhu \cite[Proposition 4.6]{BZ2026} state the sharper packing condition $\mathsf P^\Sigma_{d-8}(D_\Sigma)$. We prove and use only $\mathsf P^\Sigma_{d-6}(D_\Sigma)$, which suffices because $\bar\tau_{d-1}<3$ by Lemma \ref{lem:parameter}. The proof of Proposition \ref{prop:bubble} is given in Appendix \ref{app:bubbles}.

\subsection{Pointwise stability}
The next lemma follows from the pointwise argument in \cite[Proposition 2.5]{BZ2026}, taking the lower bound $\sigma$ there to be the weighted scalar curvature itself. We include the argument in our convention.

\begin{lemma}\label{lem:pointwise}
\textit{Let $(N^d,g)$ be a Riemannian manifold, let $\lambda<1/d$ and $a=a_d(\lambda)$. Let $\Sigma^{d-1}\subset N$ be a smooth, embedded, two-sided hypersurface with unit normal $\nu$, mean curvature $H=\operatorname{div}_\Sigma\nu$ and second fundamental form $A_\Sigma$, let $\psi\in C^\infty(N)$, and let $\widehat h$ be a smooth function near $\Sigma$ such that $H=-\langle\nabla_N\psi,\nu\rangle+\widehat h$ on $\Sigma$ and
\begin{equation}\label{eq:stability}
\int_\Sigma e^\psi|\nabla_\Sigma\varphi|^2\ge\int_\Sigma e^\psi\left(\operatorname{Ric}_g(\nu,\nu)+|A_\Sigma|^2-\nabla^2_N\psi(\nu,\nu)+\langle\nabla_N\widehat h,\nu\rangle\right)\varphi^2
\end{equation}
for all $\varphi\in C^\infty_c(\Sigma)$. Put $\mathcal V=\mathcal R^{a}_g(\psi)$, and let $g_\Sigma$ be the induced metric and $\psi_\Sigma=\psi|_\Sigma$. Then for all $\varphi\in C^\infty_c(\Sigma)$,
\begin{equation}\label{eq:pointwise-Q}
\int_\Sigma e^{\psi_\Sigma}|\nabla_\Sigma\varphi|^2+\frac12\int_\Sigma e^{\psi_\Sigma}\left(\mathcal R^{a}_{g_\Sigma}(\psi_\Sigma)-\sigma_\Sigma\right)\varphi^2\ge0,\qquad \sigma_\Sigma=\mathcal V+2\langle\nabla_N\widehat h,\nu\rangle+\frac{\widehat h^2}{1-\lambda}.
\end{equation}
In particular, if $\sigma_*$ is a constant and
$$
\mathcal V-2|\nabla_N\widehat h|+\frac{\widehat h^2}{1-\lambda}\ge\sigma_*\qquad\text{on a neighborhood of }\Sigma,
$$
then $\mathcal Q^{g_\Sigma,\psi_\Sigma}_{a,\sigma_*}\ge0$.}
\end{lemma}
\begin{proof}
By \eqref{eq:stability}, it suffices to show pointwise on $\Sigma$ that
$$
I\coloneqq\operatorname{Ric}_g(\nu,\nu)+|A_\Sigma|^2-\nabla^2_N\psi(\nu,\nu)+\langle\nabla_N\widehat h,\nu\rangle+\frac12\mathcal R^a_{g_\Sigma}(\psi_\Sigma)\ge\frac12\sigma_\Sigma.
$$
Write $u=\langle\nabla_N\psi,\nu\rangle$. By the Gauss equation $R_\Sigma=R_N-2\operatorname{Ric}_g(\nu,\nu)+H^2-|A_\Sigma|^2$ and the identities
$$
\Delta_\Sigma\psi_\Sigma=\Delta_N\psi-\nabla^2_N\psi(\nu,\nu)-Hu,\qquad |\nabla_\Sigma\psi_\Sigma|^2=|\nabla_N\psi|^2-u^2,
$$
we obtain
$$
I=\frac12\left(R_N-2\Delta_N\psi\right)-\frac a2|\nabla_\Sigma\psi_\Sigma|^2+\frac12|A_\Sigma|^2+\frac12H^2+Hu+\langle\nabla_N\widehat h,\nu\rangle.
$$
Since $R_N-2\Delta_N\psi=\mathcal V+a|\nabla_N\psi|^2=\mathcal V+a|\nabla_\Sigma\psi_\Sigma|^2+au^2$ and $|A_\Sigma|^2\ge H^2/(d-1)$,
$$
I\ge\frac12\mathcal V+\langle\nabla_N\widehat h,\nu\rangle+\frac a2u^2+\frac{d}{2(d-1)}H^2+Hu.
$$
Substituting $H=\widehat h-u$ gives
\begin{align*}
&\frac a2u^2+\frac{d}{2(d-1)}(\widehat h-u)^2+(\widehat h-u)u\\
&\qquad=\frac{\widehat h^2}{2(1-\lambda)}+\frac B2\left(u-\frac{\widehat h}{(d-1)B}\right)^2,\qquad B=\frac{1-\lambda}{(d-1)(1-d\lambda)}>0,
\end{align*}
which is the identity obtained by completing the square \cite[(2.24)]{BZ2026} in our convention. For $\lambda=-\frac12$ it reads
$$
\frac{d+1}{2(d+2)}u^2+\frac{d}{2(d-1)}(\widehat h-u)^2+(\widehat h-u)u-\frac{\widehat h^2}3=\frac{3}{2(d-1)(d+2)}\left(u-\frac{d+2}{3}\widehat h\right)^2.
$$
Hence $I\ge\frac12\mathcal V+\langle\nabla_N\widehat h,\nu\rangle+\frac{\widehat h^2}{2(1-\lambda)}=\frac12\sigma_\Sigma$. Finally, $\langle\nabla_N\widehat h,\nu\rangle\ge-|\nabla_N\widehat h|$ gives $\sigma_\Sigma\ge\sigma_*$ under the last assumption, and the quadratic form is decreasing in $\sigma$.
\end{proof}

\subsection{Conformal blow up and warping}
We now recall the two analytic steps of \cite{BZ2026} that convert a stability inequality on the regular part of a $\mu$-bubble into a pointwise lower bound for the weighted scalar curvature with respect to a complete metric.

\begin{proposition}(\cite[Proposition 3.12 and its proof]{BZ2026})\label{prop:BZ-blowup}
\textit{Let $m\ge3$, let $(X,\mathsf d,\mu)$ be a complete $m$ dimensional AM-PI space with regular part $(\mathcal R,g)$ and singular part $\mathcal S$, let $\lambda<\frac{1}{m+1}$, and let $\psi\in C^\infty(\mathcal R)$. Assume that $e^{\bar\gamma_m(\lambda)\psi}$ is an admissible weight, that $\bar\tau_m(\lambda)<\tau_{\mathrm{pack}}(\mathcal S)$, and that $\mathcal Q^{g,\psi}_{a_{m+1}(\lambda),\sigma}\ge0$ on $C^\infty_c(\mathcal R)$ for a constant $\sigma$. Then for every $\zeta>0$ there is a smooth function $w\geq1$ on $\mathcal R$ such that $\tilde g=w^{\bar\alpha_m(\lambda)}g$ is a complete metric on $\mathcal R$ and $\tilde\psi=\psi-\bar\beta_m(\lambda)\log w$ satisfies
$$
\mathcal Q^{\tilde g,\tilde\psi}_{a_{m+1}(\lambda),\sigma-\zeta}\ge0\qquad\text{on }C^\infty_c(\mathcal R).
$$
Moreover, every point of $\mathcal S$ has a neighborhood $U$ in $X$ such that $\tilde g\ge c\,\mathsf d(\cdot,\mathcal S)^{-2}g$ on $U\cap\mathcal R$ for some constant $c>0$.}
\end{proposition}

In \cite{BZ2026} the weighted operator is written as
$$
\Delta_g u-\bar\gamma_m(\lambda)\langle\nabla_gu,\nabla_gf\rangle=\omega^{-1}\operatorname{div}_g(\omega\nabla_gu),\qquad\omega=e^{-\bar\gamma_m(\lambda) f},
$$
which in our convention means $\omega=e^{\bar\gamma_m(\lambda)\psi}$. The function $w$ constructed in the proof of \cite[Proposition 3.12]{BZ2026} has the form $w=1+\varepsilon\Theta$ with $\Theta\geq0$, and $\Theta\ge c_K\operatorname{dist}(\cdot,\mathcal S)^{-2/\bar\alpha_m(\lambda)}$ near $\mathcal S$ by \cite[(3.29)]{BZ2026}. These properties give $w\ge1$ and the asserted lower bound for the metric (compare the proof of \cite[Lemma 3.10]{BZ2026}).

\begin{lemma}(\cite[Lemma 2.6]{BZ2026})\label{lem:warping}
\textit{Let $(Y,g)$ be a Riemannian manifold, let $\psi\in C^\infty(Y)$, $a>0$ and $\sigma\in\mathbb R$, and let $u>0$ be a smooth function with $L^{g,\psi}_{a,\sigma}u\ge0$. Then
$$
\mathcal R^{2-1/a}_g(\psi+\log u)\ge\sigma+\frac1a\left|(1-a)\nabla_g\psi+\nabla_g\log u\right|^2\ge\sigma.
$$
In particular, if $a=a_d$ with $d\ge2$, then $\mathcal R^{a_{d-1}}_g(\psi+\log u)\ge\sigma$.}
\end{lemma}
\begin{proof}
Put $v=\log u$ and $a'=2-1/a$. The inequality $L^{g,\psi}_{a,\sigma}u\ge0$ means
$$
\frac{\Delta_gu}{u}\le-\langle\nabla\psi,\nabla v\rangle+\frac12\left(\mathcal R^a_g(\psi)-\sigma\right).
$$
Since $\Delta_g v=\Delta_g u/u-|\nabla v|^2$ and $R_g-2\Delta_g\psi=\mathcal R^a_g(\psi)+a|\nabla\psi|^2$, we get
\begin{align*}
\mathcal R^{a'}_g(\psi+v)-\sigma&=R_g-2\Delta_g\psi-2\Delta_g v-a'|\nabla\psi+\nabla v|^2-\sigma\\
&\ge(a-a')|\nabla\psi|^2+2(1-a')\langle\nabla\psi,\nabla v\rangle+(2-a')|\nabla v|^2\\
&=\frac1a\left|(1-a)\nabla\psi+\nabla v\right|^2,
\end{align*}
where we used $a-a'=(1-a)^2/a$, $1-a'=(1-a)/a$ and $2-a'=1/a$. The last assertion follows from Lemma \ref{lem:parameter}(2).
\end{proof}

We now combine these steps for the $\mu$-bubbles of Proposition \ref{prop:bubble}. We use the following principle of Allegretto-Piepenbrink \cite{Allegretto,Piepenbrink}, see also \cite[Theorem 1]{FCS}: if $(Y,g)$ is a complete, connected, noncompact Riemannian manifold and $-\Delta_g+U$ is a Schr\"odinger operator whose quadratic form is nonnegative on $C^\infty_c(Y)$, then there is a smooth positive solution of $-\Delta_g v+Uv=0$. If $Y$ is closed, a first eigenfunction $v>0$ of $-\Delta_g+U$ satisfies $-\Delta_gv+Uv=\ell_1v$ with $\ell_1\ge0$. Since
\begin{equation}\label{eq:conjugation}
e^{\psi/2}L^{g,\psi}_{a,\sigma}\left(e^{-\psi/2}v\right)=-\Delta_gv+\left(\frac12\Delta_g\psi+\frac14|\nabla_g\psi|^2+\frac12\left(\mathcal R^a_g(\psi)-\sigma\right)\right)v,
\end{equation}
and the corresponding quadratic forms agree after integration by parts, $\mathcal Q^{g,\psi}_{a,\sigma}\ge0$ on a complete connected manifold implies that there is a smooth function $u>0$ with $L^{g,\psi}_{a,\sigma}u\ge0$.

\begin{lemma}\label{lem:blowup-warp}
\textit{Let $\lambda=-\frac12$. In the setting of Proposition \ref{prop:bubble}, let $\Sigma=\Sigma_{\mathrm{reg}}\sqcup D_\Sigma$ be the boundary of the warped $\mu$-bubble, and let $g_\Sigma$ and $\psi_\Sigma$ be the induced metric and the restriction of $\psi$ on $\Sigma_{\mathrm{reg}}$. Suppose that $\mathcal Q^{g_\Sigma,\psi_\Sigma}_{a_d,\sigma}\ge0$ on $C^\infty_c(\Sigma_{\mathrm{reg}})$ for a constant $\sigma$. Then for every $\zeta>0$ there are a complete Riemannian metric $g'$ and a smooth function $\psi'$ on $\Sigma_{\mathrm{reg}}$ such that
\begin{itemize}
\item[(1)] $g'\ge g_\Sigma$,
\item[(2)] every $p\in D_\Sigma$ has a neighborhood $U_p$ in $P$ and a constant $c_p>0$ such that $g'\ge c_p\,d_g(\cdot,p)^{-2}g_\Sigma$ on $U_p\cap\Sigma_{\mathrm{reg}}$,
\item[(3)] $\mathcal R^{a_{d-1}}_{g'}(\psi')\ge\sigma-\zeta$ on $\Sigma_{\mathrm{reg}}$.
\end{itemize}
If $D_\Sigma=\varnothing$, one may take $g'=g_\Sigma$ and $\zeta=0$.}
\end{lemma}
\begin{proof}
If $D_\Sigma=\varnothing$, then $\Sigma=\Sigma_{\mathrm{reg}}$ is a smooth hypersurface which is closed in $P$, and by \eqref{eq:band-assumption} it is either compact or a closed subset of a complete manifold. Thus $(\Sigma,g_\Sigma)$ is complete. Applying the principle above on every component of $\Sigma$, we obtain a smooth $u>0$ with $L^{g_\Sigma,\psi_\Sigma}_{a_d,\sigma}u\ge0$, and Lemma \ref{lem:warping} shows that $g'=g_\Sigma$ and $\psi'=\psi_\Sigma+\log u$ have the required properties.

Suppose $D_\Sigma\neq\varnothing$. Then $d\ge8$. Put $\mathsf d_\Sigma=d_g|_\Sigma$ and $\mu_\Sigma=\mathcal H^{d-1}\llcorner\Sigma$. By Proposition \ref{prop:bubble}(1) and (4), for every component $C$ of $P$, the space $(\Sigma\cap C,\mathsf d_\Sigma,\mu_\Sigma)$ is a complete AM-PI space of dimension $m=d-1\ge7$ with $\tau_{\mathrm{pack}}(D_\Sigma)\ge3>\bar\tau_{m}$, by Lemma \ref{lem:parameter}(4). Every compact subset of $\Sigma$ is a compact subset of $P$, so $\psi_\Sigma$, being the restriction of the smooth function $\psi$, is bounded on it. Hence $e^{\bar\gamma_m\psi_\Sigma}$ is an admissible weight. Since $a_d=a_{m+1}$, Proposition \ref{prop:BZ-blowup}, applied on each of these spaces, gives $w\ge1$, a complete metric $\tilde g=w^{\bar\alpha_m}g_\Sigma$ and $\tilde\psi=\psi_\Sigma-\bar\beta_m\log w$ with $\mathcal Q^{\tilde g,\tilde\psi}_{a_d,\sigma-\zeta}\ge0$. As above, the principle of Allegretto-Piepenbrink applied to each component of $(\Sigma_{\mathrm{reg}},\tilde g)$ gives $u>0$ with $L^{\tilde g,\tilde\psi}_{a_d,\sigma-\zeta}u\ge0$, and Lemma \ref{lem:warping} shows that $g'=\tilde g$ and $\psi'=\tilde\psi+\log u$ satisfy (3). Property (1) follows from $w\ge1$ and $\bar\alpha_m>0$, and (2) follows from the last assertion of Proposition \ref{prop:BZ-blowup}, since $\operatorname{dist}(x,D_\Sigma)\le d_g(x,p)$ for $p\in D_\Sigma$.
\end{proof}

    \section{\texorpdfstring{Closed sphere-torus rigidity under nonnegative Ricci curvature}{Closed sphere-torus rigidity under nonnegative Ricci curvature}}\label{sec-sup}
The last step of the proof of Theorem \ref{main} needs a rigidity result for closed manifolds mapping to $\mathbb S^k\times\mathbb T^m$. When we apply closed rigidity, the Cheeger-Gromoll splitting has already produced a closed manifold with nonnegative Ricci curvature and constant scalar curvature $k(k-1)$. Under these assumptions, Theorem \ref{thm:WWXZ-Llarull} suffices.

For $n\ge3$, the trace norm bound needed in \cite[Lemma~2.7]{WWXZ2026} follows directly under the hypotheses of Theorem~\ref{thm:WWXZ-Llarull}, without the dichotomy used there (see \cite[Lemma~2.1]{WX2024} and \cite[Theorem~2.3]{WWX2026}). Set $\kappa_g=n(n-1)(\|df\|_{\mathrm{tr},g}/n)^2$, where $\|df\|_{\mathrm{tr},g}$ is the sum of the singular values of $df$. Since $f$ is $1$-Lipschitz, $\kappa_g\le n(n-1)\le R_g$. If $R_g\equiv\kappa_g$, all singular values are one and $f$ is a Riemannian covering of $\mathbb S^n$, hence an isometry. Otherwise choose a smooth, nonzero function $0\le\vartheta\le R_g-\kappa_g$. Let $v>0$ be a first eigenfunction of $-\frac{4(n-1)}{n-2}\Delta_g+\vartheta$, with eigenvalue $\ell_1>0$. For $g_1=v^{4/(n-2)}g$, the conformal scalar curvature formula gives
$$
R_{g_1}-\kappa_{g_1}\ge\ell_1v^{-4/(n-2)}>0.
$$
This gap is uniform by compactness and suffices for the boundary estimates \cite[(2.34), (2.36)]{WWXZ2026}.

The following lemma asserts that a non-zero degree map induces an injective homomorphism in cohomology.
\begin{lemma}\label{inj}
   \textit{ Suppose $\Psi:P^n\to Q^n$ is a continuous map of non-zero degree between two closed, connected, oriented $n$ dimensional manifolds $P$ and $Q$. Then the induced cohomological homomorphism
$$
\Psi^*:H^r(Q;\mathbb{R})\longrightarrow H^r(P;\mathbb{R})
$$
is injective for $0\le r\le n$.}
\end{lemma}
\begin{proof}
This follows from Poincaré duality and the naturality of the cup product.
\end{proof}
We also need the following elementary linear algebra observation.
\begin{lemma}\label{linear}
\textit{Let $V$, $W$, $Z$ be finite dimensional
Euclidean spaces. Suppose
$L:V\oplus W\to Z$
is a linear map satisfying
$$
|Lv|\le|v|~~\text{for all}\,\,v\in V\oplus W,
$$
and that the restriction  $L|_V:V\to Z$ is an isometric isomorphism. Then $L|_W\equiv0$.}
\end{lemma}
\begin{proof}
 Let $v\in V$, $w\in W$ and $0\neq t\in\mathbb{R}$. Since $L$ is non-expansive,
 \begin{align*}
 |Lv|^2+2t\langle Lv,Lw\rangle+t^2|Lw|^2=&|L(v+tw)|^2\\
 \le&|v+tw|^2\\
 =&|v|^2+t^2|w|^2.
 \end{align*}
 Using $|Lv|=|v|$, we obtain
$$
2t\langle Lv,Lw\rangle\le t^2\left(|w|^2-|Lw|^2\right).
$$
First divide by $t>0$ and let $t\to 0^{+}$. Then repeat with $t<0$. It follows that $\langle Lv,Lw\rangle=0$ for any $v\in V,\,w\in W$. Because $L(V)=Z$, $L|_W\equiv0$.
\end{proof}

We now state the closed rigidity result used in the proof of Theorem \ref{main}.
\begin{proposition}\label{prop:closed-rigidity}
\textit{
Let $k\ge2$ and $m\ge0$, and let $(N^{k+m},g)$ be a closed, connected, oriented Riemannian manifold with
$$
\operatorname{Ric}_N\ge0,\qquad R_N\equiv k(k-1).
$$
Suppose that
$$
A\coloneqq(A_\mathbb{S},A_\mathbb{T}):(N,g) \longrightarrow\mathbb{S}^{k}\times\mathbb{T}^{m}
$$
is a smooth map of non-zero degree such that
$A_\mathbb{S}:(N,g)\to(\mathbb{S}^{k},g_{\mathbb{S}^{k}})$ is
     $1$-Lipschitz. Then there exists a full lattice $\Lambda\subset\mathbb{R}^{m}$ such that $(N,g)$ is isometric to
$(\mathbb{S}^{k}\times\mathbb{T}^{m}_\Lambda,\,g_{\mathbb{S}^{k}}+g_{\mathrm{flat}})$, where $\mathbb{T}^{m}_\Lambda\coloneqq\mathbb{R}^{m}/\Lambda$ is a flat torus. Moreover, under this isometry $A_\mathbb S(x,y)=J(x)$ for an isometry $J$ of $\mathbb S^k$.
}
    \end{proposition}
\begin{proof}
If $m=0$, this is Theorem \ref{thm:WWXZ-Llarull}. Assume $m\ge1$. The metric of the target torus plays no role in the hypotheses, so we may identify $\mathbb T^m$ with $\mathbb R^m/\mathbb Z^m$ by a linear isomorphism. Let $\theta_1,\ldots,\theta_m$ be the corresponding coordinates.

\emph{Step 1: parallel $1$-forms.} Since $k\geq2$, $H^1(\mathbb{S}^{k}\times\mathbb{T}^{m};\mathbb{R})\cong\mathbb{R}^{m}$ is spanned by the classes of $d\theta_1,\ldots,d\theta_m$. By Lemma \ref{inj}, $A^*:H^1(\mathbb{S}^{k}\times\mathbb{T}^{m};\mathbb{R})\to H^1(N;\mathbb R)$ is injective. Let $\omega_i$ be the harmonic representative of $A^*[d\theta_i]$. Since $\operatorname{Ric}_N\ge0$ and $N$ is closed, the Bochner formula shows that every $\omega_i$ is parallel. The forms $\omega_1,\ldots,\omega_m$ are linearly independent as cohomology classes, hence as forms. If $\sum_ic_i\omega_i$ vanishes at one point, then this parallel form vanishes identically, so $c=0$. Thus $\omega_1,\ldots,\omega_m$ are linearly independent at every point. The dual vector fields $X_i=\omega_i^\sharp$ are parallel, hence Killing, and $[X_i,X_j]=\nabla_{X_i}X_j-\nabla_{X_j}X_i=0$. The matrix $G_{ij}=\langle X_i,X_j\rangle$ is constant and positive definite. Put $\mathscr V=\operatorname{span}\{X_i\}$ and $\mathscr H=\mathscr V^\perp$. Both distributions are parallel, so $\mathscr H$ is integrable and its leaves are totally geodesic.

\emph{Step 2: a fibration over the torus.} Each $\omega_i$ represents the integral class $A^*[d\theta_i]=A_{\mathbb T}^*[d\theta_i]$, so all its periods are integers. Fix $x_0\in N$ and define
$$
B(x)=A_{\mathbb T}(x_0)+\left(\int_{x_0}^x\omega_1,\ldots,\int_{x_0}^x\omega_m\right)\mod\mathbb Z^m.
$$
This is a well defined smooth map $N\to\mathbb T^m$ with $dB=(\omega_1,\ldots,\omega_m)$, which has full rank everywhere. Hence $B$ is a submersion of a closed manifold, thus a fiber bundle, and $\ker dB=\mathscr H$. For every $i$ we have $B^*[d\theta_i]=[\omega_i]=A_{\mathbb T}^*[d\theta_i]$ in $H^1(N;\mathbb R)$. Since $H^1(N;\mathbb Z)=\operatorname{Hom}(H_1(N),\mathbb Z)$ is torsion free, the map $H^1(N;\mathbb Z)\to H^1(N;\mathbb R)$ is injective, so $B^*=A_{\mathbb T}^*$ on $H^1(\mathbb T^m;\mathbb Z)$. Since $\mathbb T^m$ is a $K(\mathbb Z^m,1)$, the homotopy class of a map $N\to\mathbb T^m$ is determined by its pullback on $H^1(\,\cdot\,;\mathbb Z)$, and therefore $B\simeq A_{\mathbb T}$. Consequently $A'\coloneqq(A_{\mathbb S},B)$ is homotopic to $A$, and $\deg A'=\deg A\neq0$.

\emph{Step 3: degree on a fiber.} Let $(p,t)$ be a regular value of $A'$ and $P_t=B^{-1}(t)$. This is a closed $k$ dimensional submanifold whose components are leaves of $\mathscr H$. At $x\in A'^{-1}(p,t)$ we have $T_xN=T_xP_t\oplus\mathscr V_x$, $dB|_{T_xP_t}=0$, and $dB|_{\mathscr V_x}$ is an isomorphism. Orient $\mathscr V$ so that $dB|_{\mathscr V}:\mathscr V\to\mathbb R^m$ is orientation preserving, and orient $P_t$ so that $T_xP_t\oplus\mathscr V_x$ carries the orientation of $T_xN$. With respect to this splitting the matrix of $dA'$ is block triangular, and $\det dA'=\det(dA_{\mathbb S}|_{T_xP_t})\cdot\det(dB|_{\mathscr V_x})$, where the second factor is positive. Hence
$$
0\neq\deg A'=\deg\left(A_{\mathbb S}|_{P_t}\right)=\sum_{P}\deg\left(A_{\mathbb S}|_P\right),
$$
where the sum is over the components $P$ of $P_t$. Hence $\deg(A_{\mathbb S}|_P)\ne0$ for some leaf $P$.

\emph{Step 4: a local product and the first sphere.} Let $\varphi_s=\varphi^{X_1}_{s_1}\circ\cdots\circ\varphi^{X_m}_{s_m}$, $s\in\mathbb R^m$, where $\varphi^{X_i}$ is the flow of $X_i$. This is an isometric action of $\mathbb R^m$ on $N$ which preserves every $X_j$, hence also $\mathscr V$ and $\mathscr H$. Define
$$
\Phi:(P\times\mathbb R^m,g_P+g_{\mathrm E})\longrightarrow(N,g),\qquad \Phi(x,s)=\varphi_s(x),
$$
where $g_{\mathrm E}=\sum_{i,j}G_{ij}ds_ids_j$. Since $d\Phi(v,e_j)=d\varphi_s(v)+X_j$ with $d\varphi_s(v)\in\mathscr H$ and $X_j\in\mathscr V$, the map $\Phi$ is a local isometry. Its domain is complete, so $\Phi$ is a Riemannian covering \cite[Chapter 7, Lemma 3.3]{docarmo}. Since the factor $(\mathbb R^m,g_{\mathrm E})$ is flat, $R_P\equiv R_N\equiv k(k-1)$. The map $A_{\mathbb S}|_P:P\to\mathbb S^k$ is $1$-Lipschitz with respect to the induced metric and has non-zero degree, so it is an isometry by Theorem \ref{thm:WWXZ-Llarull}. Hence $P$ is a unit round sphere, $P\times\mathbb R^m\cong\mathbb S^k\times\mathbb R^m$ is simply connected, and $\Phi$ is the universal covering.

\emph{Step 5: every spherical fiber.} Identify $P$ with $\mathbb S^k$ by means of the isometry $A_{\mathbb S}|_P$, and lift $A_{\mathbb S}$ by setting
$$
\tilde A_{\mathbb{S}}=A_{\mathbb{S}}\circ\Phi:\mathbb{S}^{k}\times\mathbb{R}^{m}\longrightarrow\mathbb{S}^k.
$$
For every $s$, the map $\tilde A_{\mathbb S}(\,\cdot\,,s)=A_{\mathbb S}\circ\varphi_s|_P$ is $1$-Lipschitz and is homotopic to $A_{\mathbb S}|_P$, so its degree is $\pm1$. By Theorem \ref{thm:WWXZ-Llarull},
\begin{equation}\label{A_s}
\tilde A_{\mathbb{S}}(\,\cdot\,,s):\mathbb{S}^k\longrightarrow\mathbb{S}^k
\end{equation}
is an isometry for every $s\in\mathbb{R}^{m}$. Fix a point $(x,s)\in\mathbb{S}^{k}\times\mathbb{R}^{m}$. Since $\tilde A_{\mathbb{S}}$ is 1-Lipschitz, its differential is non-expansive, and by \eqref{A_s} its restriction to $T_x\mathbb S^k$ is an isometric isomorphism. By Lemma \ref{linear}, $d\tilde A_{\mathbb{S}}$ vanishes on $\mathbb R^m$, which shows that $\tilde A_{\mathbb{S}}(\,\cdot\,,s)$ is independent of $s$. Therefore, there is a fixed isometry
$$
J\coloneqq\tilde A_{\mathbb{S}}(\,\cdot\,,s):\mathbb{S}^{k}\longrightarrow\mathbb{S}^{k}.
$$

\emph{Step 6: the deck transformation group.} Let $\Gamma$ be the deck transformation group of $\Phi$. Since every isometry of $\mathbb{S}^{k}\times\mathbb{R}^{m}$ preserves its spherical and Euclidean parallel distributions, every $\gamma\in\Gamma$ has the form
$$
\gamma(x,s)=(\rho_\gamma(x),L_\gamma s+a_\gamma),
$$
where $\rho_\gamma\in\mathrm{Isom}(\mathbb{S}^{k})$, $L_\gamma$ is orthogonal with respect to $g_{\mathrm E}$, and $a_\gamma\in\mathbb{R}^{m}$. The vector fields $X_j$ are defined on $N$ and their lifts are $\partial_{s_j}$. Deck transformations preserve these lifts, so $L_\gamma=I$. Since $\tilde A_{\mathbb{S}}=A_{\mathbb{S}}\circ\Phi$ is $\Gamma$-invariant, for any $\gamma\in\Gamma$ we have
$$
J(\rho_\gamma(x))=\tilde A_{\mathbb{S}}(\gamma(x,s))=\tilde A_{\mathbb{S}}(x,s)=J(x).
$$
Consequently, $\rho_\gamma=id_{\mathbb{S}^k}$ and every deck transformation takes the form
$$
\gamma(x,s)=(x,s+a_\gamma).
$$
The homomorphism $a:\Gamma\to\mathbb{R}^{m}$ is injective, and proper discontinuity implies that $\Lambda'\coloneqq a(\Gamma)$ is discrete. Thus
$$
N\cong(\mathbb{S}^{k}\times\mathbb{R}^{m})/\Lambda'\cong\mathbb{S}^{k}\times(\mathbb{R}^{m}/\Lambda').
$$
The compactness of $N$ forces $\Lambda'$ to have full rank, and by \cite[Theorem 5.3.2]{HyperMR1299730}, $\Lambda'$ is a full lattice. Finally, a linear isometry $(\mathbb R^m,g_{\mathrm E})\to(\mathbb R^m,g_{\mathrm{eucl}})$ carries $\Lambda'$ to a full lattice $\Lambda$, and $(N,g)$ is isometric to $(\mathbb{S}^{k}\times\mathbb{T}^{m}_\Lambda,\,g_{\mathbb{S}^{k}}+g_{\mathrm{flat}})$, with $A_{\mathbb S}(x,y)=J(x)$.
\end{proof}

\section{The terminal obstruction}\label{sec:terminal}
After all torus directions have been removed by the descent of Section \ref{sec:descent}, we are left with a $k$ dimensional manifold, obtained from a compact set by removing a closed singular set, which carries a complete metric, a $1$-Lipschitz map of non-zero degree to $\mathbb S^k$, and a constant gap for the weighted scalar curvature. In this section we show that such a configuration does not exist. The argument follows the proof of \cite[Lemma 2.7]{WWXZ2026}, where the spherical rigidity theorem \cite[Theorem 1.2]{WWXZ2026} is deduced from the comparison theorem \cite[Theorem 2.1]{WWXZ2026} on the half cylinder $(-\infty,0]\times M$, which is conformal to the punctured cone over $M$. We first modify the map so that it is constant outside a compact set (Lemma \ref{lem:terminal-map}), and then apply the scalar-mean comparison theorem of \cite{WWXZ2026} to an explicit conformal deformation of the half cylinder $(-\infty,0]\times\Sigma$ (Lemma \ref{lem:half-cylinder}). Since the curvature gap is a constant, the conformal factor can be written down explicitly, instead of being obtained from a Robin eigenvalue problem as in \cite[Claim 2.8]{WWXZ2026}.

\subsection{A capping map}
\begin{lemma}\label{lem:cap}
\textit{Let $k\ge2$, let $p_+\in\mathbb S^k$, $p_-=-p_+$, and $r=d_{\mathbb S^k}(p_+,\cdot)$. For all $0<\theta_1<\theta_2<\pi/4$ there is a smooth map $\Upsilon:\mathbb S^k\to\mathbb S^k$ such that
\begin{itemize}
\item[(a)] $\Upsilon^{-1}(p_+)=\{r\le\theta_1\}$ and $\Upsilon^{-1}(p_-)=\{r\ge\pi-\theta_1\}$,
\item[(b)] the pointwise operator norm satisfies $\|d\Upsilon\|\le\frac{\pi}{\pi-2\theta_2}$,
\item[(c)] $\deg\Upsilon=1$, and $\Upsilon$ maps $\{\theta_1<r<\pi-\theta_1\}$ diffeomorphically and orientation preservingly onto $\mathbb S^k\setminus\{p_\pm\}$.
\end{itemize}}
\end{lemma}
\begin{proof}
Choose $\theta'\in(\theta_1,\theta_2)$ so close to $\theta_1$ that $2s_0(\theta'-\theta_1)\le\theta_1$, where $s_0=\frac{\pi}{\pi-2\theta'}$. The piecewise linear function $\xi_0(r)=\max\{0,s_0(r-\theta')\}$ on $[0,\pi/2]$ satisfies $\xi_0(\pi/2)=\pi/2$. Mollify $\xi_0$ with a symmetric smooth kernel which is positive exactly on $(-\varrho,\varrho)$, where $\varrho=\theta'-\theta_1$. The resulting smooth function $\xi$ vanishes exactly on $[0,\theta_1]$, satisfies $\xi'>0$ on $(\theta_1,\pi/2]$ and $0\le\xi'\le s_0$, coincides with $\xi_0$ on $[\theta'+\varrho,\pi/2]$, and satisfies $\xi(r)\le r$ on $[0,\pi/2]$. The last property holds because $\xi_0(r)\le r$ on $[\theta',\pi/2]$, while on $[\theta_1,\theta'+\varrho]$ we have $\xi(r)\le s_0(r-\theta_1)\le r$, by the choice of $\theta'$. Extend $\xi$ to $[0,\pi]$ by $\xi(\pi-r)=\pi-\xi(r)$. Since $\xi$ is linear near $\pi/2$, the extension is smooth. Define
$$
\Upsilon(\exp_{p_+}(rv))=\exp_{p_+}(\xi(r)v),\qquad v\in T_{p_+}\mathbb S^k,\ |v|=1.
$$
Then $\Upsilon$ is constant near the two poles and smooth elsewhere. At a point with polar coordinates $(r,v)$, $d\Upsilon$ maps $\partial_r$ to $\xi'(r)\partial_r$ and multiplies unit vectors orthogonal to $\partial_r$ by $\sin\xi(r)/\sin r$. The inequality $\xi(r)\le r\le\pi/2$, together with symmetry, gives $\sin\xi\le\sin r$, so $\|d\Upsilon\|\le s_0\le\frac{\pi}{\pi-2\theta_2}$. Properties (a) and (c) follow from the construction, and $\Upsilon$ is homotopic to the identity through $\xi_t=(1-t)r+t\xi$, so its degree is one.
\end{proof}

\subsection{Modifying the terminal map}
\begin{lemma}\label{lem:terminal-map}
\textit{Let $k\ge2$, let $(V,g_V)$ be a Riemannian manifold, let $X\subset V$ be compact and $D\subset X$ closed, and assume that $P\coloneqq X\setminus D$ is an oriented $k$ dimensional smooth embedded submanifold of $V$ without boundary. Let $g_P$ be a complete Riemannian metric on $P$ satisfying the blow up condition
\begin{itemize}
\item[(B)] for every $p\in D$ there are $c_p,r_p>0$ such that $g_P\ge c_p\,d_{g_V}(\cdot,p)^{-2}g_V$ on $\{x\in P:d_{g_V}(x,p)<r_p\}$.
\end{itemize}
Let $\eta:V\to\mathbb S^k$ be a smooth map such that $\eta|_P$ is $1$-Lipschitz with respect to $g_P$, $\mathbb S^k\setminus\eta(D)$ is connected and dense, and $\eta|_P$ is proper over $\mathbb S^k\setminus\eta(D)$ with degree $c\ne0$. Then there are $p_+\in\mathbb S^k$ and a smooth map $\eta_1:P\to\mathbb S^k$ which is $1$-Lipschitz with respect to $g_P$, is identically $p_+$ outside a compact subset of $P$, and satisfies $\deg\eta_1=c$. Here $\deg\eta_1$ is the degree of the proper map $\eta_1:P\setminus\eta_1^{-1}(p_+)\to\mathbb S^k\setminus\{p_+\}$.}
\end{lemma}
\begin{proof}
\emph{Step 1: $\|d\eta\|_{g_P}\to0$ at infinity.} Put $C_\eta=\sup_X|d\eta|_{g_V}<\infty$. We claim that for every $\varepsilon_1>0$ there is a compact set $K_1\subset P$ with $\|d\eta\|_{g_P}\le\varepsilon_1$ on $P\setminus K_1$. Otherwise there are $x_j\in P$ leaving every compact subset of $P$ with $\|d\eta\|_{g_P}(x_j)>\varepsilon_1$. Since $X$ is compact, after passing to a subsequence, we have $x_j\to x_\infty\in X$. If $x_\infty\in P$, then, since $P$ is open in $X$ and an embedded submanifold, $x_\infty$ has a compact neighborhood in $P$ which contains all but finitely many $x_j$, a contradiction. Hence $p=x_\infty\in D$. By (B), $|v|_{g_V}\le c_p^{-1/2}d_{g_V}(x_j,p)|v|_{g_P}$ for $v\in T_{x_j}P$ and $j$ large, so $\|d\eta\|_{g_P}(x_j)\le C_\eta c_p^{-1/2}d_{g_V}(x_j,p)\to0$, a contradiction.

\emph{Step 2: the point $p_-$.} Since $\eta(D)$ is compact, $\mathbb S^k\setminus\eta(D)$ is open, and it is dense by assumption. By Sard's theorem we may choose $p_-\in\mathbb S^k\setminus\eta(D)$ which is a regular value of $\eta|_P$, and $\theta_0\in(0,\pi/2)$ with $\bar B_{\theta_0}(p_-)\cap\eta(D)=\varnothing$. Put $p_+=-p_-$. By properness, $K_2=(\eta|_P)^{-1}(\bar B_{\theta_0}(p_-))$ is compact, and $\eta^{-1}(p_-)\cap P$ is finite.

\emph{Step 3: a homotopy.} Let $\varsigma:\mathbb S^k\setminus\{p_-\}\to\mathbb R^k$ be the stereographic projection from $p_-$, so that $\varsigma(p_+)=0$, and put $\Xi(t,y)=\varsigma^{-1}((1-t)\varsigma(y))$ for $t\in[0,1]$. Then $\Xi$ is smooth, $\Xi(0,\cdot)=\operatorname{id}$, $\Xi(1,\cdot)\equiv p_+$, and $\Xi(t,y)\neq p_-$. On the compact set $[0,1]\times(\mathbb S^k\setminus B_{\theta_0}(p_-))$ we have $|\partial_t\Xi|\le C_\Xi$ and $\|d_y\Xi\|\le C_\Xi$ for some $C_\Xi\ge1$.

\emph{Step 4: cutting off.} Take $\varepsilon_1=1/(2C_\Xi)$, let $K_1$ be given by Step 1, and put $K_0=K_1\cup K_2$. Although $P$ need not be connected, the compact set $K_0$ meets only finitely many components of $P$, and we denote their union by $P_0$. On $P\setminus P_0$ we simply put $\eta_1\equiv p_+$. These components do not meet $K_2\supset\eta^{-1}(p_-)\cap P$, so the degree computation in Step 6 is not affected. Let $\rho=d_{g_P}(\cdot,K_0)$, which is finite on $P_0$. Since $g_P$ is complete, there is a smooth function $\tilde\rho$ on $P_0$ with $|\tilde\rho-\rho|\le1$ and $|d\tilde\rho|_{g_P}\le2$, see \cite{GW}. Let $\beta:\mathbb R\to[0,1]$ be smooth and nondecreasing, with $\beta=0$ on $(-\infty,0]$, $\beta=1$ on $[1,\infty)$ and $|\beta'|\le2$. Put $\ell=8C_\Xi$ and $\chi=\beta((\tilde\rho-2)/\ell)$. Then $\chi=0$ on the neighborhood $\{\rho<1\}$ of $K_0$. On $\{\chi>0\}\subset\{\rho>1\}$ we have $\|d\eta\|_{g_P}\le\varepsilon_1$ and $\eta\notin\bar B_{\theta_0}(p_-)$. The set $\{\chi<1\}\subset\{\rho<3+\ell\}$ is bounded, hence precompact in $P_0$ by the Hopf-Rinow theorem. Define $\eta_1=\eta$ on $\{\rho<1\}$ and $\eta_1=\Xi(\chi,\eta)$ on $\{\rho>1/2\}$. The two definitions agree on the overlap, where $\chi=0$. Thus $\eta_1$ is smooth and identically $p_+$ outside the compact set $\overline{\{\chi<1\}}$.

\emph{Step 5: the Lipschitz constant.} On $\{\chi=0\}$ we have $\eta_1=\eta$, so $\|d\eta_1\|_{g_P}\le1$. On $\{\rho>1/2\}\cap\{\chi>0\}$,
$$
\|d\eta_1\|_{g_P}\le|\partial_t\Xi|\,|d\chi|_{g_P}+\|d_y\Xi\|\,\|d\eta\|_{g_P}\le C_\Xi\cdot\frac{2\cdot2}{\ell}+C_\Xi\varepsilon_1=\frac12+\frac12=1.
$$

\emph{Step 6: the degree.} On $\{\chi>0\}$ we have $\eta_1=\Xi(\chi,\eta)\ne p_-$, and on $\{\chi=0\}\setminus K_0$ we have $\eta_1=\eta\ne p_-$. Hence $\eta_1^{-1}(p_-)=\eta^{-1}(p_-)\cap P\subset K_2$, and $\eta_1=\eta$ near these points. Thus $p_-$ is a regular value of $\eta_1$ with the same local degrees as for $\eta$, and $\deg\eta_1=c$.
\end{proof}

\subsection{A weighted half cylinder}
We use the following scalar-mean comparison theorem of Wang-Wang-Xie-Zhu, expressed in the notation of \eqref{eq:Ra}.
\begin{theorem}(\cite[Theorem 3.3]{WWXZ2026})\label{thm:WWXZ-mean}
\textit{Let $(Z^n,g)$, $n\ge3$, be a complete, oriented Riemannian manifold with boundary $\partial Z$, which may be noncompact, and let $\nu$ be the outward unit normal. Let $p_\pm$ be antipodal points of $\mathbb S^{n-1}$, and let $\eta:\partial Z\to\mathbb S^{n-1}$ be a smooth map such that every connected component of $B_\pm\coloneqq\eta^{-1}(p_\pm)$ has nonempty interior in $\partial Z$ and $\partial Z\setminus B_+$ is precompact. Suppose that there are $\psi\in C^\infty(Z)$, $\delta>0$ and $\mu>\frac{n-1}{n}$ such that
$$
\mathcal R^\mu_g(\psi)\ge\delta\quad\text{on }Z,\qquad H_{\partial Z}+\nu(\psi)-\|d\eta\|_{\mathrm{tr}}\ge\delta\quad\text{on }\partial Z.
$$
Then $\deg\eta=0$.}
\end{theorem}
Here $H_{\partial Z}$ is the mean curvature with respect to $\nu$, normalized so that the unit sphere in $\mathbb R^n$ has mean curvature $n-1$. We write $\|d\eta\|_{\mathrm{tr}}$ for the sum of the singular values of $d\eta$ with respect to the induced metric and $g_{\mathbb S^{n-1}}$, and $\deg\eta$ for the degree of the proper map $\partial Z\setminus B_+\to\mathbb S^{n-1}\setminus\{p_+\}$.

\begin{lemma}\label{lem:half-cylinder}
\textit{Let $k\ge2$, let $(\Sigma^k,g_\Sigma)$ be a complete, oriented Riemannian manifold without boundary, possibly noncompact, and let $\psi\in C^\infty(\Sigma)$, $\mu>\frac{k}{k+1}$ and $L_0>0$. Let $\eta:\Sigma\to\mathbb S^k$ be a smooth map with $\|d\eta\|_{g_\Sigma}\le L_0$ everywhere, which is identically equal to a point $p_+$ outside a compact set. If there is a constant $\sigma$ such that
$$
\mathcal R^\mu_{g_\Sigma}(\psi)\ge\sigma>k(k-1)L_0^2\qquad\text{on }\Sigma,
$$
then $\deg\eta=0$. Here $\deg\eta$ is the degree of the proper map $\Sigma\setminus\eta^{-1}(p_+)\to\mathbb S^k\setminus\{p_+\}$.}
\end{lemma}
\begin{proof}
Suppose $\deg\eta\neq0$, and let $\eta\equiv p_+$ outside the compact set $K$. Only finitely many components of $\Sigma$ meet $K$. On the others $\eta\equiv p_+$, and they do not contribute to the degree. Hence some component of $\Sigma$ has non-zero degree, and since it satisfies all assumptions, we may assume that $\Sigma$ is connected.

\emph{Step 0: choice of constants.} First choose $\varepsilon_*>0$ with $\sigma>k(k-1)L_*^2$, where $L_*=(1+\varepsilon_*)L_0$. Then choose $\theta_2\in(0,\pi/8)$ with $\frac{\pi}{\pi-2\theta_2}\le1+\varepsilon_*$, and $\alpha>0$ with
\begin{equation}\label{eq:alpha}
\frac{k-1}{2}L_*<\alpha,\qquad\frac{4k}{k-1}\alpha^2<\sigma.
\end{equation}
Such $\alpha$ exists because $\frac{4k}{k-1}\bigl(\frac{k-1}2L_*\bigr)^2=k(k-1)L_*^2<\sigma$. A small constant $\varepsilon>0$ is chosen in Step 5.

\emph{Step 1: capping.} Let $p_-=-p_+$ and $r=d_{\mathbb S^k}(p_+,\cdot)$. On the open set $\eta^{-1}(\mathbb S^k\setminus\{p_\pm\})$ the function $r\circ\eta$ is smooth. By Sard's theorem choose $\theta_1\in(0,\theta_2)$ such that $\theta_1$ and $\pi-\theta_1$ are regular values of $r\circ\eta$ on this set. Let $\Upsilon$ be the map of Lemma \ref{lem:cap} and $\eta'=\Upsilon\circ\eta$. Then $\eta'$ is smooth, $\|d\eta'\|_{g_\Sigma}\le L_*$, $\eta'\equiv p_+$ outside a compact set, and $\deg\eta'=\deg\eta\neq0$. For the last claim, let $q\in\mathbb S^k\setminus\{p_\pm\}$ be a regular value of $\eta'$. By Lemma \ref{lem:cap}(c), $\Upsilon^{-1}(q)=\{q'\}$ is a single point near which $\Upsilon$ is an orientation preserving diffeomorphism, so $\eta'^{-1}(q)=\eta^{-1}(q')$, the point $q'$ is a regular value of $\eta$, and the signed counts agree. Moreover,
$$
B_+=\eta'^{-1}(p_+)=\{r\circ\eta\le\theta_1\},\qquad B_-=\eta'^{-1}(p_-)=\{r\circ\eta\ge\pi-\theta_1\}.
$$
Since the two cutoff levels are regular, $B_+$ and $B_-$ are $k$ dimensional manifolds with boundary. Every component therefore has nonempty interior in $\Sigma$. Moreover, $\Sigma\setminus B_+\subset\{\eta\ne p_+\}$ is precompact.

\emph{Step 2: the half cylinder and conformal formulas.} Let $Z=(-\infty,0]\times\Sigma$, $\bar g=dt^2+g_\Sigma$, and $\bar\psi(t,x)=\psi(x)$. This is a complete manifold with boundary $\partial Z=\{0\}\times\Sigma\cong\Sigma$, outward normal $\nu=\partial_t$, $H_{\partial Z}(\bar g)=0$, $\nu(\bar\psi)=0$, and $\mathcal R^\mu_{\bar g}(\bar\psi)=\mathcal R^\mu_{g_\Sigma}(\psi)$. Its dimension is $n'=k+1\ge3$. For a positive function $u$ and $g_u=u^{4/(n'-2)}\bar g$ we have
\begin{align}
\mathcal R^\mu_{g_u}(\bar\psi)&=u^{-\frac{n'+2}{n'-2}}\left(-\frac{4(n'-1)}{n'-2}\Delta_{\bar g}u-4\langle\nabla\bar\psi,\nabla u\rangle_{\bar g}+\mathcal R^\mu_{\bar g}(\bar\psi)\,u\right),\label{eq:conf-scalar}\\
H_{\partial Z}(g_u)+\nu_u(\bar\psi)&=u^{-\frac2{n'-2}}\left(H_{\partial Z}(\bar g)+\nu(\bar\psi)+\frac{2(n'-1)}{n'-2}\frac{\partial_\nu u}{u}\right),\label{eq:conf-mean}\\
\|d\eta'\|_{\mathrm{tr},g_u}&=u^{-\frac2{n'-2}}\|d\eta'\|_{\mathrm{tr},g_\Sigma}.\label{eq:conf-trace}
\end{align}
The terms not involving $\bar\psi$ are the classical conformal formulas, see \cite[(2.2)--(2.4)]{WWXZ2026}. For the remaining terms write $g_u=e^{2\varphi}\bar g$ with $\varphi=\frac{2}{n'-2}\log u$. Then
$$
\Delta_{g_u}\bar\psi=e^{-2\varphi}\left(\Delta_{\bar g}\bar\psi+(n'-2)\langle\nabla\varphi,\nabla\bar\psi\rangle\right)=u^{-\frac4{n'-2}}\left(\Delta_{\bar g}\bar\psi+2u^{-1}\langle\nabla u,\nabla\bar\psi\rangle\right),
$$
$|\nabla\bar\psi|^2_{g_u}=u^{-\frac4{n'-2}}|\nabla\bar\psi|^2_{\bar g}$ and $\nu_u=u^{-\frac{2}{n'-2}}\nu$, which give \eqref{eq:conf-scalar} and \eqref{eq:conf-mean}. Put
$$
w(t)=\frac{\varepsilon+e^{\alpha t}}{1+\varepsilon},\qquad \breve g=w^{4/(k-1)}\bar g.
$$

\emph{Step 3: completeness.} On $Z$ we have $t\le0$, hence $\frac{\varepsilon}{1+\varepsilon}\le w\le1$, so $\breve g$ is bi-Lipschitz to the complete metric $\bar g$ and is complete. (One cannot take $\varepsilon=0$: the factor $e^{\alpha t}$ would make the end $t\to-\infty$ of finite length.)

\emph{Step 4: the interior.} Since $w$ depends only on $t$ and $\bar\psi$ only on $x$, we have $\langle\nabla\bar\psi,\nabla w\rangle=0$, and $\Delta_{\bar g}w=w''$. Writing $\mathcal V=\mathcal R^\mu_{g_\Sigma}(\psi)\ge\sigma$, \eqref{eq:conf-scalar} gives
$$
\mathcal R^\mu_{\breve g}(\bar\psi)=w^{-\frac4{k-1}}\left(\mathcal V-\frac{4k}{k-1}\frac{w''}{w}\right).
$$
Since $0\le\frac{w''}{w}=\frac{\alpha^2e^{\alpha t}}{\varepsilon+e^{\alpha t}}\le\alpha^2$ and $w^{-4/(k-1)}\ge1$, \eqref{eq:alpha} implies
$$
\mathcal R^\mu_{\breve g}(\bar\psi)\ge\delta_{\mathrm{int}}\coloneqq\sigma-\frac{4k}{k-1}\alpha^2>0.
$$
(For $\varepsilon=0$, $\alpha=\frac{k-1}2$ and $\mathcal V\equiv k(k-1)$, the metric $\breve g=e^{2t}(dt^2+g_\Sigma)$ is the cone over $g_\Sigma$, and its weighted scalar curvature $\mathcal R^\mu_{\breve g}(\bar\psi)$ vanishes.)

\emph{Step 5: the boundary.} At $t=0$ we have $w=1$, so the metric induced by $\breve g$ on $\partial Z$ is $g_\Sigma$, and $\partial_\nu w=w'(0)=\frac{\alpha}{1+\varepsilon}$. By \eqref{eq:conf-mean}, \eqref{eq:conf-trace} and $\|d\eta'\|_{\mathrm{tr},g_\Sigma}\le k\|d\eta'\|_{g_\Sigma}\le kL_*$,
$$
H_{\partial Z}(\breve g)+\nu_{\breve g}(\bar\psi)-\|d\eta'\|_{\mathrm{tr},\breve g}\ge\frac{2k}{k-1}\cdot\frac{\alpha}{1+\varepsilon}-kL_*\eqqcolon\delta_\partial.
$$
By \eqref{eq:alpha}, $\frac{2k}{k-1}\alpha>kL_*$, so $\delta_\partial>0$ for $\varepsilon$ sufficiently small.

\emph{Step 6: conclusion.} Let $\delta=\min\{\delta_{\mathrm{int}},\delta_\partial\}$. The manifold $(Z^{k+1},\breve g)$, the weight $\bar\psi$, the parameter $\mu>\frac{k}{k+1}=\frac{n'-1}{n'}$ and the map $\eta':\partial Z\to\mathbb S^k$ satisfy all hypotheses of Theorem \ref{thm:WWXZ-mean}: $B_\pm$ are the preimages of antipodal points and all their components have interior points, $\partial Z\setminus B_+$ is precompact (Step 1), and $Z$ is complete (Step 3). Hence $\deg\eta'=0$, a contradiction.
\end{proof}

\subsection{The terminal obstruction}
\begin{corollary}\label{cor:terminal}
\textit{Let $\lambda=-\frac12$ and $k\ge2$. Assume the setting of Lemma \ref{lem:terminal-map}: $(V,g_V)$ is a Riemannian manifold, $X\subset V$ is compact, $D\subset X$ is closed, $P=X\setminus D$ is an oriented $k$ dimensional embedded submanifold without boundary, $g_P$ is a complete metric on $P$ satisfying \textup{(B)}, and $\eta:V\to\mathbb S^k$ is a smooth map such that $\eta|_P$ is $1$-Lipschitz with respect to $g_P$ and is proper over $\mathbb S^k\setminus\eta(D)$ with non-zero degree. Assume in addition that $\dim_{\mathcal H}D\le k-7$. In particular, $D=\varnothing$ if $k\le6$. If $\psi_k\in C^\infty(P)$ and a constant $\sigma$ satisfy
$$
\mathcal R^{a_k}_{g_P}(\psi_k)\ge\sigma\qquad\text{on }P,
$$
then $\sigma\le k(k-1)$.}
\end{corollary}
\begin{proof}
The map $\eta$ is Lipschitz with respect to $g_V$ near the compact set $X$, so $\dim_{\mathcal H}\eta(D)\le k-7<k-1$. Hence $\mathbb S^k\setminus\eta(D)$ is a connected and dense open set, and Lemma \ref{lem:terminal-map} applies. Suppose $\sigma>k(k-1)$, and let $\eta_1$ be the $1$-Lipschitz map given by Lemma \ref{lem:terminal-map}. We apply Lemma \ref{lem:half-cylinder} to $(P,g_P)$, the weight $\psi_k$, the parameter $\mu=a_k=\frac{k+1}{k+2}$, which satisfies $\mu>\frac{k}{k+1}$ by Lemma \ref{lem:parameter}(3), and $L_0=1$. Since $\mathcal R^{\mu}_{g_P}(\psi_k)\ge\sigma>k(k-1)$, we obtain $\deg\eta_1=0$, a contradiction.
\end{proof}

\begin{remark}\label{rem:k2}
For $k=2$ we have $D=\varnothing$, so $P$ is a closed surface, and Corollary \ref{cor:terminal} also follows from the Gauss-Bonnet theorem, as in \cite[proof of Theorem C]{BZ2026}. Indeed, let $P_0$ be a component on which $\eta$ has non-zero degree. Then $\sigma\operatorname{Area}(P_0)\le\int_{P_0}\mathcal R^{a_2}_{g_P}(\psi_2)\le\int_{P_0}R_{g_P}=4\pi\chi(P_0)$, since $a_2>0$ and $\int_{P_0}\Delta_{g_P}\psi_2=0$, while $\operatorname{Area}(P_0)\ge4\pi$ since $\eta$ is $1$-Lipschitz of non-zero degree. If $\sigma>0$, then $\chi(P_0)=2$ and $\sigma\le2=k(k-1)$.
\end{remark}

\section{\texorpdfstring{Sphere-torus data and circle-by-circle descent}{Sphere-torus data and circle-by-circle descent}}\label{sec:descent}
In this section the torus directions are removed one circle at a time. Each step takes place in the infinite cyclic covering associated with the last circle factor: the lifted circle coordinate $\tau$ unwraps the last circle to a line, a band is cut out on both sides of a level set of $\tau$, and a warped $\mu$-bubble is constructed in this band. Since the target is always a product $\mathbb S^k\times\mathbb T^j$, the bands have no side boundary, and the blow up of Proposition \ref{prop:BZ-blowup} is always applied on a complete space. This differs from \cite{BZ2026}, where the descent \cite[Proposition 4.7]{BZ2026} is carried out over targets of the form $Y\times[-1,1]^l$, the $\mu$-bubbles are constructed in bands which may have side boundary, and torus factors are handled by passing to coverings with large cubes, as in the proof of \cite[Theorem B]{BZ2026}.

\subsection{Sphere-torus data}
Let $Y_j=\mathbb S^k\times\mathbb T^j$, with $\mathbb T^0$ a point. Since the metric on the torus plays no role below, we identify $\mathbb T^j$ with $\mathbb R^j/\mathbb Z^j$ by a linear isomorphism, and we write its last factor as $\mathbb S^1=\mathbb R/\mathbb Z$.

The following notion is modeled on the data controlled over $Y\times[-1,1]^l$ in \cite[Definition 4.2]{BZ2026}, with this target replaced by $Y_j$. In particular, the blow up condition \textup{(B)} below is the one in \cite[Definition 4.2]{BZ2026}.

\begin{definition}\label{def:datum}
Let $\lambda=-\frac12$, $k\geq2$, $j\ge0$ and $d=k+j$. A \emph{$d$ dimensional sphere-torus datum over $Y_j$} is a tuple $\mathfrak D=(V,g_V;P,g_P;\psi,\sigma;A)$ such that the following hold.
\begin{itemize}
\item[(D1)] (Compactification) $V$ is a smooth manifold without boundary and $g_V$ is a Riemannian metric on $V$, which need not be complete. $P\subset V$ is a $d$ dimensional, oriented, smooth embedded submanifold without boundary, not necessarily connected, whose closure $X=\overline P$ in $V$ is compact. The set $D\coloneqq X\setminus P$ satisfies $\dim_{\mathcal H}D\le d-7$ with respect to $g_V$. For $d\le6$ this means $D=\varnothing$.
\item[(D2)] (Metric) $g_P$ is a complete Riemannian metric on $P$ with $g_P\ge g_V|_P$, satisfying the blow up condition \textup{(B)} of Lemma \ref{lem:terminal-map}: for every $p\in D$ there are $c_p,r_p>0$ such that $g_P\ge c_p\,d_{g_V}(\cdot,p)^{-2}g_V$ on $\{x\in P:d_{g_V}(x,p)<r_p\}$.
\item[(D3)] (Curvature) $\psi\in C^\infty(P)$ and $\sigma\in\mathbb R$ is a constant with $\mathcal R^{a_d}_{g_P}(\psi)\ge\sigma$ on $P$.
\item[(D4)] (Map) $A=(\eta,q_1,\ldots,q_j):V\to Y_j$ is smooth, $\eta|_P:(P,g_P)\to\mathbb S^k$ is $1$-Lipschitz, and $\deg A\ne0$, where $\deg A$ denotes the degree of $A|_P$ over $Y_j\setminus A(D)$, see Remark \ref{rem:datum}.
\end{itemize}
\end{definition}

\begin{remark}\label{rem:datum}
(a) An embedded submanifold is locally closed, so $P$ is open in $X$ and $D$ is compact. The map $A$ is Lipschitz with respect to $g_V$ near the compact set $X$, so $\dim_{\mathcal H}A(D)\le d-7<\dim Y_j-1$. Hence $Y_j\setminus A(D)$ is a connected, dense open set.

(b) If $C\subset Y_j\setminus A(D)$ is compact, then $A^{-1}(C)\cap X$ is compact and disjoint from $D$, hence a compact subset of $P$. Thus $A|_P$ is proper over $Y_j\setminus A(D)$, and its degree there is well defined. It can be computed at any regular value in this connected open set, and it is the sum of the degrees over the components of $P$. For a given value only finitely many components contribute.

(c) For $j=0$, a datum satisfies the hypotheses of Corollary \ref{cor:terminal}, with $\eta=A$. Hence $\sigma\le k(k-1)$.

(d) If $N^d$ is a closed oriented manifold, $\psi\in C^\infty(N)$, and $A$ satisfies (D4) with $P=N$, then $(N,g;N,g;\psi,\min_N\mathcal R^{a_d}_g(\psi);A)$ is a datum with $D=\varnothing$.
\end{remark}

The following slicing property is the real-valued case of \cite[Lemma 4.3]{BZ2026}.
\begin{lemma}\label{lem:slicing-dim}
\textit{Let $U$ be a smooth manifold with a Riemannian metric, let $\mathcal A\subset U$ be closed with $\dim_{\mathcal H}\mathcal A\le s$, and let $\tau:U\to\mathbb R$ be smooth. Then $\dim_{\mathcal H}\left(\mathcal A\cap\tau^{-1}(a)\right)\le s-1$ for almost every $a\in\mathbb R$. If $s<1$, this means that $\mathcal A\cap\tau^{-1}(a)=\varnothing$ for almost every $a$.}
\end{lemma}
For $s<1$, local Lipschitz continuity gives $\dim_{\mathcal H}\tau(\mathcal A)\le s<1$, so $\tau(\mathcal A)$ has Lebesgue measure zero.

\subsection{The descent}
\begin{proposition}\label{prop:descent}
\textit{Let $j\ge1$, and let $\mathfrak D=(V,g_V;P,g_P;\psi,\sigma;A)$ be a $d$ dimensional sphere-torus datum over $Y_j$, $d=k+j$. Then for every $\varepsilon>0$ there is a $(d-1)$ dimensional sphere-torus datum $\mathfrak D'=(V',g_{V'};P',g_{P'};\psi',\sigma';A')$ over $Y_{j-1}$ with $\sigma'\ge\sigma-\varepsilon$. More precisely, $V'$ is the infinite cyclic covering of $V$ obtained by pulling back $\mathbb R\to\mathbb S^1$ along $q_j$, $g_{V'}$ is the lift of $g_V$, $P'$ is the regular part of the boundary of a warped $\mu$-bubble constructed in the lift of $(P,g_P)$ to $V'$, and $A'$ is the lift of $(\eta,q_1,\ldots,q_{j-1})$.}
\end{proposition}
\begin{proof}
\emph{Step 1: the infinite cyclic covering.} Let
$$
\widehat V=\{(x,t)\in V\times\mathbb R:q_j(x)=t\mod1\},\qquad \pi(x,t)=x,\qquad\tau(x,t)=t.
$$
Then $\pi$ is a smooth covering, the pullback of $\mathbb R\to\mathbb S^1$ along $q_j$, with deck transformation $(x,t)\mapsto(x,t+1)$. We lift the metrics, sets and weight: $\hat g_V=\pi^*g_V$, $\widehat X=\pi^{-1}(X)$, $\widehat D=\pi^{-1}(D)$, $\widehat P=\pi^{-1}(P)$, $\hat g_P=\pi^*g_P$ and $\hat\psi=\psi\circ\pi$. Then $\widehat P$ is an oriented smooth embedded submanifold of $\widehat V$ with $\widehat P=\widehat X\setminus\widehat D$, $(\widehat P,\hat g_P)$ is complete as a Riemannian covering of a complete manifold, $\hat g_P\geq\hat g_V$, $\mathcal R^{a_d}_{\hat g_P}(\hat\psi)\ge\sigma$, and $\dim_{\mathcal H}\widehat D=\dim_{\mathcal H}D$. The blow up condition (B) lifts with the same constants: if $\hat p\in\widehat D$ and $p=\pi(\hat p)$, then $d_{g_V}(\pi\hat x,p)\le d_{\hat g_V}(\hat x,\hat p)$, hence $\hat g_P\ge c_p\,d_{g_V}(\pi\,\cdot,p)^{-2}\hat g_V\ge c_p\,d_{\hat g_V}(\cdot,\hat p)^{-2}\hat g_V$ on $\{\hat x\in\widehat P:d_{\hat g_V}(\hat x,\hat p)<r_p\}$. Put
$$
A^\flat=(\eta,q_1,\ldots,q_{j-1}):V\longrightarrow Y_{j-1},\qquad G=A^\flat\circ\pi:\widehat V\longrightarrow Y_{j-1}.
$$
Then $G$ is smooth, $G(\widehat D)=A^\flat(D)$ is compact, and the spherical component $\eta\circ\pi$ of $G$ is $1$-Lipschitz with respect to $\hat g_P$. Since $\widehat X\cap\tau^{-1}(I)=\{(x,t):x\in X,\ t\in I,\ q_j(x)=t\mod1\}$ is closed in $X\times I$, we have
\begin{equation}\label{eq:slab-compact}
\widehat X\cap\tau^{-1}(I)\text{ is compact for every bounded closed interval }I.
\end{equation}
Moreover $d\tau=\pi^*(q_j^*d\theta)$. Put $C_q=1+\max_X|q_j^*d\theta|_{g_V}$. Since $g_P\ge g_V|_P$, the $g_P$-norm of a covector is at most its $g_V$-norm, so
\begin{equation}\label{eq:tau-lip}
|d\tau|_{\hat g_P}\le C_q\ \text{ on }\widehat P,\qquad\text{hence}\qquad|\tau(x)-\tau(x')|\le C_q\,d_{\hat g_P}(x,x')
\end{equation}
for $x,x'$ in the same component of $\widehat P$.

\emph{Step 2: a good slice.} Choose $a\in\mathbb R$ such that (i) $a\mod1$ is a regular value of $q_j|_P$, so that $a$ is a regular value of $\tau|_{\widehat P}$, and (ii) $\dim_{\mathcal H}(\widehat D\cap\tau^{-1}(a))\le d-8$. By Sard's theorem and Lemma \ref{lem:slicing-dim}, applied to the closed set $\widehat D$, almost every $a$ has these properties. Let
$$
S=\widehat P\cap\tau^{-1}(a).
$$
This is a properly embedded smooth hypersurface of $\widehat P$, which we orient as the boundary of $\{\tau<a\}$, with the outward normal pointing in the direction of increasing $\tau$.

\emph{The degree on the slice.} Let $y\in Y_{j-1}\setminus G(\widehat D)$ be a regular value of $G|_S$. We claim that
\begin{equation}\label{eq:slice-degree}
\deg_y(G|_S)\coloneqq\sum_{x\in S\cap G^{-1}(y)}\operatorname{sgn}\det\left(dG|_{T_xS}\right)\ne0.
\end{equation}
Indeed, $\pi$ maps $\tau^{-1}(a)$ diffeomorphically onto $q_j^{-1}([a])$ and identifies $S\cap G^{-1}(y)$ with $P\cap A^{-1}(y,[a])$. Since $y\notin G(\widehat D)=A^\flat(D)$, the point $(y,[a])$ lies outside $A(D)$. The lift $\widehat A=(G,\tau)$ shows that $(y,[a])$ is a regular value of $A|_P$ and that the corresponding local degrees of $A|_P$ and $G|_S$ differ by a common sign. The common fiber is compact by Remark \ref{rem:datum}(b), hence finite. Computing $\deg A$ at $(y,[a])$ therefore gives $|\deg_y(G|_S)|=|\deg A|>0$, proving \eqref{eq:slice-degree}.

\emph{Step 3: the slice has points far away on both sides.} Let $\widehat P_1$ be the union of the components of $\widehat P$ which meet $S$. It is open and closed in $\widehat P$, $(\widehat P_1,\hat g_P)$ is complete, $S\subset\widehat P_1$, and $d_{\hat g_P}(\cdot,S)$ is finite on $\widehat P_1$. Choose $y_0\in Y_{j-1}\setminus G(\widehat D)$ which is a regular value of $G|_{\widehat P}$ and of $G|_S$. Since $G(\widehat D)$ is compact with $\dim_{\mathcal H}G(\widehat D)\le d-7<\dim Y_{j-1}$, almost every point qualifies. Let $\Gamma=G^{-1}(y_0)\cap\widehat P_1$, oriented as in Lemma \ref{lem:intersection}. Then $\Gamma$ is a one dimensional manifold without boundary which is closed in $\widehat V$: its limit points in $\widehat V$ lie in the closed set $\widehat X$ and satisfy $G=y_0$, so they do not lie in $\widehat D$. Hence they lie in $\widehat P$, and therefore in $\Gamma$, since $\widehat P_1$ is closed in $\widehat P$. By \eqref{eq:slab-compact}, $\tau|_\Gamma:\Gamma\to\mathbb R$ is proper. Since $y_0$ is a regular value of $G|_S$, $\Gamma$ is transverse to $S$, that is, $a$ is a regular value of $\tau|_\Gamma$. By the last assertion of Lemma \ref{lem:intersection}, which is pointwise and does not use the compactness of $\Gamma$, applied with $E=\{\tau<a\}$, we have $\operatorname{sgn}\det(dG|_{T_xS})=\operatorname{sgn}d\tau(\dot\Gamma)$ for $x\in\Gamma\cap S$. Since $\Gamma\cap S=\Gamma\cap\tau^{-1}(a)$ is compact and discrete, it is finite, and by \eqref{eq:slice-degree},
$$
\sum_{x\in\Gamma\cap S}\operatorname{sgn}d\tau(\dot\Gamma)=\deg_{y_0}(G|_S)\ne0.
$$
Every component of $\Gamma$ is either a circle or a properly embedded line along which $\tau$ tends to $\pm\infty$ at each end. Circles, and lines along which $\tau$ tends to the same infinity at both ends, contribute zero to the left hand side, since the sign of $\tau-a$ changes at each transverse intersection and agrees at the two ends. Therefore there is a component $\Gamma_c$ along which $\tau$ tends to $-\infty$ at one end and to $+\infty$ at the other, so $\tau(\Gamma_c)=\mathbb R$. By \eqref{eq:tau-lip}, $d_{\hat g_P}(x,S)\ge|\tau(x)-a|/C_q$ for $x\in\Gamma_c$. Consequently, both $\{\tau>a\}$ and $\{\tau<a\}$ contain points arbitrarily far from $S$.

\emph{Step 4: a band.} On $\widehat P_1$ define
$$
\rho(x)=
\begin{cases}
d_{\hat g_P}(x,S),&\tau(x)\ge a,\\
-d_{\hat g_P}(x,S),&\tau(x)<a.
\end{cases}
$$
The signed distance $\rho$ is $1$-Lipschitz since every path joining points on opposite sides meets $S$. By Step 3 it takes all real values on $\Gamma_c$. By \cite{GW} there is a smooth function $\tilde s$ with $|\tilde s-\rho/2|<\frac12$ and $|\nabla\tilde s|_{\hat g_P}\le1$. By Sard's theorem choose a regular value $c\in(-\frac12,\frac12)$ of $\tilde s$, and put $s=\tilde s-c$. Then $0$ is a regular value of $s$, and
$$
|s-\rho/2|<1,\qquad|\nabla s|_{\hat g_P}\le1.
$$
Choose $L\ge2$ such that $\pm L$ are regular values of $s$ and $\frac{3\pi^2}{2L^2}<\frac\varepsilon2$, and let
$$
W=s^{-1}([-L,L]),\qquad\partial_\pm W=s^{-1}(\pm L).
$$
Since $s$ is unbounded from above and from below on the connected set $\Gamma_c$, both $\partial_\pm W$ are nonempty. Thus $W$ is a band in $\widehat P_1$ in the sense of Section \ref{sec:bubbles}. It may be noncompact, with its noncompact directions approaching $\widehat D$. From $|s-\rho/2|<1$ we obtain
\begin{equation}\label{eq:band-position}
\begin{gathered}
W\subset\mathcal N_{2L+2},\qquad S\subset\{|s|<1\}\subset\overset{\circ}{W},\\
\partial_-W\subset\{\rho<0\}=\{\tau<a\},\qquad\partial_+W\subset\{\tau>a\},
\end{gathered}
\end{equation}
where $\mathcal N_\ell=\{x\in\widehat P_1:d_{\hat g_P}(x,S)\le\ell\}$.

\emph{Step 5: localization in finite width.} For every $\ell>0$,
\begin{equation}\label{eq:finite-width}
\overline{\mathcal N_\ell}\text{ is compact in }\widehat V,\qquad\text{and}\qquad\overline{\mathcal N_\ell}\cap\widehat D\subset\widehat D\cap\tau^{-1}(a).
\end{equation}
By \eqref{eq:tau-lip}, $\mathcal N_\ell\subset\widehat X\cap\tau^{-1}([a-C_q\ell,a+C_q\ell])$, so the first assertion follows from \eqref{eq:slab-compact}. For the second, let $\hat p\in\widehat D$ with $\tau(\hat p)\ne a$, and suppose $x_i\in\mathcal N_\ell$ converge to $\hat p$. Choose $r>0$ smaller than $r_{\hat p}$ in (B) such that the $\hat g_V$-ball $B(\hat p,r)$ does not meet $\tau^{-1}(a)$. Choose $z_i\in S$ and curves $\gamma_i$ in $\widehat P_1$ from $x_i$ to $z_i$ with $\hat g_P$-length at most $\ell+1$. Since $z_i\notin B(\hat p,r)$, the curve $\gamma_i$ leaves $B(\hat p,r)$. Let $\varrho=d_{\hat g_V}(\cdot,\hat p)$. Inside $B(\hat p,r)$ we have $|\gamma_i'|_{\hat g_P}\ge\sqrt{c_{\hat p}}\,\varrho^{-1}|\gamma_i'|_{\hat g_V}\ge\sqrt{c_{\hat p}}\,|(\log\varrho\circ\gamma_i)'|$, hence
$$
\ell+1\ge L_{\hat g_P}(\gamma_i)\ge\sqrt{c_{\hat p}}\,\log\frac{r}{\varrho(x_i)}\longrightarrow\infty,
$$
a contradiction. This is the argument of \cite[Lemma 4.5]{BZ2026}.

\emph{Step 6: the $\mu$-bubble and its stability.} Let $h_L$ be as in Lemma \ref{lem:tan} and $\widehat h=h_L\circ s$. We apply Proposition \ref{prop:bubble} to the complete manifold $(\widehat P_1,\hat g_P)$, the weight $\hat\psi$, the function $s$, the band $W$ and the potential $h_L$, with $\Omega_0=\widehat P_1\cap\{s\le0\}$. We obtain a warped $\mu$-bubble $\Omega\subset\widehat P_1$ with $\Omega\triangle\Omega_0\subset W$, such that every point of $\partial_-W$ has a neighborhood contained in $\Omega$ and every point of $\partial_+W$ has a neighborhood disjoint from $\Omega$. The boundary $\Sigma=\partial\Omega$ is closed in $\widehat P_1$ and contained in $\overset{\circ}{W}$, $\Sigma=\Sigma_{\mathrm{reg}}\sqcup D_\Sigma$ with $\dim_{\mathcal H}D_\Sigma\le d-8$, and $\Sigma_{\mathrm{reg}}$ satisfies the first variation identity and the stability inequality of Proposition \ref{prop:bubble}(3). Since $|\nabla s|\le1$, Lemma \ref{lem:tan} gives
$$
\mathcal R^{a_d}_{\hat g_P}(\hat\psi)-2|\nabla\widehat h|+\frac23\widehat h^2\ge\sigma_L\coloneqq\sigma-\frac{3\pi^2}{2L^2}.
$$
By Lemma \ref{lem:pointwise}, $\mathcal Q^{g_\Sigma,\psi_\Sigma}_{a_d,\sigma_L}\ge0$ on $C^\infty_c(\Sigma_{\mathrm{reg}})$, where $g_\Sigma=\hat g_P|_{\Sigma_{\mathrm{reg}}}$ and $\psi_\Sigma=\hat\psi|_{\Sigma_{\mathrm{reg}}}$.

\emph{Step 7: the new compactification.} Put
$$
P'=\Sigma_{\mathrm{reg}},\qquad X'=\overline{\Sigma_{\mathrm{reg}}}\ (\text{closure in }\widehat V),\qquad D'=X'\setminus P'.
$$
By \eqref{eq:band-position} and \eqref{eq:finite-width}, $X'\subset\overline{\mathcal N_{2L+2}}$ is compact. A point of $X'$ which lies in $\widehat P$ lies in the closed set $\Sigma$, whereas a point which does not lie in $\widehat P$ lies in $\widehat D\cap\overline{\mathcal N_{2L+2}}\subset\widehat D\cap\tau^{-1}(a)$. Hence
\begin{equation}\label{eq:new-singular}
D'\subset D_\Sigma\cup\left(\widehat D\cap\tau^{-1}(a)\right),\qquad\dim_{\mathcal H}D'\le d-8=(d-1)-7.
\end{equation}
Here $D_\Sigma\subset\widehat P$ is a countable union of compact subsets of $\widehat P$, on which the distances $d_{\hat g_P}$ and $d_{\hat g_V}$ are locally bi-Lipschitz equivalent, so the Hausdorff dimension of $D_\Sigma$ is the same for both. The contribution from the old singular set is controlled by Step 2(ii). Moreover, $P'$ is a $(d-1)$ dimensional smooth embedded submanifold of $\widehat V$ without boundary, which we orient as the boundary of $\Omega$, the outward normal first.

\emph{Step 8: the degree is preserved.} Choose $y\in Y_{j-1}$ such that $y\notin G(\widehat D)\cup G(D_\Sigma)$ and $y$ is a regular value of the restrictions of $G$ to $\widehat P$, $S$, $\partial_\pm W$ and $\Sigma_{\mathrm{reg}}$. Since $G$ is locally Lipschitz and $D_\Sigma$ is $\sigma$-compact, Sard's theorem and the dimension estimates show that almost every $y$ has these properties. Let
$$
\Gamma_y=G^{-1}(y)\cap W\subset\widehat P_1.
$$
As in Step 3, $G^{-1}(y)\cap\widehat P_1$ is closed in $\widehat V$, so $\Gamma_y$ is closed in $\widehat V$. It is contained in the compact set $\widehat X\cap\tau^{-1}([a-C_q(2L+2),a+C_q(2L+2)])$, hence compact. By transversality, $\Gamma_y$ is a compact one dimensional manifold with $\partial\Gamma_y=G^{-1}(y)\cap\partial W$, which does not meet $D_\Sigma$ and is transverse to $\Sigma_{\mathrm{reg}}$ and to $S$.

We apply Lemma \ref{lem:intersection} to $(\widehat P_1,Y_{j-1},G,\Gamma_y)$ twice: first with $E=\Omega$, for which $\partial E\cap O=\Sigma_{\mathrm{reg}}$ on the neighborhood $O=\widehat P_1\setminus D_\Sigma$ of $\Gamma_y$, and then with $\Omega_\tau=\widehat P_1\cap\{\tau<a\}$, for which $\partial\Omega_\tau=S$. By the barrier property and \eqref{eq:band-position}, $\chi_\Omega(e)=\chi_{\Omega_\tau}(e)$ for every $e\in\partial\Gamma_y$: both are $1$ for $e\in\partial_-W$ and $0$ for $e\in\partial_+W$. Since $G^{-1}(y)\cap\Sigma_{\mathrm{reg}}=\Gamma_y\cap\Sigma_{\mathrm{reg}}$ (as $\Sigma\subset W$) and $G^{-1}(y)\cap S=\Gamma_y\cap S$ (as $S\subset\overset{\circ}W$), Lemma \ref{lem:intersection} and \eqref{eq:slice-degree} give
\begin{align*}
\deg_y\left(G|_{\Sigma_{\mathrm{reg}}}\right)&=-\sum_{e\in\partial\Gamma_y}o(e)\chi_\Omega(e)=-\sum_{e\in\partial\Gamma_y}o(e)\chi_{\Omega_\tau}(e)\\
&=\deg_y(G|_S)\ne0.
\end{align*}
Since $y\notin G(D')$ and $y$ is a regular value of $G|_{P'}$, by Remark \ref{rem:datum}(b) the left hand side is the degree of $G|_{P'}$ over $Y_{j-1}\setminus G(D')$. Since the count takes place on the compact curves $\Gamma_y$, we need neither compactness of $\Omega\triangle\Omega_0$ nor a fundamental class for the possibly noncompact $\Sigma_{\mathrm{reg}}$.

\emph{Step 9: blow up and warping.} Put $\zeta=\varepsilon/2$. By Step 6 and Lemma \ref{lem:blowup-warp}, applied to the band $W$ in the complete, possibly disconnected manifold $(\widehat P_1,\hat g_P)$, there are a complete metric $g'$ and a function $\psi'$ on $\Sigma_{\mathrm{reg}}$ such that $g'\ge g_\Sigma$, every $p\in D_\Sigma$ has a neighborhood on which $g'\ge c_p\,d_{\hat g_P}(\cdot,p)^{-2}g_\Sigma$, and
$$
\mathcal R^{a_{d-1}}_{g'}(\psi')\ge\sigma_L-\zeta\qquad\text{on }\Sigma_{\mathrm{reg}}.
$$

\emph{Step 10: the new datum.} Let
\begin{gather*}
(V',g_{V'})=(\widehat V,\hat g_V),\qquad P'=\Sigma_{\mathrm{reg}},\qquad g_{P'}=g',\\
\psi'\ \text{as in Step 9},\qquad \sigma'=\sigma-\frac{3\pi^2}{2L^2}-\zeta,\qquad A'=G.
\end{gather*}
We verify Definition \ref{def:datum} for $d-1=k+(j-1)$ and the target $Y_{j-1}$. (D1) was shown in Step 7. For (D2), $g'$ is complete and $g'\ge g_\Sigma=\hat g_P|_{P'}\ge\hat g_V|_{P'}$. For the blow up condition, let $p\in D'$. If $p\in D_\Sigma\subset\widehat P$, then near $p$ the manifold $\widehat P$ is embedded and $\hat g_P$ is smooth, so $d_{\hat g_P}(x,p)\le C\,d_{\hat g_V}(x,p)$ for $x\in\widehat P$ near $p$. Together with $g_\Sigma\ge\hat g_V$ this gives $g'\ge c_pC^{-2}d_{\hat g_V}(\cdot,p)^{-2}\hat g_V$ near $p$. If $p\in\widehat D\cap\tau^{-1}(a)$, the condition follows from the lifted condition (B) in Step 1 and $g'\ge\hat g_P|_{P'}$. (D3) was shown in Step 9. For (D4), $A'=G$ is smooth, $\eta\circ\pi$ is $1$-Lipschitz with respect to $g'\ge\hat g_P|_{P'}$, and the degree is non-zero by Step 8. Finally, by the choice of $L$ and $\zeta$, $\sigma'\ge\sigma-\varepsilon$.
\end{proof}

\subsection{The obstruction for sphere-torus data}
\begin{theorem}\label{thm:datum-obstruction}
\textit{Let $\lambda=-\frac12$, $k\ge2$ and $j\ge0$. Every $(k+j)$ dimensional sphere-torus datum over $Y_j$ satisfies $\sigma\le k(k-1)$.}
\end{theorem}
\begin{proof}
We argue by induction on $j$. The case $j=0$ is Remark \ref{rem:datum}(c). Let $j\ge1$ and assume the claim for $j-1$. If a datum over $Y_j$ had $\sigma>k(k-1)$, Proposition \ref{prop:descent} with $\varepsilon=\frac12(\sigma-k(k-1))$ would produce a datum over $Y_{j-1}$ with $\sigma'\ge\sigma-\varepsilon>k(k-1)$, contradicting the induction hypothesis.
\end{proof}

\begin{corollary}\label{cor:closed-obstruction}
\textit{Let $k\ge2$, $j\ge0$, and let $(N^{k+j},g)$ be a closed, connected, oriented manifold. Suppose that $A=(A_{\mathbb S},A_{\mathbb T}):N\to\mathbb S^k\times\mathbb T^j$ is a smooth map of non-zero degree such that $A_{\mathbb S}$ is $1$-Lipschitz. Then there is no smooth function $\psi$ on $N$ such that
$$
\mathcal R_g(\psi)>k(k-1).
$$}
\end{corollary}
\begin{proof}
Otherwise, by Lemma \ref{lem:parameter}(5) and compactness, $\sigma=\min_N\mathcal R^{a_{k+j}}_g(\psi)\ge\min_N\mathcal R_g(\psi)>k(k-1)$, and Remark \ref{rem:datum}(d) and Theorem \ref{thm:datum-obstruction} give a contradiction.
\end{proof}

Corollary \ref{cor:closed-obstruction} extends the obstruction part of Theorem \ref{Chow} to all dimensions without a spin assumption. The full weighted rigidity is established in Corollary \ref{cor:closed-rigidity}. Its proof first uses the stronger coefficient $a_d$ obstruction in Theorem \ref{thm:datum-obstruction}, with $a_d<1$, to force the weight to be constant, and then applies the noncompact argument to $N\times\mathbb R$.

\begin{remark}\label{rem:descent-low-dim}
If $D=\varnothing$, which is automatic for $d\le6$, then $P=X$ is compact, every slab $\widehat P\cap\tau^{-1}(I)$ is compact by \eqref{eq:slab-compact}, and the band $W$ is compact. If moreover $d\le7$, then $\Sigma$ is a smooth closed hypersurface, Steps 5 and 7 of the proof of Proposition \ref{prop:descent} are trivial, Step 9 reduces to a warping by a first eigenfunction, and the datum produced is again a closed manifold. In particular, a descent which starts from a closed datum of dimension at most seven only involves closed manifolds.
\end{remark}

\section{Weighted obstructions on noncompact manifolds}\label{w-o}
\subsection{Targets with one real factor}
The following proposition is the key input for Sections \ref{sec-scalar-ricci} and \ref{sec-main}. One warped $\mu$-bubble in a compact band produces a sphere-torus datum, and Theorem \ref{thm:datum-obstruction} applies to it.
\begin{proposition}\label{prop-noncompact-obstruction}
\textit{
Let $n\ge3$, $2\leq k\leq n-1$, and $m=n-k-1$. Let
$(M^n,g)$ be a complete, connected, noncompact, oriented manifold. Suppose that there is a smooth proper map of non-zero degree
$$
F=(A_{\mathbb{S}},A_{\mathbb{T}},b):M\longrightarrow\mathbb{S}^k\times\mathbb{T}^m\times\mathbb{R}
$$
such that $A_{\mathbb{S}}$ is 1-Lipschitz. Then there is no smooth function $\psi$ on $M$ such that
$$
\mathcal R_g(\psi)>k(k-1).
$$
}
\end{proposition}

\begin{proof}
Suppose to the contrary that there is a smooth function $\psi$ on $M$ with $\mathcal R_g(\psi)>k(k-1)$. Put
$$
\mathcal V=\mathcal R^{a_n}_g(\psi),\qquad E=\frac12\left(\mathcal V-k(k-1)\right).
$$
By Lemma \ref{lem:parameter}(5), $\mathcal V\ge\mathcal R_g(\psi)$, so $E$ is a positive continuous function on $M$. The hypothesis requires a pointwise strict inequality but no uniform gap. The barrier of Lemma \ref{lem-h} nevertheless yields the constant gap $2\delta_1$ below.

\emph{Step 1: slicing and the barrier.} Write $A=(A_{\mathbb{S}},A_{\mathbb{T}})$ and $X=\mathbb{S}^k\times\mathbb{T}^m$. By
Lemma \ref{lem-degree-slice}, there is a regular value $t_0$ of $b$ such that $S_0=b^{-1}(t_0)$ is a closed oriented hypersurface with
$$
\deg(A|_{S_0})=\deg F\neq0.
$$
Lemma \ref{lem-proper-fcn} gives a smooth, proper and surjective function $\phi:M\to\mathbb R$ with $\phi^{-1}(0)=S_0$, $|\nabla\phi|<1$ and $\{\phi<0\}=\{b<t_0\}$, such that $0$ is a regular value of $\phi$. Apply Lemma \ref{lem-h} to $E$ and $\phi$, and let
$$
W=\phi^{-1}([-a,a]),\qquad
\widehat h=h\circ\phi,\qquad
\Omega_0=\{\phi\le0\},\qquad\sigma_1=k(k-1)+2\delta_1.
$$
Since $\phi$ is proper and surjective and $\pm a$ are regular values, $W$ is a compact band, with $\partial_\pm W=\phi^{-1}(\pm a)$.

\emph{Step 2: the $\mu$-bubble.} We apply Proposition \ref{prop:bubble} to $(P,g)=(M,g)$, the weight $\psi$, the function $s=\phi$, $L=a$, and the potential $h$. We obtain a warped $\mu$-bubble $\Omega$ with $\Omega\triangle\Omega_0\subset W$, such that $\Omega$ contains a neighborhood of $\partial_-W$ and is disjoint from a neighborhood of $\partial_+W$. Its boundary $\Sigma=\partial\Omega$ is a compact subset of $\overset{\circ}{W}$, $\Sigma=\Sigma_{\mathrm{reg}}\sqcup D_\Sigma$ with $\dim_{\mathcal H}D_\Sigma\le n-8$, and $\Sigma_{\mathrm{reg}}$ satisfies the first variation identity and the stability inequality of Proposition \ref{prop:bubble}(3). If $n\le7$, then $D_\Sigma=\varnothing$ and $\Sigma$ is a smooth closed hypersurface.

\emph{Step 3: stability.} By Lemma \ref{lem-h}, on $\overset{\circ}{W}$,
$$
\mathcal V-2|\nabla\widehat h|+\frac23\widehat h^2=k(k-1)+2\left(E-|\nabla\widehat h|+\frac13\widehat h^2\right)\ge\sigma_1.
$$
By Lemma \ref{lem:pointwise} with $d=n$, we have $\mathcal Q^{g_\Sigma,\psi_\Sigma}_{a_n,\sigma_1}\ge0$ on $C^\infty_c(\Sigma_{\mathrm{reg}})$, where $g_\Sigma$ is the induced metric and $\psi_\Sigma=\psi|_{\Sigma_{\mathrm{reg}}}$.

\emph{Step 4: blow up and warping.} Fix $\zeta\in(0,\delta_1)$. By Lemma \ref{lem:blowup-warp}, there are a complete metric $g'$ and a function $\psi'$ on $\Sigma_{\mathrm{reg}}$ such that $g'\ge g_\Sigma$, every $p\in D_\Sigma$ has a neighborhood on which $g'\ge c_p\,d_g(\cdot,p)^{-2}g_\Sigma$, and
$$
\mathcal R^{a_{n-1}}_{g'}(\psi')\ge\sigma_1-\zeta\qquad\text{on }\Sigma_{\mathrm{reg}}.
$$

\emph{Step 5: the first sphere-torus datum.} We claim that
$$
\mathfrak D=\left(M,g;\ \Sigma_{\mathrm{reg}},g';\ \psi',\sigma_1-\zeta;\ A\right),\qquad A=(A_{\mathbb S},A_{\mathbb T}):M\longrightarrow Y_m=\mathbb S^k\times\mathbb T^m,
$$
is an $(n-1)$ dimensional sphere-torus datum over $Y_m$. For (D1), the closure of $\Sigma_{\mathrm{reg}}$ is contained in the compact set $\Sigma$. The set $D=\overline{\Sigma_{\mathrm{reg}}}\setminus\Sigma_{\mathrm{reg}}$ is contained in $D_\Sigma$, so $\dim_{\mathcal H}D\le n-8=(n-1)-7$. Moreover, $\Sigma_{\mathrm{reg}}$ is a smooth embedded submanifold of $M$ without boundary, which we orient as the boundary of $\Omega$, with the outward normal first. (D2) and (D3) were shown in Step 4. For (D4), $A_{\mathbb S}$ is $1$-Lipschitz with respect to $g$, hence with respect to $g'\ge g|_{\Sigma_{\mathrm{reg}}}$. It remains to compute the degree.

Choose $y\in Y_m\setminus A(D_\Sigma)$ which is a regular value of the restrictions of $A$ to $M$, $\partial_\pm W$, $S_0$ and $\Sigma_{\mathrm{reg}}$. Then $\Gamma_y=A^{-1}(y)\cap W$ is a compact one dimensional manifold with $\partial\Gamma_y=A^{-1}(y)\cap\partial W$, which does not meet $D_\Sigma$ and is transverse to $\Sigma_{\mathrm{reg}}$ and to $S_0$. We apply Lemma \ref{lem:intersection} to $(M,Y_m,A,\Gamma_y)$, first to the set $\Omega$, for which $\partial\Omega\cap(M\setminus D_\Sigma)=\Sigma_{\mathrm{reg}}$, and then to the set $\Omega_\phi=\{\phi<0\}=\{b<t_0\}$. Since $0$ is a regular value of $\phi$, we have $\partial\Omega_\phi=S_0$, and $\Omega_\phi$ lies on one side of $S_0$. For $e\in\partial\Gamma_y$ we have $\chi_\Omega(e)=\chi_{\Omega_\phi}(e)$: both are $1$ on $\partial_-W$ and $0$ on $\partial_+W$. Since $\Sigma\cup S_0\subset\overset{\circ}{W}$, we have $A^{-1}(y)\cap\Sigma_{\mathrm{reg}}=\Gamma_y\cap\Sigma_{\mathrm{reg}}$ and $A^{-1}(y)\cap S_0=\Gamma_y\cap S_0$, and $\partial\Gamma_y$ does not meet $\Sigma\cup S_0$. Lemma \ref{lem:intersection} and Step 1 therefore give
$$
\left|\deg_y\left(A|_{\Sigma_{\mathrm{reg}}}\right)\right|=\left|\deg_y\left(A|_{S_0}\right)\right|=|\deg F|>0.
$$
By Remark \ref{rem:datum}(b), the degree of $\mathfrak D$ is therefore nonzero, so $\mathfrak D$ is a sphere-torus datum.

\emph{Step 6: conclusion.} By Theorem \ref{thm:datum-obstruction} with $j=m$, we have $\sigma_1-\zeta\le k(k-1)$. However, $\sigma_1-\zeta=k(k-1)+2\delta_1-\zeta>k(k-1)$, a contradiction.
\end{proof}

\subsection{A weighted band width estimate}
In this subsection, a Riemannian band $(W^n,g)$ is a compact, connected, orientable, smooth manifold with a metric $g$ and nonempty boundary $\partial W$ such that
$$
\partial W=\partial_{-}W\cup \partial_{+}W, \quad\partial_{-}W \neq \emptyset,\quad \partial_{+}W \neq \emptyset, \quad\partial_{-}W\cap \partial_{+}W=\emptyset.
$$
Each side is a union of components of $\partial W$. Its width $\operatorname{width}(W,g)$ is the distance between $\partial_-W$ and $\partial_+W$ with respect to the intrinsic distance $d_W$ of $W$.

\begin{proposition}\label{prop:bandwidth}
\textit{Let $k\ge2$, $j\ge0$ and $n=k+j+1$. Let $(W^n,g)$ be an oriented Riemannian band, and let
$$
G=(A,t):W\longrightarrow\mathbb S^k\times\mathbb T^j\times[-1,1]
$$
be a smooth map with $t^{-1}(-1)=\partial_-W$ and $t^{-1}(1)=\partial_+W$, whose spherical component $A_{\mathbb S}$ is $1$-Lipschitz, and such that the proper map $G|_{\overset{\circ}W}:\overset{\circ}{W}\to\mathbb S^k\times\mathbb T^j\times(-1,1)$ has non-zero degree. If there are $\psi\in C^\infty(W)$ and $\delta>0$ with
$$
\mathcal R_g(\psi)\ge k(k-1)+\delta\qquad\text{on }W,
$$
then
$$
\operatorname{width}(W,g)\le\pi\sqrt{\frac{6}{\delta}}.
$$}
\end{proposition}
\begin{proof}
The map $G|_{\overset{\circ}W}$ is proper because $W$ is compact and $t^{-1}(\pm1)=\partial_\pm W$. Write $Y=Y_j=\mathbb S^k\times\mathbb T^j$ and $\mathcal V=\mathcal R^{a_n}_g(\psi)$, so that $\mathcal V\ge\mathcal R_g(\psi)\ge k(k-1)+\delta$ by Lemma \ref{lem:parameter}(5). Suppose that $\ell_0\coloneqq\operatorname{width}(W,g)>\pi\sqrt{6/\delta}$. Choose $L$ with
$$
\pi\sqrt{\frac{3}{2\delta}}<L<\frac{\ell_0}2,
$$
so that $\frac{3\pi^2}{2L^2}<\delta$ and $2L<\ell_0$, and put $\zeta=\frac12\bigl(\delta-\frac{3\pi^2}{2L^2}\bigr)>0$.

\emph{Step 1: a band in the interior.} Let $u=d_W(\cdot,\partial_-W)$. Since the Riemannian distance on the open manifold $\overset{\circ}{W}$ is at least $d_W$, the function $u$ is $1$-Lipschitz on $\overset{\circ}{W}$. Choose $\epsilon>0$ with $2(1+\epsilon)L+8\epsilon<\ell_0$. By \cite{GW} there is a smooth function $\tilde u$ on $\overset{\circ}{W}$ with $|\tilde u-u|<\epsilon$ and $|\nabla\tilde u|\le1+\epsilon$. By Sard's theorem choose $c\in(3\epsilon,4\epsilon)$ such that $c$, $c+(1+\epsilon)L$ and $c+2(1+\epsilon)L$ are regular values of $\tilde u$, and put
$$
s=\frac{\tilde u-c-(1+\epsilon)L}{1+\epsilon},\qquad W'=s^{-1}([-L,L])=\tilde u^{-1}\bigl([c,c+2(1+\epsilon)L]\bigr).
$$
Then $|\nabla s|\le1$ and $0,\pm L$ are regular values of $s$. On $W'$ we have $2\epsilon<u<\ell_0-2\epsilon$, so every $x\in W'$ satisfies $d_W(x,\partial_-W)>2\epsilon$ and $d_W(x,\partial_+W)\ge \ell_0-u(x)>2\epsilon$. Hence $W'$ is a closed subset of $W$ at positive distance from $\partial W$, that is, a compact subset of $\overset{\circ}{W}$. At points of $\overset{\circ}W$ close to $\partial_-W$ we have $u<\epsilon$, hence $\tilde u<2\epsilon<c$ and $s<-L$. At points close to $\partial_+W$ we have $u>\ell_0-\epsilon$, hence $\tilde u>\ell_0-2\epsilon>c+2(1+\epsilon)L$ and $s>L$. Since $\overset{\circ}{W}$ is connected, $\partial_\pm W'=s^{-1}(\pm L)$ are nonempty. Thus $W'$ is a compact band in $P=\overset{\circ}{W}$ in the sense of Section \ref{sec:bubbles}.

\emph{Step 2: the $\mu$-bubble and its stability.} Let $h_L$ be as in Lemma \ref{lem:tan} and $\widehat h=h_L\circ s$. By Proposition \ref{prop:bubble}, applied to $(\overset{\circ}W,g)$, the weight $\psi$, the function $s$, the band $W'$, the potential $h_L$ and $\Omega_0=\{s\le0\}$, there is a warped $\mu$-bubble $\Omega$ with $\Omega\triangle\Omega_0\subset W'$. In particular, $\Omega$ contains $\{s<-L\}$ and is disjoint from $\{s>L\}$. Its boundary $\Sigma=\partial\Omega=\Sigma_{\mathrm{reg}}\sqcup D_\Sigma$ is a compact subset of $\overset{\circ}{W'}$ with $\dim_{\mathcal H}D_\Sigma\le n-8$. Since $|\nabla s|\le1$, Lemma \ref{lem:tan} gives
$$
\mathcal V-2|\nabla\widehat h|+\frac23\widehat h^2\ge\sigma_L\coloneqq k(k-1)+\delta-\frac{3\pi^2}{2L^2}\qquad\text{on }\overset{\circ}{W'},
$$
and Lemma \ref{lem:pointwise} yields $\mathcal Q^{g_\Sigma,\psi_\Sigma}_{a_n,\sigma_L}\ge0$ on $C^\infty_c(\Sigma_{\mathrm{reg}})$.

\emph{Step 3: blow up and warping.} By Lemma \ref{lem:blowup-warp}, there are a complete metric $g'\ge g_\Sigma$ and a function $\psi'$ on $\Sigma_{\mathrm{reg}}$ such that every $p\in D_\Sigma$ has a neighborhood on which $g'\ge c_p\,d_g(\cdot,p)^{-2}g_\Sigma$, where $d_g$ is the Riemannian distance on $\overset{\circ}{W}$, and $\mathcal R^{a_{n-1}}_{g'}(\psi')\ge\sigma_L-\zeta$.

\emph{Step 4: the degree.} Since $t$ is continuous on the compact set $W$ and $t^{-1}(\pm1)=\partial_\pm W$, we may choose regular values $t_-<t_1<t_+$ of $t|_{\overset{\circ}W}$ in $(-1,1)$ such that $\{t\le t_-\}\subset\{u<\epsilon\}$ and $\{t\ge t_+\}\subset\{d_W(\cdot,\partial_+W)<\epsilon\}\subset\{u>\ell_0-\epsilon\}$. Then $s<-L$ on $\{t\le t_-\}\cap\overset{\circ}W$ and $s>L$ on $\{t\ge t_+\}\cap\overset{\circ}W$, by Step 1. In particular, $W'\subset\{t_-<t<t_+\}$. Let $S_1=t^{-1}(t_1)\subset\overset{\circ}W$, oriented as the boundary of $\{t<t_1\}$. Choose $y\in Y\setminus A(D_\Sigma)$ which is a regular value of the restrictions of $A$ to $\overset{\circ}W$, $t^{-1}(t_-)$, $t^{-1}(t_+)$, $S_1$ and $\Sigma_{\mathrm{reg}}$. Then
$$
\Gamma_y=A^{-1}(y)\cap t^{-1}([t_-,t_+])
$$
is a compact one dimensional manifold in $\overset{\circ}W$ with $\partial\Gamma_y\subset t^{-1}(\{t_-,t_+\})$, which does not meet $D_\Sigma$ and is transverse to $\Sigma_{\mathrm{reg}}$ and to $S_1$. We apply Lemma \ref{lem:intersection} to $(\overset{\circ}W,Y,A,\Gamma_y)$, first to the set $\Omega$ and then to the set $\Omega_t=\overset{\circ}W\cap\{t<t_1\}$. An endpoint $e$ with $t(e)=t_-$ satisfies $s(e)<-L$, so $e\in\Omega$ and $e\in\Omega_t$. An endpoint with $t(e)=t_+$ satisfies $s(e)>L$, so $e\notin\Omega$ and $e\notin\Omega_t$. In particular $\partial\Gamma_y$ does not meet $\Sigma\cup S_1$. Since $\Sigma\subset W'\subset\{t_-<t<t_+\}$, we have $A^{-1}(y)\cap\Sigma_{\mathrm{reg}}=\Gamma_y\cap\Sigma_{\mathrm{reg}}$, and similarly $A^{-1}(y)\cap S_1=\Gamma_y\cap S_1$. Computing the degree of $G$ at the regular value $(y,t_1)$ and applying Lemma \ref{lem:intersection}, we obtain
$$
\left|\deg_y\left(A|_{\Sigma_{\mathrm{reg}}}\right)\right|=\left|\deg_y\left(A|_{S_1}\right)\right|=|\deg G|>0.
$$

\emph{Step 5: conclusion.} As in Step 5 of the proof of Proposition \ref{prop-noncompact-obstruction}, Steps 2--4 show that
$$
\mathfrak D=\left(\overset{\circ}W,g;\ \Sigma_{\mathrm{reg}},g';\ \psi',\sigma_L-\zeta;\ A|_{\overset{\circ}W}\right)
$$
is an $(n-1)$ dimensional sphere-torus datum over $Y_j$. Theorem \ref{thm:datum-obstruction} gives $\sigma_L-\zeta\le k(k-1)$, contrary to $\sigma_L-\zeta=k(k-1)+\zeta$.
\end{proof}

\begin{remark}\label{rem:bandwidth}
The constant $\pi\sqrt6$ comes from the factor $\frac1{1-\lambda}=\frac23$ in Lemma \ref{lem:pointwise} and is not optimal. For spin bands of dimension at most seven, Chow \cite[Theorem~B(ii)]{chow2025} obtains $2\pi/\sqrt\delta$. For $\lambda=-\frac12$ and one interval factor, the cube inequality of \cite[Theorem B]{BZ2026} gives the same bound $\pi\sqrt{6/\sigma}$ for the width of an $n$ dimensional band over an enlargeable base with $\mathcal R^{a_n}_g(\psi)\ge\sigma>0$. The proof of Theorem \ref{thm:infinity} uses only the finiteness of this bound. Since $d_W\ge d_M$ for a band $W$ contained in a Riemannian manifold $M$, Proposition \ref{prop:bandwidth} also bounds the distance in $M$ between the two sides.
\end{remark}

\section{Scalar curvature equality and nonnegative Ricci curvature}\label{sec-scalar-ricci}
In this section, following the method of \cite{Zhu1}, we show that  $R_M\equiv k(k-1)$ and $\operatorname{Ric}_g\ge0$. We first recall the following theorem of Kazdan \cite{Kazdan}, which Zhu used repeatedly in his proof \cite{Zhu1}.

\begin{theorem}(\cite[Theorem A]{Kazdan})\label{thm-kazdan}
\textit{
Let $(M,g)$ be a complete, connected, noncompact Riemannian manifold without
boundary, and let $L_g=-\Delta_g+P$ have smooth real potential $P$.
Suppose that a precompact connected smooth domain $D\Subset M$ has positive first Neumann eigenvalue for $L_g$, and that
$P\geq0$ on $M\setminus D$. Then there is a smooth function $u$ on $M$ such that $0<\mathrm{const}\le u\le\mathrm{const}$ and
$$
L_gu>0\quad\text{on}\,M.
$$
}
\end{theorem}

We now assume the hypotheses of Theorem \ref{main}. The weighted obstruction forces the scalar curvature inequality to hold with equality.
\begin{proposition}\label{prop-scalar}
\textit{
Under the hypotheses of Theorem \ref{main},
$$
R_M(g)\equiv k(k-1).
$$
}
\end{proposition}

\begin{proof}
Consider the Schr\"odinger operator
$$
L_g=-\Delta_g+\frac14\left(R_M(g)-k(k-1)\right).
$$
Its potential $P=\frac14(R_M(g)-k(k-1))$ is nonnegative. If it is not identically zero, it must be positive at some point $q$ in $M$. Choose a precompact connected smooth domain $D$ which contains $q$. The first Neumann eigenvalue of $L_g$ given by
$$
\mu(D,L_g)=\inf_{0\neq\chi\in H^1(D)}\frac{\int_{D}|\nabla_g\chi|^2+P\chi^2d\mu_g}{\int_{D}\chi^2d\mu_g}
$$
is positive: otherwise a first eigenfunction $\chi$ would satisfy $\int_D|\nabla_g\chi|^2+P\chi^2\,d\mu_g=0$, so $\chi$ would be constant and $P\chi^2\equiv0$, contradicting $P(q)>0$. By Theorem \ref{thm-kazdan}, there exists a smooth function $u$ on $M$ with $0<\mathrm{const}\le u\le\mathrm{const}$ such that $L_gu>0$. Setting $\psi=2\log u$, we obtain
\begin{align*}
0<&-\frac{\Delta_gu}{u}+\frac{1}{4}(R_M(g)-k(k-1))\\
=&\frac{1}{4}\left(\mathcal R_g(\psi)-k(k-1)\right).
\end{align*}
This contradicts Proposition \ref{prop-noncompact-obstruction}. Therefore $P\equiv0$.
\end{proof}

The next perturbation enlarges the metric along a direction in which the Ricci curvature is negative. This preserves the spherical 1-Lipschitz condition and yields precisely the local spectral positivity needed for Kazdan's theorem.

\begin{proposition}\label{prop-Ricci}
\textit{
Under the hypotheses of Theorem \ref{main},
$$
\operatorname{Ric}_g\geq0.
$$
}
\end{proposition}

\begin{proof}
By Proposition \ref{prop-scalar}, $R_g=k(k-1)$. Suppose, to the contrary, that there are a point $q\in M$ and a unit vector $v\in T_qM$ with $\operatorname{Ric}_g(v,v)<0$. Extend $v$ to a unit vector field $V$ on some neighborhood of $q$ and let $\omega$ be the corresponding dual $1$-form. By continuity, after shrinking this neighborhood we may assume that $\operatorname{Ric}_g(V,V)<0$ on it. Let $B=\operatorname{Ric}_g(V,V)\omega\otimes \omega$. It is clear that
$$
\langle B,\operatorname{Ric}_g\rangle_g=\operatorname{Ric}_g(V,V)^2>0
$$
on this neighborhood. Take a nonnegative cut-off function $\eta\not\equiv0$ with compact support in this neighborhood and define
$$
g_t = g - 2t\eta B.
$$
Since $B$ is negative semi-definite, $g_t\ge g$ for all $t\ge0$, and $g_t$ is complete because it agrees with $g$ outside a compact set. Write $A_\mathbb{S}=\mathrm{pr}_{\mathbb{S}^k}\circ F$. For any tangent vector $Z$,
$$
|dA_\mathbb{S}(Z)|_{g_{\mathbb{S}^k}}\leq|Z|_g\leq|Z|_{g_t}.
$$
Thus $A_\mathbb{S}$ remains $1$-Lipschitz with respect to $g_t$.

Choose a precompact connected smooth domain $D$ containing
$\mathrm{supp}\,\eta$ in its interior. Let $\mu(t)$ be the first Neumann
eigenvalue on $D$ of
$$
L_{g_t}=-\Delta_{g_t}+\frac14\left(R_M(g_t)-k(k-1)\right).
$$
At $t=0$, $L_{g_0}=-\Delta_g$, so $\mu(0)=0$, with constant eigenfunctions. This simple eigenvalue varies differentiably with $t$. Taking the derivative of the scalar curvature function with respect to $g_t$, from \cite[Theorem 1.174]{Bese} we obtain
$$
\left.\frac{\partial}{\partial t}\right|_{t=0} R_M(g_t)= 2\Delta_g \operatorname{tr}_g (\eta B)-2\operatorname{div}_g \operatorname{div}_g (\eta B)+2\eta\langle B, \operatorname{Ric}_g\rangle_g.
$$
By a calculation similar to that in the proof of \cite[Theorem 1.7]{LiChao}, in which the two divergence terms integrate to zero because $\operatorname{supp}\eta\Subset D$, we conclude that
\begin{align*}
\mu'(0)
&=\frac{1}{4\mathrm{Vol}_g(D)}\int_D\left.\frac{\partial}{\partial t}\right|_{t=0} R_M(g_t)\,d\mu_g\\
&=\frac{1}{2\mathrm{Vol}_g(D)}\int_D\eta\,\langle B,\operatorname{Ric}_g\rangle_g\,d\mu_g>0.
\end{align*}
Fix a small $t>0$ for which $\mu(t)>0$. Outside $D$, the metric is unchanged and $R_{g_t}=k(k-1)$. Thus the potential of $L_{g_t}$ is zero there. By Theorem \ref{thm-kazdan}, there is a smooth function $u$ with $L_{g_t}u>0$ on $M$. Once more, setting $\psi=2\log u$, we obtain
\begin{align*}
0<&-\frac{\Delta_{g_t}u}{u}+\frac{1}{4}(R_M({g_t})-k(k-1))\\
=&\frac{1}{4}\left(\mathcal R_{g_t}(\psi)-k(k-1)\right).
\end{align*}
Applying Proposition \ref{prop-noncompact-obstruction} to the new metric $g_t$ gives a contradiction. This completes the proof.
\end{proof}

\section{Proofs of the main theorems}\label{sec-main}
\subsection{Proof of Theorem~\ref{main}}\label{sec:proof-main}
In this subsection, we apply the Cheeger-Gromoll splitting theorem \cite{C-G} to prove Theorem \ref{main}.

We first show that $M$ has at least two ends.
\begin{lemma}\label{lem-two-ends}
\textit{
Under the hypotheses of Theorem \ref{main}, the manifold $M$ has at least two ends.}
\end{lemma}
\begin{proof}
By Lemma \ref{lem-degree-slice}, the function $b:M\to\mathbb{R}$ is proper and
surjective. Suppose to the contrary that $M$ has only one end. Since
$b^{-1}([-1,1])$ is compact, we may choose a compact set $K\subset M$
such that $b^{-1}([-1,1])\subset K$ and $M\setminus K$ is connected.

Since $K$ is compact, there exists $T>1$ such that
$|b(x)|<T$ for every $x\in K$. The surjectivity of $b$ provides points $x_-,x_+\in M\setminus K$ satisfying $b(x_-)=-T$ and $b(x_+)=T$. The set $M\setminus K$ is an open connected subset of a manifold and is therefore path connected. Hence there exists a continuous curve
$$
c:[0,1]\longrightarrow M\setminus K
$$
joining $x_-$ to $x_+$. By the intermediate value theorem, there is
some $s_0\in(0,1)$ such that $b(c(s_0))=0$. It follows that
$$
c(s_0)\in b^{-1}(0)\subset K,
$$
which leads to a contradiction. Hence $M$ has at least two ends.
\end{proof}

By Proposition \ref{prop-Ricci} and Lemma \ref{lem-two-ends}, $M$ has nonnegative Ricci curvature and contains a line. By the
Cheeger-Gromoll splitting theorem \cite{C-G}, there is an isometry
$$
(M,g)\cong(N^{n-1}\times\mathbb{R},g_N+ds^2).
$$
If $N$ were noncompact, then $N\times\mathbb R$ would have only one end. Hence $N$ is compact and connected.
The slice $N$ satisfies
$$
\operatorname{Ric}_{g_N}\geq0,\qquad R_{g_N}=k(k-1).
$$
We next transfer the topological information to this
slice.
\begin{lemma}\label{lem-split-degree}
\textit{
Under the hypotheses of Theorem \ref{main}, set $X=\mathbb{S}^k\times\mathbb{T}^m$. Regard $F=(A,b)$ as a map from $N\times\mathbb{R}$ to $X\times\mathbb{R}$ through
the splitting isometry. Then for every $s\in\mathbb{R}$ the map
$$
A_s=A|_{N\times\{s\}}:N\longrightarrow X
$$
has the same non-zero degree. Its spherical component is 1-Lipschitz.}
\end{lemma}

\begin{proof}
The maps $A_s$ are homotopic, so they all have the same degree.
Let $\Theta$ be a top form on $X$ such that $\int_{X}\Theta=1$, so $\mathrm{deg}(A_s)=\int_{N}A_s^*\Theta$. Consider the inclusion map
$$
i_s:N\longrightarrow N\times\mathbb{R},\quad i_s(x)=(x,s).
$$
The map $i_s$ is a homotopy equivalence, so
$$
i_s^*:H^{n-1}(N\times\mathbb{R};\mathbb{R})\longrightarrow H^{n-1}(N;\mathbb{R})
$$
is an isomorphism. On the other hand, since $A_s=A\circ i_s$, $i_s^*[A^*\Theta]=[A_s^*\Theta]$.

If $\mathrm{deg}(A_s)=0$, then
$[A_s^*\Theta]=0$. Consequently, $0=[A^*\Theta]\in H^{n-1}(N\times\mathbb{R};\mathbb{R})$. There exists an $(n-2)$-form $\beta$ on $N\times\mathbb{R}$ with $d\beta=A^*\Theta$. Choose $\zeta\in C_c^\infty(\mathbb{R})$ with integral one. The 1-form $\alpha=\zeta(b)\,db$ is closed and compactly supported because $b$ is proper. Hence $\beta\wedge\alpha$ is compactly supported, and
$$
\deg F=\int_{N\times\mathbb{R}}A^*\Theta\wedge\alpha
=\int_{N\times\mathbb{R}}d(\beta\wedge\alpha)=0,
$$
a contradiction. So $\mathrm{deg}(A_s)\neq0$. The spherical Lipschitz property follows by
restriction to the isometrically embedded slice.
\end{proof}

\begin{proof}[Proof of Theorem \ref{main}]
    We have proved that $(M,g)$ is isometric to $(N^{n-1}\times\mathbb{R},g_N+ds^2)$, where $N$ is a closed, connected, oriented manifold with $\operatorname{Ric}_N\ge0$ and $R_N=k(k-1)$. Lemma \ref{lem-split-degree} provides a smooth map
$$
A_0:N\longrightarrow\mathbb{S}^k\times\mathbb{T}^m
$$
with non-zero degree such that $\mathrm{pr}_{\mathbb{S}^k}\circ A_0$ is 1-Lipschitz. By Proposition \ref{prop:closed-rigidity}, which reduces to Theorem \ref{thm:WWXZ-Llarull} when $m=0$, $(M,g)$ is isometric to
$$
\left(
\mathbb{S}^k\times\mathbb{T}_\Lambda^m\times\mathbb{R},\,
g_{\mathbb{S}^k}+g_{\mathrm{flat}}+dt^2
\right),
$$
where $\Lambda\subset\mathbb{R}^m$ is a full lattice and
$\mathbb{T}_\Lambda^m=\mathbb{R}^m/\Lambda$ carries the induced
Euclidean metric $g_{\mathrm{flat}}$.
\end{proof}

\begin{proof}[Proof of Corollary \ref{cor:closed-rigidity}]
Put $d=k+m$, $a=a_d=(d+1)/(d+2)<1$ and $c=k(k-1)$. By Remark \ref{rem:datum}(d) and Theorem \ref{thm:datum-obstruction}, for every $\varphi\in C^\infty(N)$ we have
$$
\min_N\mathcal R^a_{g_N}(\varphi)\le c.
$$
Let $v>0$ be a smooth first eigenfunction of the self-adjoint operator
$$
L=-\frac4a\Delta_{g_N}+R_{g_N},\qquad Lv=\ell_1v.
$$
Taking $\varphi=(2/a)\log v$ gives
$$
\mathcal R^a_{g_N}(\varphi)=R_{g_N}-\frac4a\frac{\Delta_{g_N}v}{v}=\ell_1,
$$
so $\ell_1\le c$. On the other hand, set $u=e^{a\psi/2}>0$. The weighted curvature assumption gives
$$
\frac{Lu}{u}=\mathcal R^a_{g_N}(\psi)=\mathcal R_{g_N}(\psi)+(1-a)|\nabla_{g_N}\psi|^2\ge c.
$$
By self-adjointness and integration by parts on $N$,
$$
\begin{aligned}
(\ell_1-c)\int_Nuv\,dV_{g_N}
&=\int_Nv(Lu-cu)\,dV_{g_N}\\
&=\int_Nuv\bigl(\mathcal R_{g_N}(\psi)-c+(1-a)|\nabla_{g_N}\psi|^2\bigr)\,dV_{g_N}.
\end{aligned}
$$
The left hand side is nonpositive and the right hand side is nonnegative, so both vanish. Since $u,v>0$ and $1-a=1/(d+2)>0$, we obtain $\nabla_{g_N}\psi\equiv0$ and $\mathcal R_{g_N}(\psi)\equiv c$. Thus $\psi$ is constant and $R_{g_N}\equiv k(k-1)$.

Let $(M,g)=(N\times\mathbb R,\,g_N+dt^2)$ and define
$$
F:M\longrightarrow\mathbb S^k\times\mathbb T^m\times\mathbb R,\qquad F(x,t)=(A(x),t).
$$
Since $N$ is closed, $M$ is complete and $F$ is proper, with $\deg F=\deg A\ne0$ by the product formula. The spherical component $A_\mathbb S\circ\mathrm{pr}_N$ remains $1$-Lipschitz, and $R_M(g)=R_N=k(k-1)$. Proposition \ref{prop-Ricci} therefore gives $\operatorname{Ric}_g=\operatorname{Ric}_N\ge0$. Proposition \ref{prop:closed-rigidity} now yields the claimed product and the asserted form of $A_\mathbb S$.
\end{proof}

For $m=0$, the same spectral argument gives $\psi$ constant and $R_N\equiv k(k-1)$. Theorem \ref{thm:WWXZ-Llarull} then shows that $A_\mathbb S$ is an isometry onto the unit round sphere.

\subsection{Proof of Theorem~\ref{thm:infinity}}\label{sec:infinity}

The argument uses the weighted band width estimate of Proposition \ref{prop:bandwidth} and the Euclidean torical band construction in \cite[Section~3.4]{HSS}.

\begin{proof}[Proof of Theorem~\ref{thm:infinity}]
Set $X=\mathbb S^k\times\mathbb T^m$. Since $X$ is compact and $F$ is proper, $b$ is proper. Moreover, $F$ is surjective because its degree is non-zero, so $b$ is also surjective. Suppose, for a contradiction, that some $\psi\in C^\infty(M)$, compact set $K\Subset M$, and $\delta>0$ satisfy
\eqlabel{eq:infinity-gap}
$$
\mathcal R_g(\psi)\ge k(k-1)+\delta
\qquad\text{on }M\setminus K.
\leqno(\theequation)
$$
First let $Q\subset\mathbb R^r$ be a compact smooth domain diffeomorphic to $\mathbb T^{r-1}\times[-1,1]$, with $b$ transverse to $\partial Q$. Then $b^{-1}(Q)$ is compact and smooth. The restriction
$$
F|_{b^{-1}(Q)}:b^{-1}(Q)\longrightarrow X\times Q
$$
has total relative degree $\deg F$: a regular value in the interior has exactly the same signed preimages as for $F$. Additivity of degree therefore gives a connected component $W$ of non-zero relative degree. Such a component meets both boundary sides, by naturality of the boundary homomorphism in relative homology. After identifying $Q$ with a product and regrouping its torus factor, we obtain
\eqlabel{eq:infinity-pullback}
$$
G=(A,t):W\longrightarrow\mathbb S^k\times\mathbb T^{m+r-1}\times[-1,1],
\qquad \deg G\ne0.
\leqno(\theequation)
$$
Here $t$ is the collar coordinate of $Q$, affinely rescaled to take values in $[-1,1]$, composed with $b$. Thus $t^{-1}(\pm1)=\partial_\pm W$, and $\deg G$ is the degree of the proper map $G|_{\overset{\circ}W}$. Its spherical component is $A_{\mathbb S}|_W$. Whenever $Q\cap b(K)=\varnothing$, we have $W\subset M\setminus K$, and Proposition~\ref{prop:bandwidth}, applied with torus dimension $m+r-1$ and the same $\delta$, gives
\eqlabel{eq:infinity-width}
$$
\operatorname{width}(W,g)\le\pi\sqrt{\frac6\delta}.
\leqno(\theequation)
$$

\smallskip
\noindent\emph{The case $r=2$.}
Choose a positive regular value $a$ of $|b|$ with $b(K)\subset\{|y|<a\}$, and regular values $R_j>a$ tending to infinity. Sard's theorem applies away from $b^{-1}(0)$. The annuli
$$
Q_j=\{y\in\mathbb R^2:a\le|y|\le R_j\}
\cong\mathbb T^1\times[-1,1]
$$
give connected bands $W_j$ as in~\eqref{eq:infinity-pullback}. Put $S_t=b^{-1}(\{|y|=t\})$. These level sets are nonempty and compact. For every $T>0$, completeness and the Hopf-Rinow theorem imply that
$$
C_T=\{x\in M:d_g(x,S_a)\le T\}
$$
is compact. Hence $|b|$ is bounded on $C_T$, and $S_{R_j}\cap C_T=\varnothing$ for all sufficiently large $j$. It follows that
$$
\operatorname{width}(W_j,g)
\ge\operatorname{dist}_M(S_a,S_{R_j})\longrightarrow\infty,
$$
contrary to~\eqref{eq:infinity-width}.

\smallskip
\noindent\emph{The case $r\ge3$.}
Choose a smooth embedding
$$
\tau:\mathbb T^{r-1}\times[-2,2]\longrightarrow\mathbb R^r.
$$
Such an embedding can be constructed inductively: start with $\mathbb S^1\subset\mathbb R^2$. An oriented embedded $\mathbb T^j\subset\mathbb R^{j+1}$ has trivial rank two normal bundle in $\mathbb R^{j+2}$, so the boundary of a small tubular neighborhood is an embedded $\mathbb T^{j+1}$. A two-sided collar of the final hypersurface gives $\tau$. The disjoint compact sets
$$
E_-=\tau(\mathbb T^{r-1}\times[-1,-\tfrac12]),
\qquad
E_+=\tau(\mathbb T^{r-1}\times[\tfrac12,1])
$$
have positive Euclidean distance $d_*$.

Take $\lambda_j\to\infty$ and translations $v_j$ such that the images of
$$
\tau_j(z,t)=v_j+\lambda_j\tau(z,t)
$$
avoid $b(K)$. Let $u_j$ be the collar coordinate on the interior of the image of $\tau_j$. By Sard's theorem, choose regular values $t_j^-\in(-1,-\tfrac12)$ and $t_j^+\in(\tfrac12,1)$ of $u_j\circ b$ on its domain. Then
$$
Q_j=\tau_j(\mathbb T^{r-1}\times[t_j^-,t_j^+])
$$
is disjoint from $b(K)$, its boundary is transverse to $b$, and
$$
\operatorname{dist}_{\mathbb R^r}(\partial_-Q_j,\partial_+Q_j)
\ge\lambda_jd_*.
$$
Select a connected band $W_j$ of non-zero relative degree as above. For $x\in\partial_-W_j$ and $y\in\partial_+W_j$, condition~\eqref{eq:infinity-coarse} gives
$$
\lambda_jd_*\le |b(x)-b(y)|
\le Ld_g(x,y)+C_0
\le Ld_{W_j}(x,y)+C_0.
$$
Consequently,
$$
\operatorname{width}(W_j,g)
\ge\frac{\lambda_jd_*-C_0}{L}\longrightarrow\infty,
$$
again contradicting~\eqref{eq:infinity-width}. Both cases exclude~\eqref{eq:infinity-gap}, which proves the theorem.
\end{proof}

\begin{remark}\label{rem:r1}
The same argument applies when $r=1$, without any condition on $b$. Indeed, choose a regular value $a$ of $b$ with $b(K)\subset(-\infty,a)$ and regular values $R_j\to\infty$, and use the bands of non-zero relative degree in $b^{-1}([a,R_j])$, which correspond to $Q_j=[a,R_j]\cong\mathbb T^0\times[-1,1]$. Their widths tend to infinity as in the case $r=2$, and Proposition \ref{prop:bandwidth} applies with $j=m\ge0$. Hence, under the hypotheses of Proposition \ref{prop-noncompact-obstruction}, every $\psi\in C^\infty(M)$ satisfies $\liminf_{x\to\infty}\mathcal R_g(\psi)(x)\le k(k-1)$.
\end{remark}

\appendix
\section{\texorpdfstring{Warped $\mu$-bubbles in bands}{Warped mu-bubbles in bands}}\label{app:bubbles}
In this appendix we prove Proposition \ref{prop:bubble}, following \cite[Proposition 4.6]{BZ2026}. We keep the notation of Section \ref{sec:bubbles}. For a set $\Omega$ of locally finite perimeter, an open set $U$ and a measurable set $B$ we write
$$
\operatorname{Per}_\psi(\Omega;U)=\int_{\partial^\ast\Omega\cap U}e^\psi\,d\mathcal H^{d-1},\qquad \operatorname{Vol}_\psi(B)=\int_Be^\psi\,d\mathcal H^d.
$$
On precompact sets, $e^\psi$ is bounded above and below by positive constants, so these quantities are comparable with the unweighted perimeter and volume.

\subsection{End barriers}
\begin{lemma}\label{lem:end-barrier}
\textit{Let $p\in\partial_-W$. There are an open neighborhood $V_p$ of $p$ in $P$ and a compact set $R_p\subset W\cap\{s<0\}$, depending only on $g$, $\psi$, $s$ and $h$ near $p$, such that $V_p\cap W\subset R_p$ and the following holds: if $U\supset R_p$ is a precompact open set and $\Omega\subset P$ is a set of locally finite perimeter with $\Omega\triangle\Omega_0\subset W$ and $\mathcal A_{\widehat h}(\Omega;U)<\infty$, then
$$
\mathcal A_{\widehat h}(\Omega\cup R_p;U)\le\mathcal A_{\widehat h}(\Omega;U)-\operatorname{Vol}_\psi(R_p\setminus\Omega).
$$
Similarly, for $p\in\partial_+W$ there are $V_p$ and a compact set $R_p\subset W\cap\{s>0\}$ with $V_p\cap W\subset R_p$ and $\mathcal A_{\widehat h}(\Omega\setminus R_p;U)\le\mathcal A_{\widehat h}(\Omega;U)-\operatorname{Vol}_\psi(\Omega\cap R_p)$.}
\end{lemma}
\begin{proof}
Let $p\in\partial_-W$. The other case is symmetric. Since $-L$ is a regular value of $s$, the flow of $\nabla s/|\nabla s|^2$ starting from $\partial_-W$ gives coordinates $(y,t)\in\{|y|<3r\}\times(-L-\varepsilon_1,-L+\varepsilon_1)$ near $p$ in which $s=t$. Let $\varpi(y)=1-|y|^2/(4r^2)$, fix $\varsigma_0'\in(0,\min\{\varepsilon_1,L\})$, and for $\varsigma_0\in(0,\varsigma_0']$ let
$$
R_p=\left\{(y,t):|y|\le2r,\ -L\le t\le-L+\varsigma_0\varpi(y)\right\}.
$$
Its edge $\mathcal E_p=\{|y|=2r,\ t=-L\}$ has codimension two. On $\{|y|<2r\}$ the function $\varsigma=(t+L)/\varpi(y)$ is smooth with $\partial_t\varsigma\ne0$, so the leaves $\Lambda_\varsigma=\{t=-L+\varsigma\varpi(y)\}$, where $\varsigma$ may also take small negative values, form a smooth foliation of a neighborhood of $R_p\setminus\mathcal E_p$. Let $\mathbf N$ be the unit normal of the leaves pointing in the direction of increasing $\varsigma$, and $Z=e^\psi\mathbf N$. Then $\operatorname{div}Z=e^\psi(H_{\Lambda_\varsigma}+\langle\nabla\psi,\mathbf N\rangle)$, where $H_{\Lambda_\varsigma}=\operatorname{div}\mathbf N$ is the mean curvature of the leaves. Since the leaves are graphs of functions $-L+\varsigma\varpi$ with uniformly bounded $C^2$ norms, there is a constant $C_p$, depending only on $\varsigma_0'$, such that $|\operatorname{div}Z|\le C_pe^\psi$ on $R_p\setminus\mathcal E_p$ for every $\varsigma_0\le\varsigma_0'$. Having fixed $C_p$, we choose $\varsigma_0\le\varsigma_0'$ so small that $\widehat h\ge C_p+1$ on $R_p\cap\overset{\circ}{W}$. This is possible since $h(t)\to+\infty$ as $t\downarrow-L$. Put $V_p=\{|y|<r,\ -L-\varepsilon_1<t<-L+\frac34\varsigma_0\}$. Since $\varpi>\frac34$ on $\{|y|<r\}$, we have $V_p\cap W\subset R_p$.

Let $R_p'=R_p\setminus\Omega$ and $\Omega'=\Omega\cup R_p=\Omega\sqcup R_p'$. For $\mathcal H^{d-1}$-almost every $x\in\partial^\ast R'_p$, either $x\in\partial^\ast\Omega$, in which case the outer normals satisfy $\nu_{R_p'}=-\nu_\Omega$, or $x$ lies on the top leaf $\Lambda_{\varsigma_0}$, in which case $\nu_{R_p'}=\mathbf N$. Indeed, in the interior of $R_p$ we have $\chi_{R_p'}=1-\chi_\Omega$. On the bottom $\{t=-L\}$ the outside $\{t<-L\}=\{s<-L\}$ is contained in $\Omega$, because $\Omega\triangle\Omega_0\subset W$, so points of $\partial^\ast R'_p$ on the bottom also belong to $\partial^\ast\Omega$. Moreover, $\mathcal E_p$ is $\mathcal H^{d-1}$-negligible. Let $\mathcal T=(\partial^\ast R'_p\cap\Lambda_{\varsigma_0})\setminus\partial^\ast\Omega$ and $\mathcal J=\partial^\ast R'_p\cap\partial^\ast\Omega$. The divergence theorem for $Z$ on $R'_p$ gives
$$
C_p\int_{R'_p}e^\psi\ge\int_{R'_p}\operatorname{div}Z=\int_{\partial^\ast R'_p}\langle Z,\nu_{R'_p}\rangle\ge\int_{\mathcal T}e^\psi-\int_{\mathcal J}e^\psi.
$$
Here $Z$ is bounded and smooth outside $\mathcal E_p$, and we may apply the divergence theorem to $\chi_\epsilon Z$, where $\chi_\epsilon$ is a cut-off function vanishing near $\mathcal E_p$ with $|\nabla\chi_\epsilon|\le C/\epsilon$ supported in the $2\epsilon$-neighborhood of $\mathcal E_p$, whose volume is $O(\epsilon^2)$. The error tends to zero as $\epsilon\to0$. Since $\Omega$ and $R'_p$ are disjoint, $\operatorname{Per}_\psi(\Omega';U)=\operatorname{Per}_\psi(\Omega;U)+\operatorname{Per}_\psi(R'_p;U)-2\int_{\mathcal J}e^\psi$, see \cite[Chapter 16]{Maggi}, and $\operatorname{Per}_\psi(R'_p;U)=\int_{\mathcal T}e^\psi+\int_{\mathcal J}e^\psi$. Hence the perimeter term increases by $\int_{\mathcal T}e^\psi-\int_{\mathcal J}e^\psi\le C_p\operatorname{Vol}_\psi(R'_p)$, while the volume term changes by $-\int_{R'_p}\widehat he^\psi\le-(C_p+1)\operatorname{Vol}_\psi(R'_p)$. Adding these estimates proves the claim.
\end{proof}

\subsection{Existence}\label{app:existence}
We first show that the functional \eqref{eqn-omega} is well defined. Choose $\delta'\in(0,L)$ such that $h>0$ on $(-L,-L+\delta']$ and $h<0$ on $[L-\delta',L)$. Since $\Omega_0$ contains $\{s<0\}$ and is disjoint from $\{s>0\}$, the integrand $-(\chi_\Omega-\chi_{\Omega_0})\widehat h$ is nonnegative on $\overset{\circ}W\cap\{|s|>L-\delta'\}$ for every $\Omega$ with $\Omega\triangle\Omega_0\subset W$, while $|\widehat h|\le C_{\delta'}$ on $W\cap\{|s|\le L-\delta'\}$. Hence $\mathcal A_{\widehat h}(\Omega;U)\in(-\infty,+\infty]$, and
\begin{equation}\label{eq:lower-bound}
\mathcal A_{\widehat h}(\Omega;U)\ge-C_{\delta'}\operatorname{Vol}_\psi(U\cap W).
\end{equation}
If $W$ is compact, let $K_i=W$ and let $U_i=U$ be a fixed precompact open neighborhood of $W$ for all $i$. Otherwise, $(P,g)$ is complete by \eqref{eq:band-assumption}, so we choose a smooth proper function $\vartheta:P\to[0,\infty)$ and put $K_i=W\cap\{\vartheta\le i\}$ and $U_i=\{\vartheta<i+1\}$.

\emph{(a) Minimization.} Let $\mathcal C_i$ be the class of sets $\Omega\subset P$ of locally finite perimeter with $\Omega\triangle\Omega_0\subset K_i$, and minimize $\mathcal F_i=\mathcal A_{\widehat h}(\,\cdot\,;U_i)$ over $\mathcal C_i$. Competitors agree with $\Omega_0$ outside $K_i$, so on the artificial sides $\partial K_i\cap\overset{\circ}W$ a Dirichlet condition is imposed, and no free boundary appears. Since $0$ is a regular value of $s$ and $U_i$ is precompact, $\mathcal F_i(\Omega_0)=\int_{\{s=0\}\cap U_i}e^\psi<\infty$, and $\mathcal F_i$ is bounded from below by \eqref{eq:lower-bound}.

\emph{(b) Direct method.} By \eqref{eq:lower-bound}, a minimizing sequence has uniformly bounded perimeter in $U_i$. After passing to a subsequence, it converges in $L^1_{\mathrm{loc}}(U_i)$. Since $\Omega\triangle\Omega_0\subset K_i\Subset U_i$, this is $L^1$ convergence of $\chi_\Omega-\chi_{\Omega_0}$. The class $\mathcal C_i$ is closed under this convergence, the perimeter in the open set $U_i$ is lower semicontinuous, and the volume term is lower semicontinuous, by Fatou's lemma near the ends, where the integrand is nonnegative, and by dominated convergence in the middle. Hence there is a minimizer $\Omega_i\in\mathcal C_i$ with $\mathcal F_i(\Omega_i)<\infty$.

\emph{(c) Uniform end barriers.} Let $p\in\partial_-W$, and let $V_p$ and $R_p$ be as in Lemma \ref{lem:end-barrier}. If $i$ is so large that $R_p\subset K_i$, then $(\Omega_i\cup R_p)\triangle\Omega_0=(\Omega_i\triangle\Omega_0)\setminus R_p\subset K_i$, because $R_p\subset\Omega_0$. Thus $\Omega_i\cup R_p\in\mathcal C_i$. By minimality and Lemma \ref{lem:end-barrier}, $\operatorname{Vol}_\psi(R_p\setminus\Omega_i)=0$. Since moreover $\Omega_i\supset\{s<-L\}$, the normalized set $\Omega_i$ contains $V_p$, and $V_p\cap\partial\Omega_i=\varnothing$. For $p\in\partial_+W$, the same argument gives $V_p\cap\Omega_i=\varnothing$. The neighborhoods $V_p$ do not depend on $i$.

\emph{(d) Uniform local perimeter bounds.} Fix $l$. Cover the compact set $\partial W\cap\overline{U_l}$ by finitely many of the sets $V_p$. The remaining part of $U_l\cap W$ is contained in a compact subset of $\overset{\circ}W$, on a neighborhood of which $|\widehat h|\le\Lambda_l$. For a small geodesic ball $B$ in this neighborhood, with $B\subset K_{l+2}$, the set $\Omega_i\cup B$ belongs to $\mathcal C_i$ for $i\ge l+2$, and comparison gives $\operatorname{Per}_\psi(\Omega_i;B)\le\operatorname{Per}_\psi(B)+\Lambda_l\operatorname{Vol}_\psi(B)$. In the sets $V_p$, the set $\Omega_i$ coincides with $\Omega_0$ by (c), and outside $W$ it coincides with $\Omega_0$, which is locally constant there. Hence there is $i_l$ such that $\operatorname{Per}(\Omega_i;U_l)\le c_l$ for all $i\ge i_l$.

\emph{(e) The limit.} By (d) and a diagonal argument, a subsequence of $\Omega_i$ converges in $L^1_{\mathrm{loc}}(P)$ to a set $\Omega$ of locally finite perimeter with $\Omega\triangle\Omega_0\subset W$, and the conclusion of (c) holds for $\Omega$. Let $U$ be precompact and open, and let $\Omega'\triangle\Omega$ be contained in a compact set $Q\subset U\cap\overset{\circ}W$. Choose $\varrho_0>0$ such that the closed $2\varrho_0$-neighborhood of $Q$ lies in $U\cap\overset{\circ}W$, and let $W_\varrho=\{\operatorname{dist}(\cdot,Q)<\varrho\}$ for $\varrho\in(\varrho_0,2\varrho_0)$. By the coarea formula and Fatou's lemma, for almost every $\varrho$ we have $\liminf_i\int_{\partial W_\varrho\cap(\Omega_i\triangle\Omega)}e^\psi d\mathcal H^{d-1}=0$, and $\partial W_\varrho$ carries no perimeter measure of $\Omega$ or of the $\Omega_i$. We fix such a $\varrho$, independent of $i$, and pass to the corresponding subsequence. For large $i$ we have $W_\varrho\subset K_i$, and the set which equals $\Omega'$ in $W_\varrho$ and $\Omega_i$ outside $W_\varrho$ belongs to $\mathcal C_i$. Comparing it with $\Omega_i$, and using the lower semicontinuity of the perimeter and the boundedness of $\widehat h$ on $W_\varrho$, we obtain $\mathcal A_{\widehat h}(\Omega;U)\le\mathcal A_{\widehat h}(\Omega';U)$ in the limit. This is the standard argument showing that $(\Lambda,r_0)$-minimizers are closed under $L^1$ convergence, see \cite[Chapter 21]{Maggi}. Hence $\Omega$ is a warped $\mu$-bubble in $W$. Note that the artificial sides $\partial K_i$ leave every compact set as $i\to\infty$, so that only interior regularity is needed for the limit.

By (c), $\partial\Omega$ is disjoint from $\partial W$. Since $\Omega$ coincides with $\Omega_0$ outside $W$, where $\Omega_0$ is locally constant, we have $\Sigma=\partial\Omega=\partial(\Omega\cap W)\cap\overset{\circ}W$. Thus $\Sigma$ is closed in $P$ and contained in $\overset{\circ}W$, and it is compact if $W$ is compact. If $(P,g)$ is complete, the closed subset $\Sigma$ is complete with respect to $d_g$. This proves Proposition \ref{prop:bubble}(1).

\begin{remark}\label{rem:end-jump}
It is essential that the perimeter in \eqref{eqn-omega} is computed in open subsets of $P$, so that a jump of $\Omega$ at $\partial W$ is counted. If the perimeter were only computed in $\overset{\circ}W$ and competitors were only required to satisfy $\Omega\triangle\Omega_0\subset K_i$, the behavior at the ends would not be fixed, and a potential which tends to $\pm\infty$ at the ends would not prevent a minimizer from being empty or full. For example, let $P=\mathbb R\times\mathbb T^{d-1}$, $W=[-1,1]\times\mathbb T^{d-1}$, $s=t$, $\psi=0$ and $h(t)=-\epsilon t/\sqrt{1-t^2}$ with $0<\epsilon<1$. The vector field $Z=\epsilon\sqrt{1-t^2}\,\partial_t$ satisfies $\operatorname{div}Z=\widehat h$ and $|Z|\le\epsilon$, and it vanishes on $\partial W$. By the divergence theorem, the functional with perimeter computed in $\overset{\circ}W$ is at least $(1-\epsilon)\operatorname{Per}(\Omega;\overset{\circ}W)+\epsilon\operatorname{Area}(\mathbb T^{d-1})$, and this lower bound is attained both by the empty set and by $W$. In \eqref{eqn-omega}, the empty set in $W$ corresponds to $\Omega=\{s<-1\}$, and counting its jump at $\partial_-W$ gives $\mathcal A_{\widehat h}=(1+\epsilon)\operatorname{Area}(\mathbb T^{d-1})>\operatorname{Area}(\mathbb T^{d-1})=\mathcal A_{\widehat h}(\Omega_0)$. In general, Lemma \ref{lem:end-barrier} shows that minimizers always have the correct behavior at the ends. The barrier argument only uses that $h\to\pm\infty$ at the ends.
\end{remark}

\subsection{Local structure}\label{app:local}
The remaining assertions of Proposition \ref{prop:bubble} are local in $\overset{\circ}W$, and require only that $\widehat h$ is bounded on compact subsets of $\overset{\circ}W$. Let $K\subset\overset{\circ}W$ be compact, and let $K'\subset\overset{\circ}W$ be a compact neighborhood of $K$. We use three descriptions of $\Omega$ near $K$.

\emph{(a) An almost minimizing boundary for a conformal metric.} Let $\check g=e^{2\psi/(d-1)}g$ and $\check h=\widehat he^{-\psi/(d-1)}$. The area element of a hypersurface is multiplied by $e^\psi$ and the volume element by $e^{d\psi/(d-1)}$, so
$$
\int_{\partial^\ast\Omega\cap U}e^\psi d\mathcal H^{d-1}_g=\operatorname{Per}_{\check g}(\Omega;U),\qquad\int_U(\chi_\Omega-\chi_{\Omega_0})\widehat he^\psi d\mathcal H^d_g=\int_U(\chi_\Omega-\chi_{\Omega_0})\check h\,d\mathcal H^d_{\check g}.
$$
Let $\Lambda_K=\sup_{K'}|\check h|$, and let $r_K>0$ be smaller than the $\check g$-distance from $K$ to $\partial K'$ and the injectivity radius of $\check g$ on $K'$. By the minimizing property, for every $\check g$-ball $B$ with center in $K$ and radius less than $r_K$ and every $\Omega'$ with $\Omega\triangle\Omega'\Subset B$,
$$
\operatorname{Per}_{\check g}(\Omega;B)\le\operatorname{Per}_{\check g}(\Omega';B)+\Lambda_K\operatorname{Vol}_{\check g}(\Omega\triangle\Omega').
$$
Thus $\Omega$ is a $(\Lambda_K,r_K)$-almost minimizing set near $K$, with constants uniform in the center.

\emph{(b) An integral varifold with bounded mean curvature.} Let $Z$ be a $C^1$ vector field with compact support in $\overset{\circ}W$. The first variation of $\mathcal A_{\widehat h}$ along the flow of $e^{-\psi}Z$ vanishes, which gives
$$
\int_{\partial^\ast\Omega}\operatorname{div}_\Sigma Z\,d\mathcal H^{d-1}=\int_{\partial^\ast\Omega}\left(\widehat h-\langle\nabla\psi,\nu\rangle\right)\langle Z,\nu\rangle\,d\mathcal H^{d-1},
$$
where $\nu$ is the outer normal of $\Omega$. Here $Z$ is an ambient vector field, which may cross the singular set. Hence the multiplicity one integral varifold $\mathbf V=\mathbf v(\partial^\ast\Omega,1)$ satisfies $|\delta\mathbf V(Z)|\le\Lambda'_K\int|Z|\,d\|\mathbf V\|$ for $Z$ supported in $K'$, where $\Lambda'_K=\sup_{K'}(|\widehat h|+|\nabla\psi|)$. Thus its generalized mean curvature is bounded and its first variation has no singular part, as required by the monotonicity formula, Allard's regularity theorem and \cite{NV}. The equation $H=-\langle\nabla\psi,\nu\rangle+\widehat h$ on the regular part alone would not suffice.

\emph{(c) Compactness of blow-ups.} Let $y_k\in\Sigma\cap K'$ and $s_k\to0$. In $g$-normal coordinates at $y_k$, rescale $B^g_{Rs_k}(y_k)$ by the factor $s_k^{-1}$ and freeze the weight at $y_k$, that is, divide by $e^{\psi(y_k)}$. The rescaled metrics converge in $C^2$ to the Euclidean metric on $B_R$. On $B^g_{Rs_k}(y_k)$ the weight $e^\psi$ differs from $e^{\psi(y_k)}$ by a factor $1+O(s_k)$, and the potential term becomes $O(s_k)|B|$. Hence the rescaled sets $\Omega_k$ satisfy
$$
\operatorname{Per}(\Omega_k;B)\le(1+\eta_k)\operatorname{Per}(\Omega';B)+\eta_k|B|\qquad\text{whenever }\Omega'\triangle\Omega_k\Subset B\subset B_R,
$$
with $\eta_k\to0$, for the Euclidean perimeter and volume. The argument of \cite[Chapter 21]{Maggi} for the closedness of $(\Lambda,r_0)$-minimizers applies: after passing to a subsequence, $\Omega_k\to\Omega_\infty$ in $L^1_{\mathrm{loc}}(\mathbb R^d)$, where $\Omega_\infty$ is a perimeter minimizer in $\mathbb R^d$ with $0\in\partial\Omega_\infty$, and the perimeter measures converge weakly, $|D\chi_{\Omega_k}|\rightharpoonup|D\chi_{\Omega_\infty}|$. By Reshetnyak's continuity theorem, the corresponding varifolds converge to the multiplicity one varifold $\mathbf v(\partial^\ast\Omega_\infty,1)$, and by the density lower bound supplied by the monotonicity formula, $\partial\Omega_k\to\partial\Omega_\infty$ in the Hausdorff sense on compact sets.

From (a)--(c) we obtain the following properties, with constants depending only on $\Lambda_K$, $\Lambda'_K$, $r_K$ and the geometry of $g$ and $\psi$ on $K'$, and not on the specific form of $h$.
\begin{itemize}
\item \emph{Monotonicity and Ahlfors regularity.} The monotonicity formula holds by (b) and \cite{Allard}. The upper area bound follows from (a) by comparison with geodesic balls, and the lower bound from monotonicity and the fact that the density is at least one on $\Sigma$. Since $\check g$ and $g$ are bi-Lipschitz equivalent on $K'$ and $\mathcal H^{d-1}_{\check g}=e^\psi\mathcal H^{d-1}_g$, local Ahlfors $(d-1)$-regularity holds for both geometries.
\item \emph{Regularity.} At every point $x\in\Sigma$, the blow up limits of (c) are area-minimizing boundaries in Euclidean space, and the varifolds converge with multiplicity one. Hence all tangent cones are area-minimizing hypercones. Allard's regularity theorem \cite{Allard}, in the Riemannian case via an isometric embedding into a Euclidean space which only adds a bounded term to the mean curvature, gives $\epsilon$-regularity. Federer's dimension reduction \cite{Simon}, together with Simons' theorem \cite{Simons}, shows that area-minimizing hypercones in $\mathbb R^q$ with $q\le7$ are hyperplanes. Hence $\Sigma_{\mathrm{reg}}$ is a $C^{1,\alpha}$ hypersurface, which is smooth by the equation $H=-\langle\nabla\psi,\nu\rangle+\widehat h$. The set $D_\Sigma=\Sigma\setminus\Sigma_{\mathrm{reg}}$ is closed, $\dim_{\mathcal H}D_\Sigma\le d-8$, and $D_\Sigma=\varnothing$ if $d\le7$. For the Euclidean theory see \cite{Tamanini} and \cite[Part III]{Maggi}.
\item \emph{Variations.} On $\Sigma_{\mathrm{reg}}$ the first variation gives $H=-\langle\nabla\psi,\nu\rangle+\widehat h$. If $\varphi\in C^\infty_c(\Sigma_{\mathrm{reg}})$, normal deformations of $\Omega$ with speed $\varphi$ are admissible competitors, since they are supported in a compact subset of $\overset{\circ}W\setminus D_\Sigma$. Hence the second variation along these deformations is nonnegative, and the computation of \cite[Lemma 14]{Cho-Li} gives \eqref{eq:stability-bubble}.
\end{itemize}
This proves Proposition \ref{prop:bubble}(2) and (3). The analytic structure in Proposition \ref{prop:bubble}(4), established in the next two subsections, is only needed when $D_\Sigma\ne\varnothing$, which can only happen if $d\ge8$.

\subsection{A local Poincar\'e inequality}
\begin{lemma}\label{lem:poincare}
\textit{Let $m=d-1$, $\mu=\mathcal H^m\llcorner\Sigma$, so that $\mu(D_\Sigma)=0$, and $B_r(x)=\{y\in\Sigma:d_g(x,y)<r\}$. Let $\beta=\beta_d\in(0,1)$ be a constant for which \cite[Theorem 2]{BG}, recalled in Step 1 of the proof below, holds in $\mathbb R^d$. For every compact set $K\subset\Sigma$ there are $C_K\ge1$ and $r_K>0$ such that for all $x\in K$, $0<r<r_K$, and all functions $\varphi$ on $B_{2r/\beta}(x)$ which are Lipschitz with respect to $d_g$, we have, writing $B_\rho=B_\rho(x)$,
$$
\frac1{\mu(B_r)}\int_{B_r}|\varphi-\varphi_{B_r}|\,d\mu\le C_Kr\left(\frac1{\mu(B_{2r/\beta})}\int_{B_{2r/\beta}\cap\Sigma_{\mathrm{reg}}}|\nabla_\Sigma\varphi|^2\,d\mu\right)^{1/2}.
$$
This is the Poincar\'e inequality \cite[(3.2)]{BZ2026} for the space $(\Sigma,d_g|_\Sigma)$, with regular part $\Sigma_{\mathrm{reg}}$, singular part $D_\Sigma$, weight $\omega\equiv1$ and dilation factor $2/\beta$, which is the constant $\lambda_K$ of \cite[Definition 3.1]{BZ2026}.}
\end{lemma}
\begin{proof}
Let $K_1$ be a compact neighborhood of $K$ in $\Sigma$. We apply Subsection \ref{app:local} to a compact subset of $\overset{\circ}W$ containing $K_1$, and choose $s_A>0$ such that $c_As^m\le\mu(B_s(y))\le C_As^m$ for $y\in K_1$ and $0<s<s_A$. Then $C_1=2^mC_A/c_A$ is a doubling constant. For a measurable set $D\subset\Sigma$ let $\theta_D(y,s)=\mu(D\cap B_s(y))/\mu(B_s(y))$, and let $P_\Sigma(D;U)$ be the perimeter of $D\cap\Sigma_{\mathrm{reg}}$ relative to the Riemannian manifold $\Sigma_{\mathrm{reg}}\cap U$.

\emph{Step 0: relative perimeter as boundary mass.} Let $U$ be an oriented coordinate neighborhood in $P$ with $P_\Sigma(D;U)<\infty$, and orient $\Sigma_{\mathrm{reg}}\cap U$ by $\nu$. Then in $U$ the boundary of the current $[\Sigma]\llcorner D$ is the current of integration over the reduced boundary $\partial^\ast_\Sigma D$ of $D$ in $\Sigma_{\mathrm{reg}}$, with its induced orientation, and its mass is $P_\Sigma(D;U)$. Indeed, let $\omega$ be a smooth $(m-1)$-form with compact support in $U$. Since $\mathcal H^{m-1}(D_\Sigma)=0$, we can cover $D_\Sigma\cap\operatorname{spt}\omega$ by finitely many balls $B_{r_i}$ centered on $D_\Sigma$ with $\sum_ir_i^{m-1}<\epsilon$. Let $\xi_i$ vanish on $B_{r_i}$, be equal to one outside $B_{2r_i}$, and satisfy $|\nabla\xi_i|\le1/r_i$, and let $\chi_\epsilon=\min_i\xi_i$. By the Ahlfors upper bound, $\int_\Sigma|\nabla\chi_\epsilon|\,d\mu\le\sum_i\mu(B_{2r_i})/r_i\le C\epsilon$, and $\chi_\epsilon\to1$ pointwise on $\Sigma_{\mathrm{reg}}$, since every regular point has positive distance from $D_\Sigma$. Hence
$$
\int_Dd\omega=\lim_{\epsilon\to0}\int_Dd(\chi_\epsilon\omega)=\lim_{\epsilon\to0}\int_{\partial^\ast_\Sigma D}\chi_\epsilon\omega=\int_{\partial^\ast_\Sigma D}\omega,
$$
where the first equality uses $\int|\nabla\chi_\epsilon|\,d\mu\to0$ and the second is the Gauss-Green formula for sets of finite perimeter in $\Sigma_{\mathrm{reg}}$, since $\chi_\epsilon\omega$ has compact support in $\Sigma_{\mathrm{reg}}$. The same holds for area-minimizing boundaries in $\mathbb R^d$.

\emph{Step 1: the balanced case.} We use \cite[Theorem 2]{BG}, which we recall with the hypersurface dimension $m$ and the ambient space $\mathbb R^{m+1}$. What \cite{BG} call an oriented boundary of least area is an oriented area-minimizing boundary in our terminology: \emph{let $S$ be an oriented area-minimizing boundary in the open ball $B_r\subset\mathbb R^{m+1}$ of radius $r$ centered at $0$, with $\partial S=0$ in $B_r$, and let $S_1,S_2\in\mathcal R^{\mathrm{loc}}_m(B_r)$ satisfy $S=S_1+S_2$ and $\|S\|=\|S_1\|+\|S_2\|$. Then}
$$
\mathbf M(\partial S_1)=\mathbf M(\partial S_2)\ge\frac{1}{2\gamma}\min\left\{\mathbf M(S_1\llcorner B_{\beta r}),\mathbf M(S_2\llcorner B_{\beta r})\right\}^{\frac{m-1}m},
$$
\emph{where $\gamma$ is the constant in the isoperimetric inequality $\mathbf M(S')^{(m-1)/m}\le\gamma\,\mathbf M(\partial S')$ for minimal $m$ dimensional integral currents $S'$ with compact support in $\mathbb R^{m+1}$, and $\beta>0$ depends only on $m$.} We apply this theorem only to genuine area-minimizing boundaries in $\mathbb R^d$, and only when the left hand side vanishes. We claim that for every $\eta\in(0,\frac12]$ there are $C_\eta$ and $s_\eta>0$ such that for all $y\in K_1$, $0<s<s_\eta$, and all $D$ with $\eta\le\theta_D(y,s)\le1-\eta$,
\begin{equation}\label{eq:balanced}
s^{m-1}\le C_\eta P_\Sigma\left(D;B_{s/\beta}(y)\right).
\end{equation}
Suppose not. Then there are $y_k\in K_1$, $s_k\to0$ and sets $D_k$ with $\eta\le\theta_{D_k}(y_k,s_k)\le1-\eta$ and $s_k^{1-m}P_\Sigma(D_k;B_{s_k/\beta}(y_k))\to0$. Rescale at $y_k$ by $s_k^{-1}$ as in Subsection \ref{app:local}(c), and pass to a subsequence with $\Omega_k\to\Omega_\infty$. Put $T=\partial\Omega_\infty$, so that $0\in T$. Let $\Sigma_k=\partial\Omega_k$ be the rescaled bubbles and $W_*=B_{1/\beta}(0)$. In normal coordinates at $y_k$, the $d_g$-balls centered at $y_k$ are Euclidean balls. Let $S_k=[\Sigma_k]\llcorner D_k$ and $R_k=[\Sigma_k]\llcorner(\Sigma_k\setminus D_k)$ in $W_*$. Masses of currents in $W_*$ are computed with respect to the Euclidean metric of the normal coordinates, whereas the relative perimeters $P_{\Sigma_k}$ and the measures $\mu_k=\mathcal H^m\llcorner\Sigma_k$ refer to the rescaled metrics $g_k$. Since $g_k$ converges to the Euclidean metric in $C^2$ on $W_*$, the two notions of $m$ and $(m-1)$ dimensional area differ by factors which tend to one as $k\to\infty$, uniformly over all subsets of $W_*$. By the Ahlfors upper bound, $\mathbf M(S_k)+\mathbf M(R_k)=\|\Sigma_k\|(W_*)\le C$. By Step 0, $\mathbf M(\partial S_k\llcorner W_*)\le(1+o(1))P_{\Sigma_k}(D_k;W_*)\to0$. Also, $\partial[\Sigma_k]=\partial\partial[\Omega_k]=0$, so $\partial R_k=-\partial S_k$. By the local version of the Federer-Fleming compactness theorem in the open set $W_*$ \cite[4.2.17]{Federer}, after passing to a further subsequence $S_k\to S$ and $R_k\to R$ weakly in $W_*$, where $S$ and $R$ are integer rectifiable, $\partial S=0$ in $W_*$, and $S+R=\lim\partial[\Omega_k]=\partial[\Omega_\infty]=[T]$. For every open set $W''\subset W_*$ with $\|T\|(\partial W'')=0$, the lower semicontinuity of mass and the convergence of perimeter measures in Subsection \ref{app:local}(c) give
\begin{align*}
\|S\|(W'')+\|R\|(W'')&\le\liminf_k\|S_k\|(W'')+\liminf_k\|R_k\|(W'')\\
&\le\lim_k\|\Sigma_k\|(W'')=\|T\|(W'')\le\|S\|(W'')+\|R\|(W'').
\end{align*}
Hence $\|S\|+\|R\|=\|T\|$ in $W_*$. Now take $W''=B_1$. We have $\|T\|(\partial B_1)=0$: on $T_{\mathrm{reg}}$ we have $\Delta_T|x|^2=2m\ne0$, since $T$ is minimal, so the zero set of the smooth function $|x|^2-1$ on $T_{\mathrm{reg}}$ is $\mathcal H^m$-null, because at almost every point of the zero set of a smooth function its first and second derivatives vanish. Moreover, $\operatorname{Sing}T$ is $\mathcal H^m$-null. Therefore the inequalities above and lower semicontinuity give $\|S_k\|(B_1)\to\|S\|(B_1)$, $\|R_k\|(B_1)\to\|R\|(B_1)$ and $\|\Sigma_k\|(B_1)\to\|T\|(B_1)>0$, the last because $0\in\operatorname{spt}T$. The ratio $\theta_{D_k}(y_k,s_k)=\mu_k(D_k\cap B_1)/\mu_k(B_1)$ is invariant under the rescaling, so the balance condition and the comparison of areas give
$$
\|S_k\|(B_1)\ge(\eta-o(1))\|\Sigma_k\|(B_1),\qquad\|R_k\|(B_1)\ge(\eta-o(1))\|\Sigma_k\|(B_1).
$$
Letting $k\to\infty$, we obtain
$$
\mathbf M(S\llcorner B_1)\ge\eta\|T\|(B_1)>0,\qquad\mathbf M(R\llcorner B_1)\ge\eta\|T\|(B_1)>0.
$$
On the other hand, since $\Omega_\infty$ is a perimeter minimizer, $[T]\llcorner W_*=\partial[\Omega_\infty]\llcorner W_*$ is an oriented area-minimizing boundary in $W_*=B_{1/\beta}$ with $\partial([T]\llcorner W_*)=0$ in $W_*$, and $S,R\in\mathcal R^{\mathrm{loc}}_m(W_*)$ satisfy $S+R=[T]\llcorner W_*$ and $\|S\|+\|R\|=\|T\|$. These are the hypotheses of \cite[Theorem 2]{BG} with $r=1/\beta$ and $\beta r=1$, so $0=\mathbf M(\partial S)\ge\frac1{2\gamma}\min\{\mathbf M(S\llcorner B_1),\mathbf M(R\llcorner B_1)\}^{\frac{m-1}m}>0$, a contradiction.

\emph{Step 2: the general case.} We claim that there are $C$ and $r_1>0$ such that for $x\in K$, $0<r<r_1$, and all $D$ with $\theta_D(x,2r)\le\frac12$,
\begin{equation}\label{eq:small-sets}
\mu\left(D\cap B_r(x)\right)\le Cr\,P_\Sigma\left(D;B_{2r/\beta}(x)\right).
\end{equation}
Let $\eta=c_A/(2^{m+2}C_A)$ and $\eta'=\frac12C_1^{-2}$, and choose $r_1$ so small that $B_{2r_1/\beta}(x)\subset K_1$ for $x\in K$, $5r_1/\beta<s_A$, and $2r_1<\min\{s_\eta,s_{\eta'}\}$.

If $\theta_D(x,2r)\ge\eta$, then $\eta\le\theta_D(x,2r)\le\frac12\le1-\eta$, and \eqref{eq:balanced} with $s=2r$ gives $\mu(D\cap B_r(x))\le\mu(B_{2r}(x))\le C_A(2r)^m\le Cr\,P_\Sigma(D;B_{2r/\beta}(x))$.

If $\theta_D(x,2r)<\eta$, let $y\in B_r(x)$. Then $B_r(y)\subset B_{2r}(x)$ and $\mu(B_{2r}(x))\le C_A(2r)^m=\frac1{4\eta}c_Ar^m\le\frac{1}{4\eta}\mu(B_r(y))$, so $\theta_D(y,r)<\eta\,\mu(B_{2r}(x))/\mu(B_r(y))\le\frac14$. Let $y\in D\cap B_r(x)$ be a Lebesgue point of $D$. Since $\mu$ is locally doubling, $\mu$-almost every point of $D$ is a Lebesgue point \cite[Chapter 1]{Heinonen}, and at the chosen point we have $\theta_D(y,s)\to1$ as $s\to0$. Let $s_y=\sup\{s\in(0,r]:\theta_D(y,s)\ge\frac12\}$, choose $s'\in(s_y/2,s_y]$ with $\theta_D(y,s')\ge\frac12$, and let $t_y=\min\{2s_y,r\}$. If $t_y=2s_y$, then $t_y>s_y$, so $\theta_D(y,t_y)<\frac12$. If $t_y=r$, then $\theta_D(y,t_y)\le\frac14$. On the other hand, $t_y\ge s_y\ge s'$, so
$$
\mu\left(D\cap B_{t_y}(y)\right)\ge\frac12\mu(B_{s'}(y))\ge\frac12\mu(B_{s_y/2}(y))\ge\frac12C_1^{-2}\mu(B_{2s_y}(y))\ge\eta'\mu(B_{t_y}(y)).
$$
Thus $\eta'\le\theta_D(y,t_y)\le\frac12$, and \eqref{eq:balanced} gives $t_y^{m-1}\le C_{\eta'}P_\Sigma(D;B_{t_y/\beta}(y))$. By the Vitali covering lemma \cite[Chapter 1]{Heinonen}, applied to the balls $B_{t_y/\beta}(y)$, whose radii are at most $r/\beta$, there is a countable disjoint subfamily $\{B_{t_i/\beta}(y_i)\}$ such that $\mu$-almost every point of $D\cap B_r(x)$ lies in $\bigcup_iB_{5t_i/\beta}(y_i)$. These disjoint balls lie in $B_{(1+1/\beta)r}(x)\subset B_{2r/\beta}(x)$. Hence, using $t_i\le r$,
\begin{align*}
\mu(D\cap B_r(x))&\le\sum_i\mu\left(B_{5t_i/\beta}(y_i)\right)\le C_A\left(\frac5\beta\right)^m\sum_it_i^m\\
&\le Cr\sum_iP_\Sigma\left(D;B_{t_i/\beta}(y_i)\right)\le Cr\,P_\Sigma\left(D;B_{2r/\beta}(x)\right).
\end{align*}

\emph{Step 3: median and coarea formula.} Let $\varphi$ be Lipschitz on $B_{2r/\beta}(x)$, and let $\varphi_*$ be a median of $\varphi$ on $B_{2r}(x)$, so that $\theta_{\{\varphi>\varphi_*\}}(x,2r)\le\frac12$ and $\theta_{\{\varphi<\varphi_*\}}(x,2r)\le\frac12$. For $t>\varphi_*$, the set $D_t=\{\varphi>t\}\cap B_{2r/\beta}(x)$ satisfies $\theta_{D_t}(x,2r)\le\frac12$, so by \eqref{eq:small-sets}, $\mu(D_t\cap B_r)\le Cr\,P_\Sigma(D_t;B_{2r/\beta})$. Since $D_\Sigma$ is $\mu$-null and $\mathcal H^{m-1}$-null, the relative perimeter only has to be computed in the smooth manifold $\Sigma_{\mathrm{reg}}$, and by the coarea formula $\int_{\varphi_*}^\infty P_\Sigma(D_t;B_{2r/\beta})\,dt=\int_{B_{2r/\beta}\cap\{\varphi>\varphi_*\}}|\nabla_\Sigma\varphi|\,d\mu$. Hence
$$
\int_{B_r}(\varphi-\varphi_*)_+\,d\mu=\int_{\varphi_*}^\infty\mu(D_t\cap B_r)\,dt\le Cr\int_{B_{2r/\beta}\cap\{\varphi>\varphi_*\}}|\nabla_\Sigma\varphi|\,d\mu.
$$
The same argument applies to $(\varphi-\varphi_*)_-$, using the sets $\{\varphi<t\}$ with $t<\varphi_*$. Adding the two estimates, we obtain $\int_{B_r}|\varphi-\varphi_*|\,d\mu\le Cr\int_{B_{2r/\beta}}|\nabla_\Sigma\varphi|\,d\mu$. Finally $|\varphi_{B_r}-\varphi_*|\le\frac1{\mu(B_r)}\int_{B_r}|\varphi-\varphi_*|\,d\mu$, so
\begin{align*}
\frac{1}{\mu(B_r)}\int_{B_r}|\varphi-\varphi_{B_r}|\,d\mu&\le\frac{2}{\mu(B_r)}\int_{B_r}|\varphi-\varphi_*|\,d\mu\\
&\le2Cr\,\frac{\mu(B_{2r/\beta})}{\mu(B_r)}\cdot\frac{1}{\mu(B_{2r/\beta})}\int_{B_{2r/\beta}}|\nabla_\Sigma\varphi|\,d\mu.
\end{align*}
By Ahlfors regularity, $\mu(B_{2r/\beta})/\mu(B_r)\le C_A(2/\beta)^m/c_A$, and the Cauchy-Schwarz inequality gives the claim.
\end{proof}

To apply the framework of \cite[Section 3.1]{BZ2026}, we take $X=\Sigma$, $\mathsf d=d_g|_\Sigma$, which is complete by Proposition \ref{prop:bubble}(1), $\mathcal R=\Sigma_{\mathrm{reg}}$ with the metric $g|_{\Sigma_{\mathrm{reg}}}$, $\mathcal S=D_\Sigma$ and $\mu=\mathcal H^{d-1}\llcorner\Sigma$. The length of a curve in $\Sigma_{\mathrm{reg}}$ computed with $\mathsf d$ equals its Riemannian length with respect to $g|_{\Sigma_{\mathrm{reg}}}$. Hence, for every open set $U\subset\Sigma_{\mathrm{reg}}$, the length metric induced by $\mathsf d$ on $U$ is the Riemannian distance of $(U,g|_U)$, both computed with curves in $U$, although $\mathsf d$ itself is not the intrinsic distance of $\Sigma_{\mathrm{reg}}$. Ahlfors regularity was shown in Subsection \ref{app:local}, and the Poincar\'e inequality is given by Lemma \ref{lem:poincare}. Hence $\omega\equiv1$ is an admissible weight, that is, $(\Sigma,d_g|_\Sigma,\mathcal H^{d-1}\llcorner\Sigma)$ is a $(d-1)$ dimensional AM-PI space.

\subsection{Packing of the singular set}
\begin{lemma}\label{lem:packing}
\textit{Assume that $D_\Sigma\neq\varnothing$, so that $d\ge8$, and put $m=d-1$. For every compact set $K\subset\Sigma$ there are $C_K$ and $r_K>0$ such that for all $x\in K$ and $0<\varrho<R<r_K$, the set $D_\Sigma\cap B_R(x)$ can be covered by $C_K(R/\varrho)^{d-6}$ balls of radius $\varrho$ with respect to $d_g$. In particular, every $\varrho$-separated subset of $D_\Sigma\cap B_R(x)$ has at most $2^{d-6}C_K(R/\varrho)^{d-6}$ elements, $D_\Sigma$ satisfies the packing condition $\mathsf P^\Sigma_{m-5}(D_\Sigma)$, and $\tau_{\mathrm{pack}}(D_\Sigma)\ge m-2-(m-5)=3$.}
\end{lemma}
\begin{proof}
\emph{Step 1: the hypotheses of \cite{NV}.} Let $K'$ be a compact neighborhood of $K$ in $\overset{\circ}W$, and apply Subsection \ref{app:local} to a larger compact neighborhood $K''\subset\overset{\circ}W$ of $K'$. Near $K'$ we have $|\sec_g|\le\kappa$ and $\operatorname{inj}_g\ge\kappa^{-1}$. By Subsection \ref{app:local}(b), $\mathbf V=\mathbf v(\partial^\ast\Omega,1)$ is an integral varifold with mean curvature bounded by $\Lambda'_K$. By the Ahlfors upper bound there is $r_0>0$ such that $\|\mathbf V\|(B_\varrho(y))\le C\varrho^{d-1}$ for $y\in K'$ and $\varrho\le r_0$. We take $r_K\le\min\{\frac12,\frac1{4\kappa},\frac{r_0}4\}$ so small that the geodesic ball of radius $4r_K$ about every $x\in K$ is contained in $K'$. In particular, the singular points in $B_{2R}(x)$ used in Step 3 lie in $K'$, so that Step 2 applies to them. For $x\in K$ and $R<r_K$, we rescale the metric to $(2R)^{-2}g$. Then $|\sec|\le4\kappa R^2\le\kappa$, the injectivity radius is at least $(2\kappa R)^{-1}\ge2$, the mean curvature is at most $2\Lambda'_KR\le\Lambda'_K$, and the mass in the ball $B_2(x)$ is at most $C$. Hence for all such $x$ and $R$, the rescaled varifold satisfies the hypotheses of \cite[Theorem 1.3]{NV} with constants independent of $x$ and $R$: (i) the bounds $|\sec|\le K_0$ and $\operatorname{inj}\ge K_0^{-1}$ on $B_2(x)$ of \cite[(1.11)]{NV}, with $K_0=\max\{\kappa,\frac12\}$, (ii) the first variation bound $|\delta\mathbf V(Z)|\le\Lambda'_K\int|Z|\,d\|\mathbf V\|$ for smooth vector fields $Z$ with compact support in $B_2(x)$ of \cite[(1.12)]{NV}, and (iii) a mass bound on $B_2(x)$.

\emph{Step 2: $\epsilon$-regularity.} We claim that there are $\epsilon_0>0$ and $\bar r>0$ such that for $y\in D_\Sigma\cap K'$ and $0<\varrho\le\bar r$, the ball $B_\varrho(y)$ is not $(d-5,\epsilon_0)$-symmetric in the sense of \cite[Definition 1.1]{NV}. That is, after rescaling to unit size, $\mathbf V$ is not $\epsilon_0$-close to any cone which is invariant under translations along a $(d-5)$ dimensional plane. Suppose not. Then there are $y_k\in D_\Sigma\cap K'$ and $\varrho_k\to0$ such that $B_{\varrho_k}(y_k)$ is $(d-5,1/k)$-symmetric. Blow up at $y_k$ by $\varrho_k^{-1}$ as in Subsection \ref{app:local}(c): a subsequence converges to an area-minimizing boundary in $\mathbb R^d$, and the varifolds converge to the multiplicity one limit $\mathbf V_\infty$, while the rescaled mean curvature is at most $\Lambda'_K\varrho_k\to0$. By \cite[Definition 1.1]{NV}, $(k',\epsilon)$-symmetry means that, after $\mathbf V$ is viewed as a varifold in $T_{y_k}P$ via the exponential map and rescaled, its weak distance in $B_1$ to some $k'$-symmetric integral varifold, namely a cone invariant under translations along a $k'$ dimensional plane, is less than $\epsilon$. The radius is required to be less than the injectivity radius. The comparison cones have uniformly bounded mass in every ball $B_\rho$, since cones have mass growth $\rho^{d-1}$ and are close to $\mathbf V$ in $B_1$. Hence a subsequence converges weakly to a cone $I$ which is invariant under translations along a $(d-5)$ dimensional plane, and $I\llcorner B_1=\mathbf V_\infty\llcorner B_1$, since cones carry no mass on spheres. Thus $\partial\Omega_\infty\cap B_1$ is part of a cone, and so is $\Omega_\infty\cap B_1$. The tangent cone of $\Omega_\infty$ at $0$, which is area-minimizing, is therefore this cone, and it has the form $\mathbb R^{d-5}\times\mathbf C_0$, where $\mathbf C_0\subset\mathbb R^5$ is an area-minimizing hypercone. Indeed, if $\mathbb R^j\times\mathbf C_0$ is area-minimizing then so is $\mathbf C_0$, a standard fact used in Federer's dimension reduction (see \cite{Simon}). By Simons' theorem and dimension reduction, $\mathbf C_0$ is a hyperplane, and $\mathbf V_\infty$ is a multiplicity one hyperplane in $B_1$. Then the mass ratio of the rescaled varifolds in $B_1$ tends to one and their mean curvature tends to zero, so by Allard's regularity theorem \cite{Allard}, $y_k$ is a regular point for large $k$, a contradiction. We now shrink $r_K$ so that $2r_K\le\bar r$.

\emph{Step 3: covering numbers.} By Step 2, in the rescaled unit ball the singular set satisfies $D_\Sigma\subset\mathcal S^{d-6}_{\epsilon_0}\subset\mathcal S^{d-6}_{\epsilon_0,r}$ for all $r\in(0,1]$, in the notation of \cite{NV}. By \cite[Theorem 1.3]{NV}, or directly by \cite[Theorem 1.4]{NV}, $\operatorname{Vol}(B_r(D_\Sigma)\cap B_1)\le Cr^{d-(d-6)}=Cr^6$ for $0<r\le1$, where $C$ depends only on $d$, $\kappa$, $\Lambda'_K$, the mass bound and $\epsilon_0$. Take a maximal $r$-separated set in $D_\Sigma\cap B_{1/2}$. The balls of radius $r/2$ centered at its points are disjoint, are contained in $B_r(D_\Sigma)\cap B_1$, and each has volume at least $cr^d$. Hence the set has at most $Cr^{-(d-6)}$ points, and the balls of radius $r$ centered at these points cover $D_\Sigma\cap B_{1/2}$. Since Step 1 was a rescaling of $B_{2R}(x)$, scaling back shows that $D_\Sigma\cap B_R(x)$ can be covered by $C(R/\varrho)^{d-6}$ balls of radius $\varrho$. Since a ball of radius $\varrho/2$ contains at most one point of a $\varrho$-separated set, the remaining assertions follow.
\end{proof}

\begin{remark}
Replacing $d-5$ by $d-7$ in Step 2, where Simons' theorem still applies to $\mathbf C_0\subset\mathbb R^7$, gives $D_\Sigma\subset\mathcal S^{d-8}_{\epsilon_0}$ and hence the sharp packing condition $\mathsf P^\Sigma_{m-7}(D_\Sigma)$ stated in \cite[Proposition 4.6]{BZ2026}. The weaker bound above suffices here. For general integral varifolds with bounded mean curvature, \cite[Theorem 1.3]{NV} gives volume estimates for the quantitative strata $\mathcal S^k_{\epsilon,r}$ with fixed parameters. Placing all singular points in a single quantitative stratum is the additional input of Step 2, which uses that the $\mu$-bubble is an almost minimizing boundary and not only a varifold with bounded mean curvature.
\end{remark}

Lemma \ref{lem:poincare}, the subsequent discussion and Lemma \ref{lem:packing} prove Proposition \ref{prop:bubble}(4). This completes the proof of Proposition \ref{prop:bubble}.

\bibliographystyle{amsplain}

\begin{thebibliography}{10}

\bibitem{Allard}
William~K. Allard, \emph{On the first variation of a varifold}, Ann. of Math.
  (2) \textbf{95} (1972), 417--491.

\bibitem{Allegretto}
Walter Allegretto, \emph{On the equivalence of two types of oscillation for
  elliptic operators}, Pacific J. Math. \textbf{55} (1974), 319--328.

\bibitem{Bese}
Arthur~L. Besse, \emph{Einstein manifolds}, Ergebnisse der Mathematik und ihrer
  Grenzgebiete (3) [Results in Mathematics and Related Areas (3)], vol.~10,
  Springer-Verlag, Berlin, 1987. \MR{867684}

\bibitem{BZ2026}
Yuchen Bi and Jintian Zhu, \emph{Positive scalar curvature obstructions via
  singular dimension descent}, arXiv:2606.20528v3 (2026).

\bibitem{BG}
Enrico Bombieri and Enrico Giusti, \emph{Harnack's inequality for elliptic
  differential equations on minimal surfaces}, Invent. Math. \textbf{15}
  (1972), 24--46.

\bibitem{BW2026}
Simon Brendle and Yipeng Wang, \emph{A dimension descent scheme for the
  positive mass theorem in arbitrary dimension}, arXiv:2604.08473v2 (2026).

\bibitem{C-W}
Xiaoxiang Chai, Juncheol Pyo, and Xueyuan Wan, \emph{Spectral constant rigidity
  of warped product metrics}, J. Lond. Math. Soc. (2) \textbf{110} (2024),
  no.~1, Paper No. e12958, 35. \MR{4767708}

\bibitem{C-G}
Jeff Cheeger and Detlef Gromoll, \emph{The splitting theorem for manifolds of
  nonnegative {R}icci curvature}, J. Differential Geometry \textbf{6}
  (1971/72), 119--128. \MR{303460}

\bibitem{Cho-Li}
Otis Chodosh and Chao Li, \emph{Generalized soap bubbles and the topology of
  manifolds with positive scalar curvature}, Ann. of Math. (2) \textbf{199}
  (2024), no.~2, 707--740. \MR{4713021}

\bibitem{chow2025}
Tsz-Kiu~Aaron Chow, \emph{Scalar curvature rigidity for products of spheres and
  tori}, arXiv:2511.04407 (2025).

\bibitem{Deng}
Jialong Deng, \emph{Curvature-dimension condition meets {G}romov's
  {$n$}-volumic scalar curvature}, SIGMA Symmetry Integrability Geom. Methods
  Appl. \textbf{17} (2021), Paper No. 013, 20. \MR{4210894}

\bibitem{docarmo}
Manfredo~Perdig\~ao do~Carmo, \emph{Riemannian geometry}, Mathematics: Theory
  \& Applications, Birkh\"auser Boston, Inc., Boston, MA, 1992.

\bibitem{Federer}
Herbert Federer, \emph{Geometric measure theory}, Die Grundlehren der
  mathematischen Wissenschaften, Band 153, Springer-Verlag New York, Inc., New
  York, 1969.

\bibitem{FCS}
Doris Fischer-Colbrie and Richard Schoen, \emph{The structure of complete
  stable minimal surfaces in {$3$}-manifolds of nonnegative scalar curvature},
  Comm. Pure Appl. Math. \textbf{33} (1980), no.~2, 199--211.

\bibitem{GW}
Robert~E. Greene and Hung-Hsi Wu, \emph{{$C^{\infty }$} approximations of convex,
  subharmonic, and plurisubharmonic functions}, Ann. Sci. \'{E}cole Norm. Sup.
  (4) \textbf{12} (1979), no.~1, 47--84.

\bibitem{GL1983}
Mikhael Gromov and H.~Blaine Lawson, Jr., \emph{Positive scalar curvature and
  the {Dirac} operator on complete {Riemannian} manifolds}, Inst. Hautes
  \'Etudes Sci. Publ. Math. (1983), no.~58, 83--196 (1984). \MR{720933}

\bibitem{HSS}
Tianze Hao, Yuguang Shi, and Yukai Sun, \emph{Llarull type theorems on complete
  manifolds with positive scalar curvature}, Trans. Amer. Math. Soc.
  \textbf{377} (2024), no.~10, 7403--7420. \MR{4855340}

\bibitem{Heinonen}
Juha Heinonen, \emph{Lectures on analysis on metric spaces}, Universitext,
  Springer-Verlag, New York, 2001.

\bibitem{Kazdan}
Jerry~L. Kazdan, \emph{Deformation to positive scalar curvature on complete
  manifolds}, Math. Ann. \textbf{261} (1982), no.~2, 227--234. \MR{675736}

\bibitem{LiChao}
Chao Li and Christos Mantoulidis, \emph{Positive scalar curvature with skeleton
  singularities}, Math. Ann. \textbf{374} (2019), no.~1-2, 99--131.
  \MR{3961306}

\bibitem{LSWZ2024}
Yihan Li, Guangxiang Su, Xiangsheng Wang, and Weiping Zhang, \emph{Llarull's
  theorem on odd dimensional manifolds: the noncompact case}, J. Differential
  Geom., to appear, arXiv:2404.18153.

\bibitem{Llarull}
Marcelo Llarull, \emph{Sharp estimates and the {D}irac operator}, Math. Ann.
  \textbf{310} (1998), no.~1, 55--71. \MR{1600027}

\bibitem{Maggi}
Francesco Maggi, \emph{Sets of finite perimeter and geometric variational
  problems}, Cambridge Studies in Advanced Mathematics, vol. 135, Cambridge
  University Press, Cambridge, 2012.

\bibitem{NV}
Aaron Naber and Daniele Valtorta, \emph{The singular structure and regularity
  of stationary varifolds}, J. Eur. Math. Soc. (JEMS) \textbf{22} (2020),
  no.~10, 3305--3382.

\bibitem{Piepenbrink}
John Piepenbrink, \emph{Nonoscillatory elliptic equations}, J. Differential
  Equations \textbf{15} (1974), 541--550.

\bibitem{HyperMR1299730}
John~G. Ratcliffe, \emph{Foundations of hyperbolic manifolds}, Graduate Texts
  in Mathematics, vol. 149, Springer-Verlag, New York, 1994. \MR{1299730}

\bibitem{RT2026}
Oskar Riedler and Thomas Tony, \emph{Scalar-rigid submersions are Riemannian products}, arXiv:2601.14117 (2026).

\bibitem{Simon}
Leon Simon, \emph{Lectures on geometric measure theory}, Proceedings of the
  Centre for Mathematical Analysis, Australian National University, vol.~3,
  Australian National University, Centre for Mathematical Analysis, Canberra,
  1983.

\bibitem{Simons}
James Simons, \emph{Minimal varieties in {R}iemannian manifolds}, Ann. of Math.
  (2) \textbf{88} (1968), 62--105.

\bibitem{Tamanini}
Italo Tamanini, \emph{Regularity results for almost minimal oriented
  hypersurfaces in {$\mathbb R^N$}}, Quaderni del Dipartimento di Matematica
  dell'Universit\`a di Lecce \textbf{1} (1984), 1--92.

\bibitem{WWXZ2026}
Gaoming Wang, Jinmin Wang, Zhizhang Xie, and Bo~Zhu, \emph{Scalar-mean rigidity
  theorem and Llarull's theorem for nonspin manifolds}, arXiv:2609.15907v1
  (2026).

\bibitem{WWX2026}
Jian Wang, Jinmin Wang, and Zhizhang Xie, \emph{{$L^\infty$}-metrics on tori and
  {S}choen's conjecture}, arXiv:2606.21325 (2026).

\bibitem{WX2024}
Jinmin Wang and Zhizhang Xie, \emph{Scalar curvature rigidity of spheres with
  subsets removed and {$L^\infty$} metrics}, arXiv:2407.21312 (2024).

\bibitem{Zhang2020}
Weiping Zhang, \emph{Nonnegative scalar curvature and area decreasing maps},
  SIGMA \textbf{16} (2020), Paper No. 033, 7 pp.

\bibitem{Zhou-Zhu}
Linfeng Zhou and Guangrui Zhu, \emph{A weighted {L}larull type theorem and its
  applications}, Proc. Amer. Math. Soc. \textbf{154} (2026), no.~10,
  4453--4467. \MR{5105323}

\bibitem{Zhu1}
Jintian Zhu, \emph{Rigidity results for complete manifolds with nonnegative
  scalar curvature}, J. Differential Geom. \textbf{125} (2023), no.~3,
  623--644. \MR{4674077}

\end{thebibliography}
\providecommand{\bysame}{\leavevmode\hbox to3em{\hrulefill}\thinspace}
\providecommand{\MR}{\relax\ifhmode\unskip\space\fi MR }
\providecommand{\MRhref}[2]{%
  \href{http://www.ams.org/mathscinet-getitem?mr=#1}{#2}
}
\providecommand{\href}[2]{#2}

\end{document}